\documentclass[11pt, dvipsnames]{amsart}
\usepackage{amsmath}
\usepackage{amsfonts}
\usepackage{amssymb}
\usepackage{amsthm}
\usepackage{mathrsfs}       
\usepackage{graphicx}     
\usepackage{todonotes}
\usepackage{latexsym}

\usepackage[T1]{fontenc}
\usepackage[utf8]{inputenc}
\allowdisplaybreaks
\usepackage[left=2.5cm,top=3cm,right=2.5cm,bottom=3cm,bindingoffset=0.5cm]{geometry}

\usepackage{xspace}
\usepackage[colorlinks=true]{hyperref}

\newcommand{\R}{\mathbb{R}}
\newcommand{\N}{\mathbb{N}}
\newcommand{\T}{\mathbb{T}}
\newcommand{\Z}{\mathbb{Z}}

\newcommand{\C}{\mathbb{C}}

\newcommand{\LL}{\mathscr{L}}

\newcommand{\Hdiv}{\mathcal{H}}

\newcommand{\ie}{\textit{i.e., }}
\numberwithin{equation}{section}

\renewcommand{\d}{\partial}

\newcommand{\nocontentsline}[3]{}
\newcommand{\tocless}[2]{\bgroup\let\addcontentsline=\nocontentsline#1{#2}\egroup}

\newcommand{\e}{\varepsilon}

\newtheorem{thm}{THEOREM}[section]
\newtheorem{remark}[thm]{REMARK}
\newtheorem{lem}[thm]{LEMMA}
\newtheorem{defn}[thm]{DEFINITION}
\newtheorem{prop}[thm]{PROPOSITION}
\newtheorem{cor}[thm]{COROLLARY}

\newtheorem{hyp}{Assumption}
\newcounter{thmbiss}

\def\<{\langle}
\def\>{\rangle}

\title[Rapid stabilization for the linearized Water Waves system]{Fredholm backstepping for critical operators and application to rapid  stabilization for the linearized water waves}

\author[L.~Gagnon]{Ludovick Gagnon} 
\thanks{Université de Lorraine, CNRS, Inria équipe SPHINX,  F-54000 Nancy, France. E-mail: \texttt{ludovick.gagnon@inria.fr.}}

\author[A.~Hayat]{Amaury Hayat}
\thanks{CERMICS, \'{E}cole des Ponts ParisTech, 6 - 8, Avenue Blaise Pascal, Cité Descartes—Champs sur Marne, 77455 Marne la Vall\'{e}e, France. E-mail: \texttt{amaury.hayat@enpc.fr.}}

\author[S.~Xiang]{Shengquan Xiang}
\thanks{School of Mathematical Sciences, Peking University, 100871, Beijing, P. R. China.  E-mail: \texttt{shengquan.xiang@math.pku.edu.cn.}}

\author[C.~Zhang]{Christophe Zhang}
\thanks{Université de Lorraine, CNRS, Inria équipe SPHINX,  F-54000 Nancy, France. E-mail: \texttt{Christophe.zhang@inria.fr.}}

\begin{document}
\maketitle

\begin{abstract}
Fredholm-type backstepping transformation,
introduced by Coron and Lü, {\color{black} intensively developed over the last decade and}
has become a powerful tool for rapid stabilization. Its strength lies in its systematic approach, allowing to deduce rapid stabilization from approximate controllability. But limitations {\color{black} to prove the existence of a Fredholm backstepping transformation} exist with the current approach for operators of the form $|D_x|^\alpha$ for $\alpha \in (1,3/2]$. We present here a new compactness/duality method which hinges on Fredholm's alternative to overcome the $\alpha=3/2$ threshold. More precisely, the compactness/duality method allows to prove the existence of a Riesz basis for the backstepping transformation for skew-adjoint operator\textcolor{black}{s satisfying} $\alpha>1$, a key step in the construction of the Fredholm backstepping transformation, where the usual methods only work for $\alpha>3/2$. The illustration of this new method is shown on the rapid stabilization of the linearized capillary-gravity water waves equation exhibiting an operator of critical order $\alpha=3/2$. \\

\end{abstract}
\setcounter{tocdepth}{2}
\tableofcontents

\section{Introduction}
 
Since its introduction by Coron and L\"u for the rapid stabilization of the Korteweg-de Vries equation \cite{CoronLu14} and Kuramoto-Sivashinsky equation \cite{CoronLu15}, the Fredholm{\color{black}-type backstepping} transformation has been applied successfully in the past decade for the rapid stabilization of a large class of equations. It consists in finding an {\color{black}isomorphism-feedback} pair $(T,K)$ {\color{black} where $K$ is the feedback operator and $T$} maps a system of the form,
    \begin{equation}
    \label{eq:sys0}
    \begin{split}
        &\partial_{t}u= Au + BKu,
    \end{split}
    \end{equation}
    to the rapidly exponentially stable system,
    \begin{equation}
    \begin{split}
        &\partial_{t}v =(A-\lambda I)v,
    \end{split}
    \end{equation}
where $A$ is the generator of a strongly continuous semigroup, $B$ is an unbounded operator and $\lambda$ is an arbitrarily large positive number.
Compared with the original Volterra transformation introduced by Krsti{\'c} and Balogh \cite{BK1}, the Fredholm transformation possesses the advantage of presenting a systematic approach to prove the rapid stabilization based on spectral properties of the spatial operator $A$ and from suitable controllability assumptions.
However, the classical approach for the Fredholm transformation fails to deal with operators {\color{black} $A$} \textcolor{black}{that have eigenvalues scaling as $n^{\alpha}$ with $\alpha\in(1,3/2]$, for instance, operators} 
of the form $|D_x|^\alpha$ for $\alpha \in (1,3/2]$. Indeed, one key step in proving the existence of the Fredholm transformation $T$ is to prove that the family $(T\varphi_n)_n$ is a Riesz basis of the state space, where, 
\begin{gather*}
    \{(\varphi_n, \lambda_n)\}_{n\in \N} \textrm{ are the eigenmodes of the spatial operator } A.
\end{gather*}
 The classical approach refers to the way of tackling the existence of $T$ by proving that the family $(T\varphi_n)_n$ is quadratically close to the eigenfunction basis $(\varphi_n)_n$ (see Definition \ref{def-riesz-bas} (3) for a precise statement). However, when $\alpha\in (1, 3/2]$, the growth of the eigenvalues is too slow, preventing the use of such criteria. Thus, for operators behaving as $|D_x|^\alpha$ with $\alpha \in (1,3/2]$, the Fredholm alternative remains an open question.

\subsection{The compactness/duality method} \vphantom{water-waves}

In this paper we present a new method to answer this question. This method is based on a new compactness/duality approach to prove that the family $(T\varphi_n)_n$ is indeed a Riesz basis in sharp spaces for the whole range $\alpha>1$ {\color{black} of skew-adjoint operators.} The challenging part in proving that the family $(T\varphi_n)_n$ is a Riesz basis is the coercivity estimate (see the left-hand side estimate of \eqref{ineq-riesz}). We proceed with a contradiction argument to prove this inequality. Using the expression of $T\varphi_n$, we are able to prove that $T$ can be decomposed in an invertible part and a compact part. Then, the desired uniform estimate can be deduced from the $\omega$-independence property. A further inspection of the duality between the $\omega$-independence of $(T\varphi_n)_n$ in $H^{r}$ and the density of $(T^*\overline{\varphi_n})_n$ in $(H^{r})^{*}\simeq H^{-r}$ finally leads to the required property\footnote{\textcolor{black}{In fact, this holds in more general spaces $\Hdiv^{r}$ and $\Hdiv^{-r}$ that correspond exactly to $H^{r}$ and $H^{-r}$ when the operator $A$ acts as a Fourier multiplier (see Section \ref{sec:notations}).}}. 

A second important step of our method is to deal with the so-called $TB=B$ uniqueness condition, introduced in \cite{Coron_ICIAM15} for finite-dimensional systems, used implicitly in \cite{CoronLu14, CoronLu15}. It was first introduced explicitly for PDEs in \cite{CGM} to handle the nonlocal term arising from the distributed controls. In the original approach, proposed by \cite{CoronLu14}, and used since then {\color{black} for the Fredholm alternative}, the {\color{black}uniqueness} condition is solved thanks to the quadratically close property mentioned above.
With this new method, we are able to sidestep this limitation thanks to a fine decomposition of the $TB=B$ condition, which allows us to define the transformation $T$ along with the feedback law $K$. 

Beyond the $\alpha= 3/2$ threshold, this new method leads to sharp Riesz basis properties for a large class of skew-adjoint operators, including Fourier multiplier operators, as long as the high frequencies scale as $n^{\alpha}$ for $\alpha >1$. We apply this new strategy to prove the rapid stabilization of the linearized capillary-gravity water waves equation, which exhibits a spatial operator behaving like $|D_x|^{3/2}$. This is an example corresponding to the critical case $\alpha=3/2$ which remained out of reach until now.


\label{sec:introduction}

\subsection{Statement of the main results}

We consider in this paper the rapid stabilization of the system \eqref{eq:sys0}, that is to seek, for any $\lambda>0$, 
a control feedback law $w(t)= K u(t,\cdot)$, such that the solution of \eqref{eq:sys0} satisfies\footnote{{\color{black}Throughout the article, $a \lesssim b$ denotes that there exists $C>0$ such that $a\leq Cb$ and $a\sim b$ means that there exist $c,C>0$ such that $c a \leq b \leq Ca$. When the quantities involved depends on $n\in \N$, the constants $c,C>0$ are uniform with respect to $n$.}},
\[
\|u(t)\| \lesssim  e^{-\lambda t} \|u_0\|, \quad \forall t\in (0, +\infty). 
\]
We stress that the feedback law $K$ is of finite rank, that is, we aim at stabilizing system \eqref{eq:sys0} with a finite number of scalar feedback controls. {\color{black}More precisely, we do it with the minimal number of scalar controls for which system \eqref{eq:sys0} can be controllable, under suitable assumptions on the control operator $B$ (see Section \ref{sec:mainresults}  for more details).}

 
{\color{black}The precise statement of our main results relies on the spectral properties of the operator $A$ and on the {\color{black}regularity} properties of the control operator $B$ such that it is admissible (see Assumption \ref{asump1}), and \eqref{eq:sys0} is controllable. In order to introduce the techniques used in the article and to motivate the new compactness/duality method, we begin by giving a simplified statement of our main results \textcolor{black}{when the operator only has simple eigenvalues} (the precise statements for the general case are found in Section \ref{sec:mainresults}, \textcolor{black}{see Theorems \ref{thalpha} and \ref{thalpha-2}})}
\begin{thm}
\label{thalphaintro}
\textcolor{black}{Let $A$ {\color{black} be} a skew-adjoint operator with simple eigenvalues satisfying 
\eqref{hyp:1}}
below, and \textcolor{black}{in particular such that its eigenvalues $\lambda_{n}$ satisfy $\textcolor{black}{|\lambda_{n}|}\sim n^{\alpha}$ with $\alpha >1$.} Let $B\in {\Hdiv}^{-{\alpha/2}}$ satisfying the controllability Assumption \ref{asump1} below. Then, for any $\lambda>0$, there exist  a bounded linear operator $K\in \mathcal{L}(\Hdiv^{\textcolor{black}{\alpha/2}};\mathbb{C})$ and an operator $T$ such that $T$ is an isomorphism from $\Hdiv^{r}$ to itself for any $r\in(1/2-\alpha,\alpha-1/2)$ and maps the system, 
\begin{equation}\label{introgeneoperaclosed}
    \partial_{t}u=A u+BK(u),\; 
\end{equation}
to the system
\begin{equation*}
 \partial_{t}v=A v-\lambda v.\;  
\end{equation*}
Consequently, the closed-loop system \eqref{introgeneoperaclosed} is exponentially stable in $\Hdiv^{r}$ for $r\in (1/2- \alpha, \alpha- 1/2)$ with decay rate $\lambda$.
\end{thm}
{\color{black}
\begin{remark}
Here, the spaces $\Hdiv^{s}$, defined properly in Section \ref{sec:notations}, are defined from the eigenbasis of the operator $A$. They are such that $\Hdiv^0\textcolor{black}{=H}$, \textcolor{black}{where H is a given Hilbert space,} and $D(A)=\Hdiv^{\alpha}$. \textcolor{black}{In particular, when} the spatial operator is defined as $A=ih(|D_{x}|)$ over a compact one-dimensional domain with \textcolor{black}{periodic} boundary conditions, with $h$ an analytic function, \textcolor{black}{and  $H$ is chosen as $L^{2}(\mathbb{T};\mathbb{C})$,} the spaces $\Hdiv^{s}$ \textcolor{black}{correspond to the} classical Sobolev spaces \textcolor{black}{$H^{s}(\mathbb{T};\mathbb{C})$, denoted $H^{s}(\mathbb{T})$ in the following}. Thus,  the spaces $\Hdiv^{s}$ are a natural generalization of the classical Sobolev spaces, and the statement of our main results holds in this  general framework. In the present article, we exhibit the spaces $\Hdiv^{s}$ with the example of the linearized water waves equation on the 1-dimensional torus $\mathbb{T}$ and highlight the \textcolor{black}{correspondance} $\Hdiv^s \textcolor{black}{= H^s(\T)}$. 
\end{remark}}  
 

The proof of Theorem \ref{thalphaintro} relies on our compactness/duality method, allowing to prove the existence of a Fredholm operator for the backstepping method \textcolor{black}{as long as $\alpha >1$ (and in particular below the critical case $\alpha=3/2$)}. 

\subsection{The capillary-gravity water waves equation} 
\vphantom{Water-waves}

{\color{black} One application of our main results is the rapid stabilization of the linearized capillary-gravity water waves equation, which exhibits a skew-adjoint operator of order $\alpha=3/2$. They are relevant for modelling the motion and stability of perfect fluids where the surface tension and capillarity cannot be neglected; for instance for small characteristic scales or when waves are breaking at certain wave frequencies (\cite{zbMATH06168816,Zakharov}). The linearized water waves equation writes (the full derivation from the free surface Euler equation is done in Section  \ref{sec:water-waves}),
\begin{equation}\label{eq:WWintro}
\d_t u(t) = \LL u(t) + Bw(t),
\end{equation}
where 
\begin{equation}
\label{op-WW}
\mathscr{L}:= -i \left((g- \d_x^2)G[0,h]\right)^{1/2},
\end{equation}
and where $G[0,h]=|D_x|\textrm{tanh}(h|D_x|)$ is defined as a Fourier multiplier on periodic functions. In \eqref{eq:WWintro}, $u$ is complex-valued, and \textcolor{black}{the operator $B=(B_1,B_2)\in \Hdiv^{-3/4} $ (which can be identified with $H^{-3/4}(\T)$) is real-valued. In this setting, Assumption \ref{asump1}, given in Section \ref{sec:mainresults}, will give a sufficient condition for the controllability of \eqref{eq:WWintro} in $\Hdiv^0 \textcolor{black}{= L^{2}(\mathbb{T})}$.
}

We prove, based on the main result of this paper, the rapid stabilization of the linearized water waves equation \eqref{eq:WWintro} (see Corollaries \ref{th1} and \ref{coro:expostab} for a more precise statement):

\begin{thm}\label{thm:intromain1}[\textcolor{black}{Water waves equation}]
Let $B\in (H^{-3/4}{(\mathbb{T})})^2$ satisfying the controllability {\color{black}and admissibility} Assumption \ref{asump1} \textcolor{black}{below}. Then, for any $\lambda>0$, there exists a bounded linear operator $K\in \mathcal{L}(H^{3/4}{(\mathbb{T})};\mathbb{C}^{2})$ and an operator $T$ being an isomorphism from $H^{r}(\mathbb{T})$ to itself for any $r\in(-1,1)$ and {mapping} the closed-loop system 
\begin{equation}\label{introclosedloopsystem}
\begin{array}{ll}
\partial_t u =\mathscr{L}u +BK(u), & (t,x) \in \mathbb{R}_{+} \times \T,
\end{array} 
\end{equation}
to the system
\begin{equation*}
  \begin{array}{ll}
 \partial_t v =\mathscr{L}v -\lambda v , & (t,x) \in \mathbb{R}_{+} \times \T, 
 \end{array} 
\end{equation*}
where $\mathscr{L}$ is the linearized  water-wave operator given in \eqref{op-WW}. Consequently, the closed-loop system \eqref{introclosedloopsystem} is exponentially stable in $H^r{(\mathbb{T})}$ for any $r\in (-1, 1)$.
\end{thm}

\subsection{The Fredholm-type backstepping method}
\label{sec:related-method}

Let us introduce the heuristics of the backstepping method for PDEs, as well as the main steps of our proof. The backstepping method consists in proving the existence of a feedback law $w= Ku$ and an invertible operator $T$ that maps the solution $u$ of the equation to be stabilized,
\begin{equation}\label{eq:astabiliser}
\begin{cases}
{\partial_{t}}u(t)=Au(t)+Bw(t),\\
u(0)=u_0,
\end{cases}
\end{equation}
onto the solution $v$ of an exponentially stable equation (due to the natural dissipation of $A$ and the strong damping effect of $-\lambda I$),
\begin{equation}\label{eq:ciblestable}
\begin{cases}
\partial_{t}v(t)=(A-\lambda I)v(t),\\
v(0)=v_0,
\end{cases}
\end{equation}
where $u(t)$ and $v(t)= T u(t)$ belong to a Hilbert  space $H$, $A$ is the generator of a $C_{0}$ semigroup over the state space, $B$ is an unbounded operator satisfying some admissibility conditions (see, for instance, \cite{TucsnakWeissBook} for a definition) and the control $w(t)=Ku(t)$ is of feedback form to achieve stabilization. 
Notice that if such $T$ exists and is invertible, then the exponential decay is proved by choosing $v_0=Tu_0$,
\[
\|u(t)\| \leq \|T^{-1}\|\|v(t)\| \lesssim e^{(\omega_0-\lambda)t}\|T^{-1}\|\|T\|\|u_0\|,
\]
with 
$$
\omega_0=\omega_0(e^{tA}):=\inf_{t\in (0,\infty)} \log \Vert e^{tA}\Vert_{\mathcal{L}(H,H)}.
$$

The main challenge is to find an isomorphism-{feedback} pair $(T,K)$ such that this mapping can be achieved.
This problem is equivalent to find $(T,K)$ which solve {formally}, \begin{equation}\label{eq:TsansTB}
T(A+BK)=(A-\lambda I)T,
\end{equation}
obtained by taking formally the time derivative of $v=Tu$ and using \eqref{eq:astabiliser} and \eqref{eq:ciblestable}.  

The backstepping method was first introduced in finite dimension as a chain of integrators for feedback laws \cite{KrsticSmyshlyaev_Book}, but was then extended to PDEs by Krsti{\'c} and Balogh \cite{BK1} for Volterra transformations of the second kind,
\begin{equation*}\label{eq:volterra}
v(t,x)=Tu=u(t,x) - \int_0^x k(x,y)u(t,y)dy.
\end{equation*}
The abstract equation \eqref{eq:TsansTB} is then transformed to a non-standard PDE on the kernel $k$ {and most of the properties of the transformation $T$ are deduced from the properties of the kernel $k$}. In addition to leaving the classical framework of Cauchy theory, the kernels resulting from the Volterra transformation present boundary conditions on the diagonal $k(x,x)$ that proves to be very difficult to handle. Despite these difficulties, several methods have been developed to solve the PDE on the kernel of the Volterra transformation (successive approximations \cite{KrsticSmyshlyaev_Book},  explicit representations \cite{KrsticSmyshlyaev_Book} or method of characteristics \cite{MR4160602}) leading to a rich literature, the invertibility of the Volterra transformation being guaranteed. 

More recently, the Fredholm-type transformation,
\begin{equation*}
v(t,\cdot)=Tu(t,\cdot)=(Id+\mathcal{T}_{comp})u(t,\cdot),
\end{equation*}
was introduced by Coron and L\"u for the rapid stabilization of the Korteweg-de Vries equation \cite{CoronLu14} and Kuramoto-Sivashinsky equation \cite{CoronLu15} by means of the backstepping method. In this transformation, $T$ is a Fredholm operator with an invertible and a compact part. Although much more technical, the Fredholm transformation provides a systematic approach to the backstepping method based on the spectral properties of the operator $A$ and the controllability properties to prove the rapid stabilization for a large class of equations. 

Let us elaborate on the techniques involved with the backstepping method with the Fredholm transformation. A first crucial step is to consider the so-called uniqueness condition $TB=B$ to change the abstract equation \eqref{eq:TsansTB} into a system of two equations,
\begin{align}
TA+BK&=(A-\lambda I)T, \label{eq:introeqT} \\
TB&=B. \label{eq:introTB=B}
\end{align}
The uniqueness condition $TB=B$ was first introduced in \cite{Coron_ICIAM15} to prove the existence of $(T,K)$ solving \eqref{eq:introeqT}-\eqref{eq:introTB=B} in finite dimension, used implicitly in \cite{CoronLu14,CoronLu15} and finally stated explicitly in \cite{CGM}  to remove nonlocal terms involved in \eqref{eq:TsansTB} for distributed controls. 
The proof of existence of an invertible transformation $T$ over the state space and a feedback law stabilizing \eqref{eq:astabiliser} is then divided in the following steps: 
\begin{enumerate}
    \item[Step 1:]  Let {\color{black}$\lambda>0$ such that} $(\lambda_n+\lambda)_n \cap \sigma(A)= \emptyset$.  Notice that \eqref{eq:introeqT} is equivalent to,
\[
T\varphi_n = K_n (A-(\lambda_n+\lambda)I)^{-1}B,
\]
where $K_n=K \varphi_n$. We prove that,  
    \[
\tilde{T}\varphi_n = (A-(\lambda_n+\lambda)I)^{-1}B,
\]
is a Riesz basis family of $H$. 
\item[Step 2:] Let $B=\sum_n b_n \varphi_n$. Use the Riesz basis properties to solve the $TB=B$ uniqueness condition in a suitable sense,
\begin{align*} 
TB  = \sum_n b_n T \varphi_n 
  = \sum_n b_n K_n \tilde{T}\varphi_n 
  =B.
\end{align*}
\item[Step 3:] Show from $TB=B$ that $(K_n)_n$ is uniformly bounded from above and below. Then, using operator equality, 
 prove that $T : H \rightarrow H $ is continuous and invertible.
\item[Step 4:] Thanks to the operator equality,  the operator $A+BK$ generates a semigroup on $H$ (and{, in fact,} it also generates semigroups in $H^{\beta}$ with a certain range of $\beta$). Conclude on the rapid stabilization using the operator equality.
\end{enumerate} 
Aside of the seemingly different approach of 
hyperbolic systems \cite{coron:hal-03161523, CoronHuOlive16, MR4160602, Zhang-finite, ZhangRapidStab}, the proof of Step 1 and 2 relied heavily in the literature on the quadratically close criterion. 
Roughly speaking, it amounts to show, after some computations, that
\begin{equation*}
\sum\limits_{n\in\mathbb{N}}\sum\limits_{p\in\mathbb{N}\setminus\{n\}}\frac{1}{|\lambda_{n}-\lambda_{p}+\lambda|^{2}}<+\infty
\end{equation*}
which holds if the eigenvalues $\lambda_{n}$ of the operator $A$ scales as $n^{\alpha}$ with $\alpha>3/2$ but fails as soon as $\alpha\leq 3/2$.

We introduce in this article the compactness/duality approach to overcome the limitations of the quadratically close criterion coupled with  $\omega$-independence/density properties. 
Indeed, this criterion corresponds to the rather strong Hilbert-Schmidt criterion for compactness. 
It turns out that the compactness part can be proved in a more general way, and we establish the $\omega$-independence thanks to a new duality observation, thus overcoming the apparent limit of $\alpha>3/2$ (see Remark \ref{rmrk-compactness-Riesz}). 
This allows us to prove Step 1 for {\color{black}skew-adjoint }operators with eigenvalues scaling as $n^{\alpha}$ (for instance $i|D_x|^\alpha$) with $\alpha\in (1,3/2]$.

Let us also stress that the $TB=B$ uniqueness condition is more difficult to handle than it seems since $B$ is an unbounded operator. Indeed, if $B$ were to be bounded, then from, 
\[
 \sum_n b_n K_n \tilde{T}\varphi_n =B,
\]
one would be tempted to deduce that the sequence $b_n K_n$ belongs to $\ell^2$. Moreover, the controllability assumption implies that the sequence of $|b_n|$ is bounded from below (it is impossible from the assumption that $B$ is bounded but let us assume it is for the sake of the argument) and therefore one would conclude that $K_n \in \ell^2$. But then, with the expression of $T$, it is not difficult to prove that in this case that the transformation $T$ would be compact and therefore not invertible. 

The proof of the decomposition of $TB=B$ for $B$ unbounded and admissible still follows the same idea, with the slight modification that $TB$ is seen as a singular and bounded part. Then, one adjusts the behaviour of $K_n$ by hand by letting $K_n\sim c+k_n$ where $c$ is independent of $n$. If the Riesz basis is quadratically close to the eigenfunctions, then one obtains,
\[
\sum_n b_n k_n \tilde{T}\varphi_n = \sum_n b_n (\varphi_n - \tilde{T}\varphi_n)
\]
and the right-hand side is bounded in $H$ using the quadratically close argument and the boundedness of the sequence $b_n$ in $\ell^\infty$ (provided, roughly speaking, by the admissibility of $B$).
Without the quadratically close property this direct argument fails. However, we prove here that even if the Riesz basis is not quadratically close to the eigenfunctions, we are still able to reach the same final conclusion by a close inspection of the left-hand side. Step 3 and 4 are obtained following standard techniques.

\subsection{Related works on the backstepping with a Fredholm transform}

There are essentially two type of systems in the literature for which rapid stabilization was achieved through the backstepping method with a Fredholm transformation : either the operator $A$ is of first order ($\alpha=1$) or of second order or higher ($\alpha\geq 2$). We have so far excluded from our discussion the case $\alpha=1$ as it seems to be a very specific case with techniques of its own. Indeed, the rapid stabilization for hyperbolic systems was established in \cite{CoronHuOlive16, MR4160602} through direct methods or by identifying the isomorphism applied to the eigenbasis leading to the Riesz basis \cite{coron:hal-03161523, Zhang-finite, ZhangRapidStab}. The other results found in the literature were concerned with operators such that $\alpha \geq 2$, and in these cases the Riesz basis properties was proved through the quadratically close criterion, thanks to the sufficient growth of the eigenvalues. Following the steps described in the previous section, the rapid stabilization was obtained for the linearized bilinear Schr\"odinger equation \cite{CGM}, the KdV equation \cite{CoronLu14}, the Kuramoto-Sivashinksi equation \cite{CoronLu15}, a degenerate parabolic operator \cite{GLM} and finally the heat equation for which the backstepping is proved in sharp spaces \cite{Gagnon-Hayat-Xiang-Zhang}. The variety of the PDEs for which this methodology can be applied tends to show that there exists an abstract theory for operators of order $\alpha >1$. Theorem \ref{thalphaintro} demonstrates this fact for skew-adjoint operators for $\alpha > 1$. This abstract setting could allow to lift some difficult questions raised when trying to apply the backstepping with the Volterra transformation. One such difficult\textcolor{black}{y} is seen for instance for degenerate parabolic equations (\cite{GLM}), where the Fredholm transformation lead to the study of well-known spectral properties of the Sturm-Liouville equation, whereas the PDE on the kernel of the Volterra transformation amounts to describe the propagation of bicharacteristics from a boundary satisfying a degenerate equation, a notoriously difficult problem. 

Finally, we shall remark that there are many other useful stabilization techniques in the literature that may also apply to sbimilar systems, for instance  the damping stabilization of waves \cite{zbMATH06919606, Bardos-Lebeau-Rauch, KriegerXiang2020}, the multiplier methods \cite{MR1359765}, Riccati theory \cite{Barbu-book, Breiten-Kunisch-Rodriques,  Lasiecka-Triggiani-book}, Gramian method \cite{Urquiza}, equivalence between observability and stabilizability \cite{Trelat-Wang-Xu}, quantitative finite time stabilization \cite{Xiang-heat-2020, Xiang-NS-2020},  various Lyapunov approaches \cite{Coron-Trelat-2004, SaintVenantPI, Hayat-Shang-2019, Hayat-Shang-2021} and among others.

\subsection{Related works on controllability and stabilization of the water waves system}
\vphantom{water-waves}

 A considerable amount of literature exists on the control of fluids. However, few works address the controllability and stabilization of the water waves. 
 Recently, this subject has drawn more and more attention. The controllability properties of these systems was first investigated in \cite{Reid-linearww} using the moment method, and the controllability assumption was sharpened to quasi-linear systems in \cite{Alazard-jems}, where the control is localized in the domain. 
 A recent work \cite{zbMATH07178296} investigated the water waves in 3-D, highlighting the need for the  geometric control condition for the controllability (the 2-D case satisfying automatically the geometric control condition as the control problem reduces to a 1-D equation).
 \\  Concerning stabilization, despite fruitful stabilization results were obtained in the literature for  fluids,  only few  result seems to exist regarding the water waves.  We may refer to the asymptotic stabilization and   exponential stabilization  results of water waves systems \cite{zbMATH06919606, Alazard-2018} that are based on external  ``damping" forces and ``observability" of the closed-loop systems. 
 Alternatively, stabilizability properties of  linearized water waves systems with controls acting on the boundary have been recently studied in \cite{Su-2021, Su-Tucsnak-Weiss}. Our contribution, as a direct consequence of the Fredholm backstepping transformation, is to obtain a rapid stabilization result, that is exponential stability with arbitrarily decay rate. \\   

\subsection{Outline of the paper}

The paper is divided in the following way. First, the main results are stated precisely in Section \ref{sec:mainresults} as well as their possible extensions, and the precise strategy of the proof is presented in Section \ref{sec:strategyandoutline}. {\color{black}Then, {the proof} is further divided in two parts. The first part deals with the proof of the main theorems in the abstract settings.}
In particular, Section \ref{sec:proof} begins with the statement of technical estimates used throughout the article as well as the proof of the Riesz basis (Step 1 of Section \ref{sec:related-method}) using the compactness/duality argument. 
In Section \ref{sec:proof2} we prove that the uniqueness condition $TB=B$ holds in $\Hdiv^{-1/2-\varepsilon}, \varepsilon \in (0,1/2)$ (Step 2) and define properly the feedback law $K$, the isomorphism $T$ (Step 3) as well as the operator equality \eqref{eq:introeqT}.
The well-posedness and rapid stabilization of the closed-loop system \eqref{closedloopsystem} is proved (Step 4) in Section \ref{sec-wellposed} and in turn proves that the sharp spaces for which the backstepping transform is established coincide with the sharp space for the well-posedness of the closed-loop system. 
{\color{black}The second part of the paper proves the rapid stabilization of the linearized water waves. The derivation of the {system} is first detailed, and then the abstract setting is shown to fit the functional framework of the linearized water waves, allowing to apply the main abstract results of this paper.} Finally, Appendix \ref{sec:Rieszdef} recalls the basic definitions for Riesz basis, Appendix \ref{sec:controllability_proof} proves the exact controllability {\color{black}of a class of skew-adjoint operators.} Appendix \ref{sec:beta} is devoted to the extension from Assumption \ref{asump1} to Assumption \ref{asump2}, Appendix \ref{sec:appendixc}--\ref{isomorph} are dedicated to technical proofs on the Fredholm transform for the backstepping method already existing in the literature.


 
\subsection{Notations and spaces}
\label{sec:notations}

{\color{black} Let $H$ be a Hilbert space and $A : D(A) \rightarrow H$ be an unbounded operator. 
Recall the definition of the domain of $A$:
\[D(A)=\left\{f\in H, \quad Af \in H\right\}.\]
We will assume throughout this paper that $A$ is densely defined, and that it is
skew-adjoint so that its eigenvalues are imaginary, and eigenvectors associated to distinct eigenvalues are orthogonal. {\color{black} See \cite{Brezis2011, CoronBook} for these classical functional analysis definitions.
}

\subsubsection{Simple eigenvalues case}
\label{subsubsec-Simple-eigenv-case}
We first consider the case that all eigenvalues $\lambda_n$ {\color{black} of the operator $A$} are simple. {\color{black}Thanks to the fact that $A$ is skew-adjoint},
we denote $(\varphi_n)_n$ an orthonormal basis of eigenvectors {\color{black} such that, 
\begin{equation*}
    A \varphi_n= \lambda_n \varphi_n, \quad \|\varphi_n\|_H= 1, \forall n\in \mathbb{N}^\ast.
\end{equation*}
}
We make the following crucial assumption on the growth of the eigenvalues:
 \begin{equation}
\label{hyp:1}
 \mathcal{P}(\alpha):
 \begin{array}{c}
\textrm{There exists} \ \alpha>1 \ \textrm{such that} \\
    |\lambda_n|{\color{black}+1}\sim n^{\alpha}, \quad n\in \N^*, \\
  |n_1-n_2|n_1^{\alpha-1} \lesssim |\lambda_{n_1}-\lambda_{n_2}|, \; \forall n_1,n_2 \in \N^*.
\end{array}
\end{equation}
It follows from this assumption that $A$ has compact resolvent. As a consequence, it has a bounded resolvent and is thus closed \cite{Brezis2011}. As it is skew-adjoint, it is dissipative on $D(A)$, and its adjoint as well. Then, by the Lumer-Phillips theorem \cite{PazyBook}, $A$ generates a strongly continuous semigroup on $H$.}
{\color{black}Now, for $s\in \R$ we can define the spaces $\Hdiv^{s}$ as, 
\begin{equation}\label{simple-H-i-s}
        \Hdiv^{s} := \{a = \sum\limits_{n\in\mathbb{N}^*} a_{n}\varphi_{n} 
    \;:\;\sum\limits_{n\in\mathbb{N}^*} |n|^{2s}|a_{n}|^{2}<+\infty\},
\end{equation}
endowed with the norms,
\[
\| a \|^2_{\Hdiv^s} :=  \sum\limits_{n\in\mathbb{N}^*} |n|^{2s}|a_{n}|^{2}<+\infty.
\]
Finally, the inner product $\langle\cdot, \cdot \rangle_{\Hdiv^s}$ is well-defined for any, $$f= \sum_{n\in\mathbb{N}^{*}} f_{n} \varphi_n, \; g= \sum_{n\in\mathbb{N}^{*}} g_{n} \varphi_n \in\Hdiv^{s},$$
it is given by, 
\begin{equation*}
\langle f, g\rangle_{\Hdiv^s}= \sum_{n\in \N^*} (n^s f_{n}) (n^s \overline{g_{n}}).
\end{equation*}

With this definition the spaces $\Hdiv^{s}$ acts like a Sobolev space but generated by $(\varphi_{n})_{n\in\mathbb{N}}$ instead of the classical basis of sin and/or cosine functions. Note that $\Hdiv^{0}=H$ \textcolor{black}{since $(\varphi_{n})_{n\in\mathbb{N}^{*}}$ is an orthonormal basis of $H$}. For usual cases where $A= i|D_{x}|^{\alpha}$ 
(such as Schrödinger equation or 
the water waves equation for instance) and the geometry considered is the torus, we can simply choose $H=L^{2}(\mathbb{T})$ and then $\Hdiv^{s}=H^{s}(\mathbb{T})$ where $H^{s}(\mathbb{T})$ is the classical Sobolev space on the torus.

\textcolor{black}{In general, let us point out the crucial embedding property, akin to Sobolev embeddings: for any $s\in\mathbb{R}$ and any $\varepsilon>0$, $\Hdiv^{s+\varepsilon}$ is compactly embedded in $\Hdiv^{s}$, each space being endowed with its corresponding norm as defined above.
}  Indeed, let us define the (finite-rank) projections and the continuous injection:
\begin{align*}P_N \, & : \ f=\sum_{n\in \N} f_n \varphi_n \in H \ \mapsto \ \sum_{n=0}^N f_n \varphi_n, \\ 
\iota \, & : \Hdiv^{s+\varepsilon}\rightarrow \Hdiv^s. 
\end{align*}
Let $f=\sum_{n\in \N} f_n \varphi_n \in \Hdiv^{s+\varepsilon}$, then
\[\|\iota f-P_N f\|^2_{\Hdiv^s}=\sum_{n\geq N+1} n^{2s} |f_n|^2=\sum_{n\geq N+1} n^{-2\varepsilon} n^{2s+2\varepsilon} |f_n|^2\leq N^{-2\varepsilon}\| f\|^2_{\Hdiv^{s+\varepsilon}}.\]
Hence, $\|\iota - P_N\|_{L(\Hdiv^{s+\varepsilon}, \Hdiv^s)}\leq N^{-\varepsilon}\xrightarrow[N\to \infty]{}0$. Hence, the injection $\iota$ is the uniform limit of a sequence of finite-rank operators. By \cite[Corollary 6.2]{Brezis2011}, it is thus compact.}

{\color{black}Moreover, in the general setting, given the property $P(\alpha)$ given by \textcolor{black}{\eqref{hyp:1}}, \ 
\[
D(A)=\Hdiv^\alpha.
\]

Notice that when $s\in \Z$, we have,
\[D(A^s)=\Hdiv^{\alpha s},\]
however for $s\in \R \setminus \Z$, $A^s$ is not \textcolor{black}{\emph{a priori}} well-defined as $A$ is not self-adjoint.}

\subsubsection{Bounded multiplicities case}

\textcolor{black}{In this subsection, we extend these notations to the case where $A$ has eigenvalues with multiplicities higher than $1$. {\color{black}This is, for example, the case for Laplace operator on the  one-dimensional torus, but these notations hold in the general setting.}  Let us assume that the eigenvalues of $A$ have at most multiplicity $m\in\mathbb{N}^{*}$.}
{\color{black}
Then, there exists a decomposition of $H$ in a finite direct sum of subspaces $H_i$, \textcolor{black}{$i\in\{1,...,m\}$}, {\color{black}invariant under the action of the unitary group $e^{t A}$}, on which $A$ has a spectral decomposition with simple eigenvalues (as in the previous subsection).
For more details on how to explicitly construct this decomposition, see Appendix \ref{sec:decomp}. 

Then, on each $H_i$, \textcolor{black}{there exists a basis of eigenvectors of $A$ that we denote $(\varphi_{n}^{i})_{\mathcal{I}_{i}}$ and} we can define the nested spaces $\Hdiv^s_i$ as in \eqref{simple-H-i-s} \textcolor{black}{with $(\varphi_{n}^{i})_{\mathcal{I}_{i}}$ instead of $(\varphi_{n})_{n\in\mathbb{N}}$}.
{\color{black} More precisely, for $i\in \{1,...,m\}$,
for $s\in \R$ we  define the spaces $\Hdiv_i^{s}$ as, 
\begin{equation}\label{simple-H-i-s-2}
        \Hdiv_i^{s} := \{a = \sum\limits_{n\in \textcolor{black}{\mathcal{I}_{i}}} a_{n}\varphi^i_{n} 
    \;:\;\sum\limits_{n\in\textcolor{black}{\mathcal{I}_{i}}} |n|^{2s}|a_{n}|^{2}<+\infty\},
\end{equation}
endowed with the norms, 
\[
\| a \|^2_{\Hdiv_i^s} :=  \sum\limits_{n\in\textcolor{black}{\mathcal{I}_{i}}} |n|^{2s}|a_{n}|^{2}<+\infty.
\]
}
For any real number $s$, we define, 
\begin{equation}\label{simple-eigenspace-decomposition}\Hdiv^s:= \Hdiv_1^s\oplus ... \oplus \Hdiv_m^s.
\end{equation}
{\color{black}Moreover, we can also define, for any $\vec s= (s_1,..., s_m)\in \mathbb{R}^m$, the space, 
\begin{equation*}
    \Hdiv^{\vec s}:= \Hdiv_1^{s_1}\oplus ... \oplus \Hdiv_m^{s_m}.
\end{equation*}
}

\textcolor{black}{Without loss of generality we can assume that there exists $m_{l}\in\{1,...,m\}$ such that
$\Hdiv_{i}$ is infinite dimensional and $\mathcal{I}_{i}=\mathbb{N}^{*}$ if $i\in\{1,...,m_{l}\}$, and $\Hdiv_{i}$ is finite dimensional if $i\in\{m_{l}+1,...,m\}$ with $\mathcal{I}_{i}=\{1,...N_{i}\}$.}

Finally, we require that on each $\Hdiv_i$, \textcolor{black}{with $i\in\{1,...,m_{l}\}$}, the ordered eigenvalues of $A$ satisfy the property $\mathcal{P}(\alpha_i)$ for some $\alpha_i>1$. 
\textcolor{black}{
More precisely, 
\begin{hyp}
\label{hyp:alpha}
For any $i\in\{1,...,m_{l}\}$, there exists $\alpha_{i}> {\color{black}1}$ such that $P(\alpha_{i})$ holds for $(\lambda_{n}^{i})_{n\in\mathbb{N}^{*}}$. 
\end{hyp}
}
}

{\color{black}
\begin{remark}
The definition of the space $\Hdiv_i^s$, with the renumbering of the eigenvalues, is made so that it behaves like classical Sobolev spaces and so that it is easier to deal with the coefficients of a function in some exotic cases. {Indeed, assume that $\lambda_{n} = n, n \in \N$ and that all the eigenvalues are simple, except when $n=m^2, m\in \N$, in which case they are double.} Then, by the definition of $\Hdiv_2$, {a function} $a$ belonging to $\Hdiv_2$ is such that 
\[
a=\sum_{n\in \N} a_n \varphi_n^2,
\]
instead of,
\[
a=\sum_{\substack{n\in \N \\ n=m^2, m\in \N}} a_n \varphi_n^2,
\]
for a more classical definition of $\Hdiv_2$ without renumbering. This way, our computations are uniform with respect of the multiplicity of the eigenvalues. 
\end{remark}

\begin{remark}\label{remark:15}
Notice that the decomposition $\Hdiv^s$ in direct sum is not unique. It may be written in multiple different ways, but from the functional setting defined here, the computations are uniform with respect to the different possible choice of definition of the spaces $\Hdiv_i^s$. Moreover, notice that the finite-dimensional spaces $\Hdiv_i^s$ for $i\in \{m_l+1, \ldots, m\}$ may be chosen arbitrarily, as long as they remain of finite dimension. Since the existence of the pair $(T_i,K_i)$ for these finite dimensional spaces are tackled with a finite dimensional result, our main results may be reinterpreted as a proof of convergence of the infinite-dimensional reminder. This could be particularly relevant for efficient finite-dimensional numerical methods. 

\end{remark}
}

\section{Main results}
\label{sec:mainresults}
{\color{black}
Let us first introduce the assumptions leading to the main results. Notice first that under the assumption of Section \ref{sec:notations}, $\overline{D(A)}^{\|.\|_{\Hdiv}}=\Hdiv$ and $A$ is \textcolor{black}{densely defined}. Therefore,  the Lumer-Philips Theorem implies that $A$ is the infinitesimal generator of the \textcolor{black}{semi}group $e^{At} : \Hdiv \rightarrow \Hdiv$. Since $A$ can be continuously extended as an operator from $\Hdiv^s$ to $\Hdiv^{s-\alpha}$, $s\in \R$, $A$ is the infinitesimal generator of the \textcolor{black}{semi}group $e^{At} : \Hdiv^s \rightarrow \Hdiv^s$ with $s\in \R$. 

{\color{black}Regarding the control operator $B$, recall that we consider the case of a finite number of scalar controls, with which system \eqref{eq:sys0} can be controllable under suitable assumptions on $A$ and $B$. In our study, we find that when the eigenvalues are simple, one control suffices. In general, given a decomposition of $H$ in subspaces on which $A$ has simple eigenvalues, using one scalar control for each subspace suffices. The smallest required number of scalar controls required is the smallest number of subspaces in such a decomposition, which corresponds to the number $m$ of subspaces in the decomposition \eqref{simple-eigenspace-decomposition}, that is, the highest multiplicity found in the spectrum of $A$}. 

We now consider\footnote{{\color{black}Throughout the paper, we shall denote ${\color{black}B_i}\in \Hdiv^{\textcolor{black}{s}}_i$ as a shorthand for ${\color{black}B_i: \mathbb{C} }\rightarrow \Hdiv^{\textcolor{black}{s}}_i$.}} the control operator $B=(B_1,\ldots , B_m)$ such that, for any $B_{i}\in \Hdiv^{-\textcolor{black}{\alpha_i/2}}_{i}$, for any $i\in\{1,...,m\}$ and 
\begin{equation}\label{eq:intro:bni}
\begin{split}
    B_{i} &= 
\sum\limits_{n=1}^{N_{i}}b_{n}^{i}\varphi_{n}^{i},\;\text{ in } \Hdiv^{-{\color{black}\alpha_i}/2}_{i},\;\text{ for }i\in\{m_{l}+1,...,m\},\\
    B_{i} &= 
\sum\limits_{n\in\mathbb{N}^{*}}b_{n}^{i}\varphi_{n}^{i},\;\text{ in } \Hdiv^{-{\color{black}\alpha_i}/2}_{i},\;\text{ for }i\in\{1,...,m_{l}\}.
\end{split}
\end{equation}
and we formulate the following assumption:}

\begin{hyp}\label{asump1}
{\color{black}The sequences $(b_n^i)_{i, n}$ introduced in \eqref{eq:intro:bni} satisfy the following inequalities:}
\textcolor{black}{
    \begin{equation}
\label{condB}
\begin{split}
&b_{n}^{i}\neq 0,\text{ for }n\in\{1,...,N_{i}\}, \;i \in\{m_{l}+1,...,m\},\\
c_{1}< |b^{i}_{n}| < c_{2},\;\;\text{ for }&i\in\{1,...,m_{l}\},\;n\in\mathbb{N}^{*},\text{ for some constants $c_1, c_2>0$.}
\end{split}
\end{equation}
}
\end{hyp}

\begin{remark}
{\color{black}Assumption \ref{asump1} \textcolor{black}{can be seen as an regularity assumption linked to admissibility and exact controllability. It is natural in the sense} that} {\color{black} the lower bound in Assumption \ref{asump1} is necessary and sufficient to get the exact controllability in $H$ (with $L^{2}(0,T)$ controls) (see the proof of Proposition \ref{prop:contfrollabilitylinear}), while the upper bound on $(b_n)_n$ is \textcolor{black}{necessary to the admissibility of the control operator (i.e. the existence of a mild solution in $C^{0}([0,T];H)$ to the system)} as shown in \cite{zbMATH06200909} (see the specific condition \cite[Section 2]{zbMATH06200909}).}
\end{remark}

Notice first that, under Assumption \ref{asump1}, the control system is well-posed as the - potentially unbounded in $\Hdiv^{s}$ - control operator \textcolor{black}{$B: \textcolor{black}{\mathbb{C}^{m}} \rightarrow \Hdiv^{-\alpha/2}$} is admissible in $\Hdiv^s$ for any $s \leq 0$ \textcolor{black}{(and in particular in $H$)}.
\begin{lem}\label{lem:admissible}
Let $A$ be a skew-adjoint operator such that Assumption \ref{hyp:alpha} holds. Assume the control operator \textcolor{black}{$B: \textcolor{black}{\mathbb{C}^{m}} \rightarrow \Hdiv^{-\alpha/2}$} satisfies Assumption \ref{asump1}. Then $B$ is admissible in $H$ (and therefore in $\Hdiv^s$ for any $s \leq 0$) that is, for any $T>0$,
\begin{equation}\label{admissible}
{\color{black}\left\|\int_0^T e^{A(T-t)}Bw(t) dt\right\|_{H}\leq C \|w\|_{L^2(0, T)}, \quad \forall w\in {\color{black}L^2((0,T);\textcolor{black}{\mathbb{C}^{m}})}}.
\end{equation}
Conversely, if the operator $B$ is admissible, then there exists $c>0$ such that $|b_n^i|<c$.
\end{lem}
The proof of Lemma \ref{lem:admissible} is found in Appendix \ref{sec:controllability_proof}.
\begin{remark}
From the proof, Lemma \ref{lem:admissible} is easily extended for controls $w\in H^\sigma((0,T);\textcolor{black}{\mathbb{C}^{m}}), \sigma \in \R$, and more general control operator $B$ satisfying Assumption \ref{asump2} below. But for the sake of conciseness, we restrict ourselves to controls in $L^2((0,T);\textcolor{black}{\mathbb{C}^{m}})$ and with control operators satisfying Assumption \ref{asump1}. 
\end{remark}
From \cite[Proposition 4.2.5]{TucsnakWeissBook} and  Lemma \ref{lem:admissible}, for any $u_0\in H$ and $w\in L^2((0,T);\textcolor{black}{\mathbb{C}^{m}})$, the problem
\begin{equation}\label{eq:controleintro}
\begin{cases}
\d_t u(t)=Au(t)+Bw(t), \quad t\in (0,T),\\
u(0)=u_0,
\end{cases}
\end{equation}
has a unique solution $u \in C([0,T];H)$.  
This solution is given by the Duhamel formula, 
\[
u(t)=e^{At}u_0+\int_0^t e^{A(t-s)}Bw(s) ds.
\]

This framework allows us to deduce the exact small-time controllability by the moments method using Haraux's refined version of Ingham's inequality. 
\begin{prop}\label{prop:contfrollabilitylinear}
Let $T>0$ \textcolor{black}{and define the control system
\begin{equation}
\label{eq:sysalpha}
    \partial_{t}u(t) =  Au(t) + Bw(t).
\end{equation}
}
Assume that {\color{black} Assumption \ref{hyp:alpha} and Assumption \ref{asump1} hold}. 
\textcolor{black}{Then, there exists some positive constant $C>0$ such that} for any $(u_0, u_f)\in (\textcolor{black}{H})^2$ 
there exists a control $w\in L^2((0,T);\textcolor{black}{\mathbb{C}^{m}} )$ satisfying {\color{black}
\begin{equation*}
    \|w\|_{L^2(0, T)}\leq C \left(\|u_0\|_{H}+  \|u_f\|_{H}\right),
\end{equation*}}
such that the unique solution of \eqref{eq:sysalpha} with initial state $u_0$ satisfies $u(T)= u_f$ in $H$.
\end{prop}

The proof of Proposition \ref{prop:contfrollabilitylinear} is postponed to the Appendix \ref{sec:controllability_proof}.\\

\textcolor{black}{Our main result is the following:}
{\color{black}
\begin{thm}
\label{thalpha}
Let the spatial operator $A$ such that Assumption \ref{hyp:alpha} holds. Let $B=(B_{1},\ldots,B_{m})\in \Hdiv^{-\textcolor{black}{\alpha_{1}/2}}_{1}\times...\times \Hdiv^{-\textcolor{black}{\alpha_{m}/2}}_{m}$  such that Assumptions \ref{asump1} holds and set $\alpha:=\min_{i}\alpha_{i}$.
For any $\lambda>0$, there exists an explicit bounded linear operator $K\in \mathcal{L}(\Hdiv^{\textcolor{black}{\alpha_{1}/2}}_{1}\times...\times \Hdiv^{\textcolor{black}{\alpha_{m}/2}}_{m};\mathbb{C}^{m})$ and an isomorphism $T$ from $\Hdiv^{r}$ to itself for $r\in(1/2-\alpha,\alpha-1/2)$ that maps the system 
\begin{equation}\label{geneoperaclosed}
    \partial_{t}u=Au+BK(u),\; (t,x) \in \mathbb{R}_{+} \times \T,
\end{equation}
to the system 
\begin{equation*}
 \partial_{t}v=Av-\lambda v,\;  (t,x) \in \mathbb{R}_{+} \times \T. 
\end{equation*}
Consequently, the closed-loop system \eqref{geneoperaclosed} is exponentially stable in $\Hdiv^{r}$ for $r\in (1/2- \alpha, \alpha- 1/2)$ with decay rate $\lambda$. \textcolor{black}{In particular it is exponentially stable in $\Hdiv^{0}=H$.}
\end{thm}
}


{\color{black}In fact, we are able to prove a more general result, as the regularity Assumption \ref{asump1} can be \textcolor{black}{generalized} to the following: 
\begin{hyp}\label{asump2}
Let $\beta_i\in \mathbb{R}$.
Let the operator $B=(B_{1},\ldots,B_{m})$ be such that $B_{i}\in \Hdiv^{\beta_i-\frac{\alpha_i}{2}}_{i}$. Then, assume the following  condition,
    \begin{equation}
\label{condB2}
\begin{split}
&b_{n}^{i}\neq 0,\text{ for }n\in\{1,...,N_{i}\}, \;i \in\{m_{l}+1,...,m\},\\
&{\color{black}c_{1} n^{-\beta_i}\leq |b^{i}_{n}| \leq c_{2}n^{-\beta_i},\;\;\text{ for }i\in\{1,...,m_{l}\},\;n\in\mathbb{N}^{*},\text{ for some constants $c_1, c_2>0$.}}
\end{split}
\end{equation}
\end{hyp}
Following the proof of Proposition \ref{prop:contfrollabilitylinear} {\color{black} and an adaptation of Lemma \ref{lem:admissible} based on Assumption \ref{asump2},} we deduce from this hypothesis,
\begin{prop}\label{prop:moreregularcontfrollabilitylinear}
Let $T>0$ and assume that {\color{black} Assumption \ref{hyp:alpha} and Assumption \ref{asump2} hold}. For any $(u_0, u_f)\in (\Hdiv^{\vec r})^2$  with $\vec r= (\beta_1,..., \beta_m)$,
there exists a control $w\in L^2((0,T);\textcolor{black}{\mathbb{C}^{m}} )$ such that the unique solution of \eqref{eq:sysalpha} with initial state $u_0$ satisfies $u(T)= u_f$ in $(\Hdiv^{\vec r})^2$.
\end{prop}
}
{\color{black}
\begin{thm}
\label{thalpha-2}
Let the spatial operator $A$ such that Assumption \ref{hyp:alpha} holds. Let $B=(B_{1},\ldots,B_{m})\in \Hdiv^{\beta_1- \frac{\alpha_1}{2}}_{1}\times...\times \Hdiv^{\beta_m- \frac{\alpha_m}{2}}_{m}$  such that Assumption \ref{asump2} holds.
For any $\lambda>0$, there exists an explicit bounded linear operator $K\in \mathcal{L}(\Hdiv^{\beta_1+ \frac{\alpha_1}{2}}_{1}\times...\times \Hdiv^{\beta_m+ \frac{\alpha_m}{2}}_{m};\mathbb{C}^{m})$ and an isomorphism $T$ from $\Hdiv^{\vec r}$ to itself for any, 
\begin{gather*}
    \vec r= (\beta_1+ r_1,..., \beta_m+ r_m) \textrm{ such that } \\
    r_i\in \left(\frac{1}{2}- \alpha_i, \alpha_i- \frac{1}{2}\right), \forall i\in \{1,..., m_l\}, \\
    r_i\in \mathbb{R}, \forall i\in \{m_l+1,..., m\}.
\end{gather*}
 that maps the system, 
\begin{equation}\label{geneoperaclosed-2}
    \partial_{t}u=Au+BK(u),\; (t,x) \in \mathbb{R}_{+} \times \T,
\end{equation}
to the system, 
\begin{equation*}
 \partial_{t}v=Av-\lambda v,\;  (t,x) \in \mathbb{R}_{+} \times \T. 
\end{equation*}
Consequently, the closed-loop system \eqref{geneoperaclosed-2} is exponentially stable in $\Hdiv^{\vec r}$  with decay rate $\lambda$. In particular it is exponentially stable in $\Hdiv^{0}=H$ provided that, 
\begin{equation*}
    \beta_i\in \left(\frac{1}{2}- \alpha_i, \alpha_i- \frac{1}{2}\right), \forall i\in \{1,..., m_l\}.
\end{equation*}
\end{thm}
}

We shall remark here that the bounds are sharp in the sense that \textcolor{black}{if, for any $i\in\{1,...,m_{l}\}$, $r_{i}\geq \beta_{i} +\alpha-1/2$, then} the unbounded operator $A+ BK$ does not anymore generate a strongly continuous semigroup in \textcolor{black}{$\Hdiv^{\vec r}$} (see Section \ref{sec-wellposed}). 
We also underline that, while the isomorphism $T$ depends on the regularity of the state space \textcolor{black}{$\Hdiv^{\vec r}$}, the feedback law is, surprisingly, independent of \textcolor{black}{$\vec r$}. This independence was already noticed in \cite{Gagnon-Hayat-Xiang-Zhang}. To prove \textcolor{black}{Theorem \ref{thalpha-2},} we introduce \textcolor{black}{the compactness-duality method}
to free ourselves from the bound $\alpha = 3/2$.\\

\textcolor{black}{This result can be directly applied to the capillarity-gravity water waves system. Indeed, given that the operator $\mathscr{L}$ defined in \eqref{op-WW} is a Fourier multiplier defined on the torus, its eigenvectors are sin and cosine functions and $\Hdiv^{s} = H^{s}(\mathbb{T})$ is a classical Sobolev space on the torus. Moreover, given \eqref{op-WW} one can check that all the eigenvalues of this operator have multiplicity 2, except for the eigenvalue {\color{black}$\lambda_0=0$} that has multiplicity $1$. Hence $m=m_{l}=2$ and the operator satisfies Assumption \ref{hyp:alpha} with $\alpha=3/2$. This is shown in Appendix \ref{sec:appendixc}.}
\begin{cor}[\textcolor{black}{Water waves equations - Backstepping transform}]
\label{th1}
Let $B=(B_{1},B_{2})\in \Hdiv^{-3/4}_{1}\times\Hdiv^{-3/4}_{2}$  such that  \eqref{condB} holds.
Then, for any $\lambda>0$, there exists an explicit bounded linear operator $K\in \mathcal{L}(\Hdiv^{3/4}_{1}\times \Hdiv^{3/4}_{2};\mathbb{C}^{2})$ and an isomorphism $T$ from $H^{r}(\mathbb{T})$ to itself for any $r\in(-1,1)$ that maps the system, 
\begin{equation}\label{closedloopsystem}
\begin{array}{ll}
\partial_t u =\mathscr{L}u +BK(u), & (t,x) \in \mathbb{R}_{+} \times \T, 
\end{array} 
\end{equation}
to the system, 
\begin{equation}\label{eq:linearizedtarget}
    \begin{array}{ll}
\partial_t v =\mathscr{L}v -\lambda v , & (t,x) \in \mathbb{R}_{+} \times \T, 
\end{array} 
\end{equation}
where $\mathscr{L}$ is the linearized  water-wave operator given in \eqref{op-WW}.
\end{cor}
As direct consequence we can deduce the existence of an explicit control law for the rapid exponential stabilization of the system \eqref{eq:linearizedWW}.
\begin{cor}[\textcolor{black}{Water waves equations - exponential stability}]\label{coro:expostab}
For any $\lambda>0$, there exists an explicit feedback functional $K\in \mathcal{L}(\Hdiv^{3/4}_{1}\times \Hdiv^{3/4}_{2};\mathbb{C}^{2})$ such that for any $r\in (-1 , 1)$ and for any initial state $u(t)|_{t=0}= u_0\in H^r(\mathbb{T})$, the closed-loop system \eqref{closedloopsystem}  has a unique solution $u\in C^0([0, +\infty); H^r(\T))$. In addition, this unique solution decays exponentially with rate $\lambda$,
\begin{equation*}
\|u(t,\cdot)\|_{H^r}\lesssim  e^{-\lambda t}\|u_{0}\|_{H^r}, \; \quad \forall t\in (0, +\infty).
\end{equation*}
\end{cor}

As stated in the introduction, the system \eqref{eq:linearizedWW} is all the more interesting as it represents the critical case $\alpha=3/2$ where the usual method fails.\\





\textcolor{black}{\begin{remark}[Rapid stabilization in $H^{s}(\mathbb{T})$ instead of $L^{2}(\mathbb{T})$]
In Corollaries \ref{th1} and \ref{coro:expostab} we chose $H=L^{2}(\mathbb{T})$, thus the feedback law stabilizes the system in $H^r$ for every $r\in (-1, 1)$. However, if  operators $B_1$ and $B_2$ have different regularity and satisfy instead \eqref{condB2} with $\beta_1= \beta_2= s$, then,  thanks to  Theorem \ref{thalpha-2},  we can construct feedback laws such that the closed-loop system is stable in $H^{r+s}$ for every $r\in (-1, 1)$.
\end{remark}
Other remarks about mass-like conservation condition and the stabilization of the water-wave equations on a bounded domain can be found in Section \ref{sec:water-waves} dedicated to the case of the water-wave equations.
}

\section{Strategy and outline}
\label{sec:strategyandoutline}
 In this section, we briefly comment on the strategy to prove Theorem \ref{thalpha-2} which is the task of the next two sections, while the proof of Corollary \ref{coro:expostab} concerning the well-posedness and stability of the closed-loop system will be discussed later on in Section  \ref{sec-wellposed}.  Given the decomposition {\color{black} of the space $\Hdiv^s$  along sub-spaces  $\Hdiv^s_i$ with $i\in \{1,...,m\}$}, showing Theorem \ref{thalpha-2} amounts to proving the following proposition. 
\begin{prop}
\label{prop:th}
Let $i\in \{1,\textcolor{black}{...,m}\}$.
Let $B_{i}\in \Hdiv_{i}^{\beta_i- \frac{\alpha_i}{2}}$ satisfying \eqref{condB2}.
For any $\lambda>0$, there exists a bounded linear operator $K_{i}\in \mathcal{L}(\Hdiv^{\beta_i+\frac{\alpha_i}{2}}_{i};\mathbb{C})$ and an isomorphism $T_{i}$ {\color{black} from $\textcolor{black}{\Hdiv_i^{r}}$ to itself, for $r\in\textcolor{black}{(\beta_i+ 1/2-\alpha_i,\beta_i+ \alpha_i-1/2)}$, }which maps the system, 
\begin{equation}\label{sysdepart1}
\begin{array}{ll}
\partial_t u =\textcolor{black}{A}u +B_{i}K_{i}(u),\;\;u\in \Hdiv^{r}_{i},
\end{array}
\end{equation}
to the system, 
\begin{equation}\label{sysarrive1}
    \begin{array}{ll}
\partial_t v =\textcolor{black}{A}v -\lambda v,\;\;v\in \Hdiv^{r}_{i}.
\end{array} 
\end{equation}
\end{prop}
Indeed, if Proposition \ref{prop:th} holds, then the {\color{black}isomorphism-feedback} pair $(T,K)$ of Theorem \ref{thalpha-2} is simply $T = T_{1}\oplus\textcolor{black}{...\oplus T_{m}}$ and $K=(K_{1},\textcolor{black}{...,K_{m}})$. 
\begin{remark}
\textcolor{black}{In the following we will select $i\in\{1,...m_{l}\}$ and work with this fixed $i$. The case of $i\in\{m_{l}+1,...,m\}$ corresponding to a finite dimensional space $\Hdiv_{i}$ can be tackled using a classical finite-dimensional backstepping described below (see also {\color{black}\cite{Coron_ICIAM15}}) and will not be described here.}
\textcolor{black}{For the reader's convenience, we will also drop from now on the index $\textcolor{black}{i}$ and denote again $\Hdiv^{r}$, $\alpha$, $\beta$, $H$, $K$, $T$, $\varphi_{n}$ instead of $\Hdiv^{r}_{\textcolor{black}{i}}$, $\alpha_{i}$, $\beta_{i}$, $H_{i}$, $K_{\textcolor{black}{i}}$, $T_{\textcolor{black}{i}}$, $\varphi_{n}^{\textcolor{black}{i}}$ respectively. Moreover,  it suffices to consider the case that $\beta= 0$ (see Appendix \ref{sec:beta} for an extension to the case $\beta\neq 0$)}
  \end{remark}

Before proving Proposition \ref{prop:th}, let us make some formal observations to get an intuition of the problem. To map the system \eqref{sysdepart1} onto \eqref{sysarrive1}, what we would like to do is to obtain (formally) the following operator equality,
\begin{equation}
\label{op-eq0}
T(A+BK) = (A-\lambda \textcolor{black}{I})T.
\end{equation}
As noted in \cite{CGM, CoronLu15, Gagnon-Hayat-Xiang-Zhang}, a good approach to this aim is to add a condition on $TB$, and to require instead the two following operator equalities,
\begin{equation}
\label{op-eq00}
\begin{split}
& TA+BK = (A-\lambda \textcolor{black}{I})T,\\
& TB = B,
\end{split}
\end{equation}
in a certain sense to be specified. Note that \eqref{op-eq00} implies formally \eqref{op-eq0} and requiring \eqref{op-eq00} instead of \eqref{op-eq0} allows to deal with operator equations that are linear with $(T,K)$. It also usually ensures the uniqueness of solution (see for instance \cite{Gagnon-Hayat-Xiang-Zhang}). In finite dimension, \eqref{op-eq0} corresponds to an equivalent formulation of the pole-shifting theorem \cite[Section 2]{coron:hal-03161523}. Applying the first operator equality to the orthonormal basis of eigenvectors $(\varphi_{n})_{n\in\mathbb{N}^{*}}$ gives,
\begin{equation}
\label{op-eq-cons1}
\lambda_{n}(T\varphi_{n})+ BK(\varphi_{n}) = (A-\lambda I)(T\varphi_{n}),
\end{equation}
where we used the fact that $\varphi_{n}$ is an eigenvector of $A$ \textcolor{black}{and where $\lambda_{n}$ is the eigenvalue associated to $\varphi_{n}$}. Observe that \eqref{op-eq-cons1} is a differential equation on $(T\varphi_{n})$. Projecting now on a vector $\varphi_{p}$ and recalling that $b_{p}=\langle B,\varphi_{p}\rangle$, this becomes,
\begin{equation}
\begin{split}
\label{op-eq-cons2}
(\lambda_{n}+\lambda)\langle (T\varphi_{n}),\varphi_{p}\rangle + \langle B,\varphi_{p}\rangle K(\varphi_{n}) &= \langle A(T\varphi_{n}),\varphi_{p}\rangle,\\
&=\langle (T\varphi_{n}),A^{*}\varphi_{p}\rangle=\lambda_{p}\langle (T\varphi_{n}),\varphi_{p}\rangle.
\end{split}
\end{equation}
This gives the following formal expression,
\begin{equation}
\label{eq:defT0}
   T\varphi_{n} = \sum\limits_{p\in\mathbb{N}^{*}} \langle (T\varphi_{n}),\varphi_{p}\rangle\varphi_{p}=(-K(\varphi_{n}))\sum\limits_{p\in\mathbb{N}^{*}}\frac{b_{p}\varphi_{p}}{\lambda_{n}-\lambda_{p}+\lambda}.
\end{equation}
\\

These formal calculations lead us to introduce the following
notations that will be used all along the proof. We define,
\begin{itemize}
    \item The families
    \begin{equation}
    \label{def-qn}
        q_{n} := \sum\limits_{p\in\mathbb{N}^{*}}\frac{\varphi_{p}}{\lambda_{n}-\lambda_{p}+\lambda},\;\;\;K_{n} = K(\varphi_{n}),\;\;n\in\mathbb{N}^{*}.
    \end{equation}
    \item The operator 
    \begin{equation}
    \label{def-S}
    S: \textcolor{black}{\varphi_{n}\mapsto q_{n}}.
    \end{equation}
    Note that $S$ is completely defined as an operator on $\Hdiv^{r}$ {\color{black} for a certain range of $r$ that will be considered later on in this paper (eventually, for $r\in(1/2-\alpha,\alpha-1/2)$).
    This is a consequence of}
    $(n^{-r}\varphi_{n})_{n\in\mathbb{N}^{*}}$ \textcolor{black}{being} an orthonormal basis of $\Hdiv^{r}$ for any $r\in \R$ {\color{black} as well as direct estimates on $(q_n)_{n\in\mathbb{N}^{*}}$.} 
    \item The operator
    \begin{equation}\label{def-tau}
    \tau : \varphi_{n}\mapsto b_{n}\varphi_{n}.
    \end{equation}
    Note that since $b_{n}$ are uniformly bounded \textcolor{black}{from} above and below, this is an isomophism from $\textcolor{black}{H}$ to $\textcolor{black}{H}$, and in fact also from $\Hdiv^r$ to $\Hdiv^r$ for any $r\in \R$.\\
    
    \item The operator $T$ defined on $\Hdiv^{r}$
    \begin{equation}
    \label{def-T}
        T : \textcolor{black}{\varphi_{n}\mapsto (-K_{n})\tau q_{n}}.
    \end{equation}
    Note that this expression of $T$ corresponds exactly to the expression \eqref{eq:defT0} obtained from the formal calculations.
\end{itemize}
In the following we will show that, for a good choice of $(K_{n})_{n\in\mathbb{N}^{*}}$, the operator $T$ thus defined is an isomorphism from $\Hdiv^{r}$ to itself for $r\in \textcolor{black}{(1/2-\alpha, \alpha-1/2)}$, and $T$ and $K$ satisfy \eqref{op-eq00} in a sense to be specified.\\

We are going to show successively the following steps: 
\begin{enumerate}
     \item Show that $S$ is a Fredholm operator from $\Hdiv^{r}\rightarrow \Hdiv^{r}$ for any $r\in \textcolor{black}{(1/2-\alpha, \alpha-1/2)}$.\\
     
    \item Show that $(q_{n})_{n\in\mathbb{N}^{*}}$ is a Riesz basis for $\textcolor{black}{H}$ using a duality argument and the fact that $S$ is Fredholm.\\
    
    \item Further show that  $(n^{-r} q_{n})_{n\in\mathbb{N}^{*}}$ is a Riesz basis for $\Hdiv^{r}$ for any $r\in \textcolor{black}{(1/2-\alpha, \alpha-1/2)}$ by showing it is $\omega$-independent using a duality argument between the density of $(n^{-r}q_{n})_{n\in\mathbb{N}^{*}}$ in $\Hdiv^{r}$ and the $\omega$-independence of $(n^{r}\overline{q_{n}})_{n\in\mathbb{N}^{*}}$ in $\Hdiv^{-r}$.\\

    \item 
    Provide an explicit candidate of $(K_n)_{n\in\mathbb{N}}$ which satisfies $TB= B$ in $\Hdiv^{-\textcolor{black}{\alpha/2}}$. Show that $(|K_{n}|)_{n\in\mathbb{N}}$ is bounded from above and that $b_{n} K_{n} = -(\lambda+k_{n})$ for any $n\in\mathbb{N}^*$, where $(k_{n} n^{\varepsilon})_{n\in\mathbb{N}^{*}}\in l^{\infty}$ for any $\varepsilon \in [0,\textcolor{black}{\varepsilon(\alpha)})$ \textcolor{black}{(where $\varepsilon(\alpha)$ is a positive constant that can be computed)}.\\

    \item Show that $T$ is uniformly bounded from $\Hdiv^{r}$ to itself for $r\in \textcolor{black}{(1/2-\alpha, \alpha-1/2)}$ and the first operator equality \eqref{op-eq00} holds in $\mathcal{L}(\Hdiv^{\textcolor{black}{\alpha/2}}; \Hdiv^{-\textcolor{black}{\alpha/2}})$.\\
    
    \item Show that $T$ is a Fredholm operator from $\Hdiv^{-\textcolor{black}{\alpha/2}}$ to $\Hdiv^{-\textcolor{black}{\alpha/2}}$.\\
    
    \item Show that $T$ is an isomorphism from $\Hdiv^{-\textcolor{black}{\alpha/2}}$ to $\Hdiv^{-\textcolor{black}{\alpha/2}}$ using a Fredholm argument and spectral theory in $\Hdiv^{-\textcolor{black}{\alpha/2}}$. \\
    
    \item Show that $T$ is an isomorphism from $\textcolor{black}{H}$ to $\textcolor{black}{H}$ and in fact an isomorphism from $\Hdiv^{r}$ to itself for any $r\in \textcolor{black}{(1/2-\alpha, \alpha-1/2)}$.
\end{enumerate}
Let us briefly discuss steps 6 to 8 as, at first sight, it seems odd to prove the invertibility in $\Hdiv^{-\textcolor{black}{\alpha/2}}$ and not in \textcolor{black}{our reference} $\textcolor{black}{H}$ space for instance. The main motivation is to avoid working in the space $D(A+BK):=\{f\in \textcolor{black}{H} \, : \, (A+BK)f \in \textcolor{black}{H}\}$ before proving the invertibility of $T$. Indeed, in the present setting, the space $D(A+BK)$ does not have nice properties shared by the \textcolor{black}{$\Hdiv^{s}$} spaces \textcolor{black}{such as the density of $\Hdiv^{\infty}$ functions\footnote{\textcolor{black}{Note that in usual cases where, for instance, $\Hdiv^{s}$ = $H^{s}(\mathbb{T})$, the $\Hdiv^{\infty}$ functions are simply $C^{\infty}$ functions, from Sobolev embedding. {\color{black}Notice that the eigenfunctions $(\varphi_n)_n$ belong to the space $\Hdiv^{\infty}$}.}}, defined as 
\begin{equation*}
    \Hdiv^{\infty}:=\bigcap_{k\in \mathbb{N}} \Hdiv^k,
\end{equation*}
see for instance \eqref{high-reg-domains}}.
This comes from the fact that $B$ is not regular enough, and therefore one is not able to conclude that $\varphi_n \in D(A+BK)$ for any $n\in \N^*$. Hence, it is easier to first prove the invertibility in weaker but classical \textcolor{black}{$\Hdiv^{s}$} space (step 6 and 7) before deducing the invertibility in the required spaces (step 8). In turn, the invertibility of $T$ in $\Hdiv^r$ allows us to construct an equivalent norm on $\Hdiv^r$, which allows us to prove that $D(A+BK)$ is an Hilbert space, a non-trivial task without the invertibility of $T$. We underline that if our setting is close to the linearized bilinear Schr\"odinger equation, the fact that the control is real-valued in \cite{CGM} allows to decouple the real and imaginary part of the solution to deal directly with the space $D(A+BK)$, which is not the case here.

We start by introducing some technical lemmas in Subsection \ref{sec:techlem}. Then we prove Proposition \ref{prop:th}, following the outline above : we prove steps (1)-(3) in Section \ref{sec:proof} and steps (4)-(8) in Section \ref{sec:proof2}. Finally, we prove the well-posedness of the closed-loop system obtained and Corollary \ref{coro:expostab} in Section \ref{sec-wellposed}.

\section{Compactness/duality method for Riesz basis}
\label{sec:proof}
Following the outline detailed in Section \ref{sec:strategyandoutline}, we devote this section to the proofs of Steps (1)-(3). These steps form a first important part of the proof of our main theorem: they revolve around Riesz basis properties for some important families of functions derived from the backstepping operator equalities \eqref{op-eq00}. As we have mentioned in the introduction, we introduce here a new method based on compactness and duality, namely, we prove in a general way that the transformations involved in our backstepping method are Fredholm operators.

\subsection{Some basic estimates}
\label{sec:techlem}
In this section we introduce some technical lemmas that will be used in the following. For the readers' convenience, the technical parts of certain proofs are postponed to Appendix \ref{sec:appendixc}. 

The first  lemma is a direct consequence of the existence of $c, C>0$ such that, 
\begin{equation*}
    c n^{\textcolor{black}{\alpha}}\leq |\lambda_n|{\color{black}+1} \leq C n^{\textcolor{black}{\alpha}}, \; \forall n\in \N^*.
\end{equation*}
\begin{lem}\label{lem:Lresoland}
Let $s\in \mathbb{R}$. Let $\rho\in \C$ be in the resolvent set of the operator $A$. We know that 
\begin{gather*}
    A: \Hdiv^{s+ \textcolor{black}{\alpha/2}}\rightarrow \Hdiv^{s- \textcolor{black}{\alpha/2}} \textrm{ is continuous,} \\
    (A- \rho I)^{-1}: \Hdiv^{s- \textcolor{black}{\alpha/2}}\rightarrow \Hdiv^{s+ \textcolor{black}{\alpha/2}} \textrm{ is continuous.}
\end{gather*}
\end{lem}

Next, we give an important technical lemma:

\textcolor{black}{
\begin{lem}
\label{lem:tech3}
For any $s<\alpha-1$ we have
\begin{equation}
    \sum\limits_{n\in\mathbb{N}^{*}\setminus\{p\}}\frac{n^{s}}{|\lambda_{n}-\lambda_{p}|}\lesssim p^{1- \alpha+ s}\log(p)+p^{-\alpha}, \; \forall p\in \N^*.
\end{equation}
\end{lem}
}
\begin{proof}[Proof of Lemma \ref{lem:tech3}]
Let $s<\textcolor{black}{\alpha-1}$, we have 
\begin{equation*}
    \sum\limits_{n\in\mathbb{N}^{*}\setminus\{p\}}\frac{n^{s}}{|\lambda_{n}-\lambda_{p}|} = I_{1}+I_{2}+I_{3},
\end{equation*}
where
\begin{equation*}
\begin{split}
    I_{1} &= \sum\limits_{n\in\mathbb{N}^{*},\;n\leq p/2}\frac{n^{s}}{|\lambda_{n}-\lambda_{p}|},\\
    I_{2} &=\sum\limits_{n\in\mathbb{N}^{*}\setminus\{p\},\;p/2<n< 2p}\frac{n^{s}}{|\lambda_{n}-\lambda_{p}|},\\
    I_{3} &= \sum\limits_{n=2p}^{+\infty}\frac{n^{s}}{|\lambda_{n}-\lambda_{p}|}.
    \end{split}
\end{equation*}
We will show that all these three terms can be bounded by $C(p^{-\textcolor{black}{(\alpha-1)}+s}\log(p)+p^{-\textcolor{black}{\alpha}})$ where $C$ is a constant independent of $p$. For this, we introduce the following basic estimates (the proofs may be found in Appendix \ref{sec:appendixc}).

\begin{lem}
\label{lem1}
There exists $c>0$ such that for any $(n,m)\in\mathbb{N}^*$
\begin{gather*}
|\lambda_{n}-\lambda_{m}|\geq c |n-m|^{\textcolor{black}{\alpha}}.
\end{gather*}
\end{lem}

\begin{lem}
\label{lem:tech1}
For any $s\neq -1$, there exists $C>0$ such that for any $p\in\mathbb{N}^{*}$
\begin{equation*}
    \sum\limits_{n=1}^{p} n^{s}\leq C(1+p^{1+s}).
\end{equation*}
\end{lem}

\begin{lem}
\label{lem:tech2}
For any $s\in\mathbb{R}$ {\color{black}and $\varepsilon>0$, there exists $C>0$ such that,  }
\begin{equation*}
        \sum\limits_{n=1}^{p} n^{s}\log(n)\leq C(1+p^{1+s+\varepsilon}).
\end{equation*}
\end{lem}

We first consider $I_{1}$. \textcolor{black}{First note that} there exists $C>0$ independent of $p$ and $n$, \textcolor{black}{such that if $n\leq p/2$, then} $|\lambda_{p}-\lambda_{n}|\geq C^{-1}p^{\textcolor{black}{\alpha}}$. Using Lemma \ref{lem1}
\begin{equation}
\label{eq:estimI1}
\sum\limits_{n\in\mathbb{N}^{*},n\leq p/2}\frac{n^{s}}{|\lambda_{n}-\lambda_{p}|}
\lesssim  p^{\textcolor{black}{-\alpha}}\left(\sum\limits_{n\in\mathbb{N}^{*}, n\leq p/2}n^{s}\right)\lesssim  p^{-\textcolor{black}{\alpha}}+p^{-\textcolor{black}{(\alpha-1)}+s}\log(p),
\end{equation}
where in the rightmost inequality we used Lemma \ref{lem:tech1} if $s\neq -1$ and $p^{1+s}\leq p^{1+s}\log(p)$ for $p$ large enough, and if $s=-1$ then we simply used that $\left(\sum\limits_{n=0}^{p}1/n\right)=O(\log(p))$.\\

Then we turn to  $I_{2}$, using Lemma \ref{lem1} we have
\begin{align*}
\sum\limits_{n\in\mathbb{N}^{*}\setminus\{p\},p/2<n< 2p}\frac{n^{s}}{|\lambda_{n}-\lambda_{p}|}
& \lesssim p^{s}\left(\sum\limits_{n\in\mathbb{N}^{*}\setminus\{p\},p/2<n< 2p}\frac{1}{|\lambda_{n}-\lambda_{p}|}\right) \\
& \lesssim  p^{s}\left(\sum\limits_{n\in\mathbb{N}^{*}\setminus\{p\}, p/2<n< 2p}\frac{1}{|n-p|p^{\textcolor{black}{\alpha-1}}}\right). 
\end{align*}
Notice that
\begin{equation*}
\sum\limits_{n\in\mathbb{N}^{*}\setminus\{p\},p/2<n< 2p}\frac{1}{|n-p|}\leq \sum\limits_{k\leq p/2}\frac{1}{k}+\sum\limits_{k=1}^{p}\frac{1}{k}\lesssim \log(p),
\end{equation*}
hence
\begin{equation}
\label{eq:estimI2}
\sum\limits_{n\in\mathbb{N}^{*}\setminus\{p\},p/2<n< 2p}\frac{n^{s}}{|\lambda_{n}-\lambda_{p}|}
\lesssim p^{-\textcolor{black}{(\alpha-1)}+s}\log(p),
\end{equation}
which gives the bound on $I_{2}$.

We finally consider $I_{3}$. Since $n>2p$, there exists a constant $C>0$ independent of $p$ such that $|\lambda_{n}-\lambda_{p}|\geq C^{-1} n^{\textcolor{black}{\alpha}}$ from Lemma \ref{lem1}, thus
\begin{equation}
\label{eq:estimI3}
I_{3}= \sum\limits_{n=2p}^{+\infty}\frac{n^{s}}{|\lambda_{n}-\lambda_{p}|}\lesssim \sum\limits_{n=2p}^{+\infty}n^{s-\textcolor{black}{\alpha}}
\lesssim \int_{2p}^{+\infty} x^{s-\textcolor{black}{\alpha}}dx
\lesssim  p^{-\textcolor{black}{(\alpha-1)}+s},
\end{equation}
where we used that $s<\textcolor{black}{(\alpha-1)}$ thus $s-\textcolor{black}{\alpha}<-1$. Combining \eqref{eq:estimI1}, \eqref{eq:estimI2} and \eqref{eq:estimI3} we deduce that 
\begin{equation*}
I_{1}+I_{2}+I_{3}\lesssim p^{-\textcolor{black}{(\alpha-1)}+s}\log(p)+p^{-\textcolor{black}{\alpha}}.
\end{equation*}
This ends the proof of Lemma \ref{lem:tech3}.
\end{proof}

\subsection{Step (1): a general Fredholm operator}
In this subsection, we show the following proposition.
\begin{prop}
\label{S:fredholm}
For any \textcolor{black}{$r\in(1/2-\alpha,\alpha-1/2)$}, there exists a compact operator $S_{c}$ from $\Hdiv^{r}$ into itself such that the operator $S$ defined by \eqref{def-S} satisfies on $\Hdiv^{r}$,
\begin{equation}
\label{decom-S}
    S = \frac{1}{\lambda}Id+S_{c}.
\end{equation}
In particular, $S$ is a Fredholm operator (of index 0) from $\Hdiv^{r}$ into itself.
\end{prop}
Let us first show that for any $a\in \Hdiv^{r}$, denoting $a_{n} = \left( a,n^{-r}\varphi_{n}\right)_{\Hdiv^{r}}$, such that $a = \sum\limits_{n\in\mathbb{N}^{*}}a_{n}n^{-r}\varphi_{n}$, we have, 
\begin{equation}
\label{decompS2}
    S a = \frac{a}{\lambda}+ \left(\sum\limits_{n\in\mathbb{N}^{*}}a_{n}n^{-r}\left(\sum\limits_{p\in\mathbb{N}^{*}\setminus\{n\}}\frac{\varphi_{p}}{\lambda_{n}-\lambda_{p}+\lambda}\right)\right).
\end{equation}
Indeed, from the definition of \eqref{def-S},
\begin{equation*}
\begin{split}
    S a =& \sum\limits_{n\in\mathbb{N}^{*}} a_{n}n^{-r}S(\varphi_{n}) = \sum\limits_{n\in\mathbb{N}^{*}} a_{n}n^{-r}q_{n}\\
    &=\sum\limits_{n\in\mathbb{N}^{*}} a_{n}n^{-r}\left(\sum\limits_{p\in\mathbb{N}^{*}}\frac{\varphi_{p}}{\lambda_{n}-\lambda_{p}+\lambda}\right)\\
    &=\frac{1}{\lambda}\left(\sum\limits_{n\in\mathbb{N}^{*}} a_{n}n^{-r}\varphi_{n}\right)+\sum\limits_{n\in\mathbb{N}^{*}}a_{n}n^{-r}\left(\sum\limits_{p\in\mathbb{N}^{*}\setminus\{n\}}\frac{\varphi_{p}}{\lambda_{n}-\lambda_{p}+\lambda}\right),
\end{split}
\end{equation*}
which, given the definition of $a$, is exactly \eqref{decompS2}. Now we show the following.
\textcolor{black}{
\begin{lem}
\label{lem4}
For any $r\in(1/2-\alpha, \alpha-1/2)$, there exists $\varepsilon =\varepsilon(r)>0$ such that the operator $S_{c}$ defined by,
\begin{equation*}
  S_c:  \sum\limits_{n\in\mathbb{N}^{*}}a_{n}n^{-r}\varphi_{n} \mapsto \sum\limits_{n\in\mathbb{N}^{*}}a_{n}n^{-r}\left(\sum\limits_{p\in\mathbb{N}^{*}\setminus\{n\}}\frac{\varphi_{p}}{\lambda_{n}-\lambda_{p}+\lambda}\right),
\end{equation*}
is continuous from $\Hdiv^{r}$ to $\Hdiv^{r+\varepsilon}$.
In particular, this operator is compact from $\Hdiv^{r}$ to itself. 
\end{lem}
}

The proof of Lemma \ref{lem4} is based on a careful estimation allowed by the Lemma \ref{lem:tech3}, and we give its proof below. 
Proposition \ref{S:fredholm} then follows from \eqref{decompS2} and Lemma \ref{lem4}.
\begin{remark}\label{remark:rightbound}
Note, as a corollary of Lemma \ref{lem4} and the expression of $S$ given by \eqref{decompS2}, that for any $r\in\textcolor{black}{(1/2-\alpha,\alpha-1/2)}$, there exists $C>0$ such that for any $(a_{n})_{n\in\mathbb{N}^{*}}\in l^{2}$ one has,
\begin{equation}
\label{S:bounded}
\left\|\sum\limits_{n\in\mathbb{N}^{*}} a_{n}n^{-r}q_{n} \right\|_{\Hdiv^{r}}^2\leq C\sum\limits_{n\in\mathbb{N}^{*}}|a_{n}|^{2}.
\end{equation}
which means that $S$ is a bounded operator from $\Hdiv^{r}$ into itself. In fact, Proposition \ref{S:fredholm} is stronger since it shows that $S$ is even a Fredholm operator {\color{black}from $\Hdiv^r$ to itself}.
\end{remark}

\begin{remark}\label{rmrk-compactness-Riesz}
As we have mentioned in the introduction, previous works on the backstepping method use the quadratically close criterion to prove that $q_n$ is a Riesz basis. In our case, one would then seek to prove that 
\begin{equation}
\label{HS0}
    \sum_{n\in \N^\ast} \left\|n^{-r} q_n - \frac1{\lambda}n^{-r}\varphi_n\right\|^2_{\Hdiv^r} < +\infty.
\end{equation}
In terms of $S_c$, this amounts to
\begin{equation}
    \label{Hilbert-Schmidt-Sc}
      \sum_{n\in \N^\ast}  \|S_c n^{-r} \varphi_n\|^2_{\Hdiv^r} < +\infty,
\end{equation}
which is the Hilbert-Schmidt compactness criterion for $S_c$. However, in our case, \eqref{HS0}--\eqref{Hilbert-Schmidt-Sc} does not hold. Our new compactness/duality method illustrates that the relevant property is not {\color{black} the Hilbert-Schmidt compactness criterion \eqref{Hilbert-Schmidt-Sc}}, but simply the compactness of $S_c$, which we prove here in a more general way. This, together with the duality argument presented in Steps (2) and (3) below, leads to Riesz basis properties thanks to the Fredholm alternative. 

Interestingly, this illustrates that there is a link between the growth of the eigenvalues and the class of compact operators that appear in the Fredholm decomposition of $S$. 
\end{remark}
\begin{proof}[Proof of Lemma \ref{lem4}]
We start by considering two different cases. First, let $r\in(\textcolor{black}{1/2-\alpha},0]$. What we need to show is that there exists $C>0$ and $\varepsilon>0$ such that for any $(a_{n})_{n\in\mathbb{N}^{*}}\in l^{2}$,
\begin{equation*}
  \left\|\sum\limits_{n\in\mathbb{N}^{*}}a_{n}n^{-r}\left(\sum\limits_{p\in\mathbb{N}^{*}\setminus\{n\}}\frac{\varphi_{p}}{\lambda_{n}-\lambda_{p}+\lambda}\right)\right\|_{\Hdiv^{r+\varepsilon}}\leq C\left\|\sum\limits_{n\in\mathbb{N}^{*}}a_{n}n^{-r}\varphi_{n}\right\|_{\Hdiv^{r}}.
\end{equation*}
Notice that the following {\color{black} equality} holds in {\color{black}$\Hdiv^{1/2-\alpha}$},
\begin{equation*}
\sum\limits_{n\in\mathbb{N}^{*}}a_{n}n^{-r}\left(\sum\limits_{p\in\mathbb{N}^{*}\setminus\{n\}}\frac{\varphi_{p}}{\lambda_{n}-\lambda_{p}+\lambda}\right) = \sum\limits_{p\in\mathbb{N}^{*}}\varphi_{p}\left(\sum\limits_{n\in\mathbb{N}^{*}\setminus\{p\}}\frac{a_{n}n^{-r}}{\lambda_{n}-\lambda_{p}+\lambda}\right).
\end{equation*}
Hence, it suffices to show that,
\begin{equation}
\label{eq:lem4goal}
\sum\limits_{p\in\mathbb{N}^{*}}p^{2r+2\varepsilon}\left|\sum\limits_{n\in\mathbb{N}^{*}\setminus\{p\}}\frac{a_{n}n^{-r}}{\lambda_{n}-\lambda_{p}+\lambda}\right|^{2}
\leq C\sum\limits_{n\in\mathbb{N}^{*}}|a_{n}|^{2}.
\end{equation}
Let us look at the left-hand side. \textcolor{black}{Recall that, since $\mathcal{A}$ is skew-adjoint,  the eigenvalues $(\lambda_{n})_{n}$ are purely imaginary while $\lambda\in\mathbb{R}_{+}\setminus\{0\}$. Therefore
\begin{equation}
\label{eq:imaginary}
    |\lambda_{n}-\lambda_{p}| \leq |\lambda_{n}-\lambda_{p}+\lambda| \leq |\lambda_{n}-\lambda_{p}|+\lambda.
\end{equation}
Thus, using a Cauchy-Schwarz inequality, we} have,
\begin{equation*}
\begin{split}
\sum\limits_{p\in\mathbb{N}^{*}}& p^{2r+2\varepsilon}\left|\sum\limits_{n\in\mathbb{N}^{*}\setminus\{p\}}\frac{a_{n}n^{-r}}{\lambda_{n}-\lambda_{p}+\lambda}\right|^{2}\\
&\leq \sum\limits_{p\in\mathbb{N}^{*}}p^{2r+2\varepsilon}\left(\sum\limits_{n\in\mathbb{N}^{*}\setminus\{p\}}\frac{|a_{n}|^{2}n^{-2r+\textcolor{black}{1-\alpha}+2\varepsilon}}{|\lambda_{n}-\lambda_{p}|}\right)\left(\sum\limits_{n\in\mathbb{N}^{*}\setminus\{p\}}\frac{n^{\textcolor{black}{(\alpha-1)}-2\varepsilon}}{|\lambda_{n}-\lambda_{p}|}\right).
\end{split}
\end{equation*}
Then, using Lemma \ref{lem:tech3}, Fubini theorem (since all terms are nonnegative), and again \textcolor{black}{\eqref{eq:imaginary} and} Lemma \ref{lem:tech3} (since $2r<\textcolor{black}{\alpha-1}$) we get, {\color{black}for $\varepsilon\in (0, 2\alpha- 2)$,}
\begin{equation}
\label{eq:complem4}
\begin{split}
\;\;\; \sum\limits_{p\in\mathbb{N}^{*}} & p^{2r+2\varepsilon}\left|\sum\limits_{n\in\mathbb{N}^{*}\setminus\{p\}}\frac{a_{n}n^{-r}}{\lambda_{n}-\lambda_{p}+\lambda}\right|^{2} \\
&\lesssim \sum\limits_{p\in\mathbb{N}^{*}}p^{2r+2\varepsilon}\left(\sum\limits_{n\in\mathbb{N}^{*}\setminus\{p\}}\frac{|a_{n}|^{2}n^{-2r-\textcolor{black}{(\alpha-1)}+2\varepsilon}}{|\lambda_{n}-\lambda_{p}|}\right)\left(\textcolor{black}{p^{-2\varepsilon}\log(p)+p^{-\alpha}}\right)\\
    &\lesssim \sum\limits_{n\in\mathbb{N}^{*}}|a_{n}|^{2}n^{-2r-\textcolor{black}{(\alpha-1)}+2\varepsilon}\sum\limits_{p\in\mathbb{N}^{*}\setminus\{n\}}\frac{p^{2r+2\varepsilon}\left(p^{-2\varepsilon}\log(p)+p^{-\textcolor{black}{\alpha}}\right)}{|\lambda_{n}-\lambda_{p}|}\\
    &\lesssim \sum\limits_{n\in\mathbb{N}^{*}}|a_{n}|^{2}n^{-2r-\textcolor{black}{(\alpha-1)}+2\varepsilon}\sum\limits_{p\in\mathbb{N}^{*}\setminus\{n\}}\frac{p^{2r+2\varepsilon}\left(p^{-(3/2)\varepsilon}+p^{-\textcolor{black}{\alpha}}\right)}{|\lambda_{n}-\lambda_{p}|}\\
    &\lesssim \sum\limits_{n\in\mathbb{N}^{*}}|a_{n}|^{2}n^{-2r-\textcolor{black}{(\alpha-1)}+2\varepsilon}\left(n^{\textcolor{black}{(1-\alpha)}+2r+\varepsilon/2}\log(n)+n^{-\textcolor{black}{(2\alpha-1)}+2r+2\varepsilon}\log(n)+n^{-\textcolor{black}{\alpha}}\right)\\
    &\lesssim \sum\limits_{n\in\mathbb{N}^{*}}|a_{n}|^{2}(n^{-\textcolor{black}{2(\alpha-1)}+3\varepsilon}+n^{-\textcolor{black}{(3\alpha-2)}+5\varepsilon}+n^{\textcolor{black}{(1-2\alpha)}-2r+2\varepsilon}).
    \end{split}
\end{equation}
Note that the limiting term is the last one, and since $r\in(\textcolor{black}{1/2-\alpha},0]$, we can choose $\varepsilon$ depending only on $r$ such that $r+\textcolor{black}{(\alpha-1/2)}-\varepsilon\geq 0$ and such that we have $(n^{-\textcolor{black}{2(\alpha-1)}+3\varepsilon}+n^{-\textcolor{black}{(3\alpha-2)}/2+5\varepsilon}+n^{\textcolor{black}{(1-2\alpha)}-2r+2\varepsilon})\leq 3$ (choose for instance $\varepsilon =\min(r+\textcolor{black}{(\alpha-1/2)},3/10)$). This means that, 
\begin{equation*}
\sum\limits_{p\in\mathbb{N}^{*}}p^{2r+2\varepsilon}\left|\sum\limits_{n\in\mathbb{N}^{*}\setminus\{p\}}\frac{a_{n}n^{-r}}{\lambda_{n}-\lambda_{p}+\lambda}\right|^{2}\lesssim \sum\limits_{n\in\mathbb{N}^{*}}|a_{n}|^{2},
\end{equation*}
and this ends the proof of Lemma \ref{lem4} in the case $r\in(\textcolor{black}{1/2-\alpha},0]$. Note that this could also work for $r\in(0,\textcolor{black}{(\alpha-1)/2})$. However, for $r \in [\textcolor{black}{(\alpha-1)/2}, \alpha-1/2)$, Lemma \ref{lem:tech3} cannot be used to get the fifth line of \eqref{eq:complem4}. Thus, we choose to treat the symmetrical cases $r\in(\textcolor{black}{1/2-\alpha},0]$ and $r\in(0,\textcolor{black}{\alpha-1/2})$.\\

Let us now assume that $r\in(0,\textcolor{black}{\alpha-1/2})$. As in the previous case, it suffices to show \eqref{eq:lem4goal}. Let us apply again \textcolor{black}{\eqref{eq:imaginary} and} Cauchy-Schwarz inequality on the left-hand side of \eqref{eq:lem4goal}, but in a slightly different way. We have \begin{equation*}
\sum\limits_{p\in\mathbb{N}^{*}}p^{2r+2\varepsilon}\left|\sum\limits_{n\in\mathbb{N}^{*}\setminus\{p\}}\frac{a_{n}n^{-r}}{\lambda_{n}-\lambda_{p}+\lambda}\right|^{2}\leq \sum\limits_{p\in\mathbb{N}^{*}}p^{2r+2\varepsilon}\left(\sum\limits_{n\in\mathbb{N}^{*}\setminus\{p\}}\frac{|a_{n}|^{2}}{|\lambda_{n}-\lambda_{p}|}\right)\left(\sum\limits_{n\in\mathbb{N}^{*}\setminus\{p\}}\frac{n^{-2r}}{|\lambda_{n}-\lambda_{p}|}\right).
\end{equation*}
Using again Lemma \ref{lem:tech3} since $-2r<\textcolor{black}{(\alpha-1)}$, then Fubini theorem, and then again Lemma \ref{lem:tech3} by choosing $\varepsilon>0$ such that {\color{black}$3\varepsilon<2(\alpha-1)$} and $2r-\textcolor{black}{\alpha}+2\varepsilon<\textcolor{black}{(\alpha-1)}$ (which always exists since $2r<\textcolor{black}{2\alpha-1}$), we have
\begin{equation*}
\begin{split}
    \;\;\;\; \sum\limits_{p\in\mathbb{N}^{*}}& p^{2r+2\varepsilon}\left|\sum\limits_{n\in\mathbb{N}^{*}\setminus\{p\}}\frac{a_{n}n^{-r}}{\lambda_{n}-\lambda_{p}+\lambda}\right|^{2} \\
    &\lesssim
    \sum\limits_{p\in\mathbb{N}^{*}}p^{2r+2\varepsilon}\left(\sum\limits_{n\in\mathbb{N}^{*}\setminus\{p\}}\frac{|a_{n}|^{2}}{|\lambda_{n}-\lambda_{p}|}\right)(p^{-\textcolor{black}{(\alpha-1)}-2r}\log(p)+p^{-\textcolor{black}{\alpha}})\\
    &\lesssim
    \sum\limits_{p\in\mathbb{N}^{*}}\left(\sum\limits_{n\in\mathbb{N}^{*}\setminus\{p\}}\frac{|a_{n}|^{2}}{|\lambda_{n}-\lambda_{p}|}\right)(p^{-\textcolor{black}{(\alpha-1)}+3\varepsilon}+p^{-\textcolor{black}{\alpha}+2r+2\varepsilon})\\
    &\lesssim
    \sum\limits_{n\in\mathbb{N}^{*}}|a_{n}|^{2}\left(\sum\limits_{p\in\mathbb{N}^{*}\setminus\{n\}}\frac{p^{-\textcolor{black}{(\alpha-1)}+3\varepsilon}+p^{-\textcolor{black}{\alpha}+2r+2\varepsilon}}{|\lambda_{n}-\lambda_{p}|}\right)\\
    &\lesssim
    \sum\limits_{n\in\mathbb{N}^{*}}|a_{n}|^{2}\left( n^{-2\textcolor{black}{(\alpha-1)}+3\varepsilon}\log(n)+n^{-\textcolor{black}{\alpha}}+n^{-\textcolor{black}{(2\alpha-1)}+2r+2\varepsilon}\log(n)\right)\\
    &\lesssim
    \sum\limits_{n\in\mathbb{N}^{*}}|a_{n}|^{2}\left( n^{-2\textcolor{black}{(\alpha-1)}+4\varepsilon}+n^{-\textcolor{black}{\alpha}}+n^{-\textcolor{black}{(2\alpha-1)}+2r+3\varepsilon}\right).
\end{split}
\end{equation*}
Then by choosing $\varepsilon$ such that $\varepsilon<\textcolor{black}{(\alpha-1)/2}$ and \textcolor{black}{$2r-\textcolor{black}{(2\alpha-1)}+3\varepsilon<0$ (which is possible since \textcolor{black}{$r\in(0,\alpha-1/2)$})},
there exists a constant $C>0$ (depending only on $r$) such that, 
\begin{equation*}
        \sum\limits_{p\in\mathbb{N}^{*}}p^{2r+2\varepsilon}\left|\sum\limits_{n\in\mathbb{N}^{*}\setminus\{p\}}\frac{a_{n}n^{-r}}{\lambda_{n}-\lambda_{p}+\lambda}\right|^{2}\leq C\sum\limits_{n\in\mathbb{N}^{*}}|a_{n}|^{2}.
\end{equation*}
{\color{black}Recalling the compact embedding property (see Section \ref{subsubsec-Simple-eigenv-case}) of the $\Hdiv^s$ spaces, this ends the proof of Lemma \ref{lem4}.}
\end{proof}

\subsection{Step (2): a Riesz basis for $H$}\label{subsec:step2}
In this section we prove the existence of a Riesz basis.
\begin{prop}
\label{qn:Riesz}
The family $(q_{n})_{n\in\mathbb{N}^{*}}$ is a Riesz basis of $\textcolor{black}{H}$.
\end{prop}
Showing this amounts to showing that $S${\color{black}, defined by \eqref{def-S},} is an isomorphism from $\textcolor{black}{H}$ to $\textcolor{black}{H}$. Since we know from Proposition \ref{S:fredholm} that $S$ is a Fredholm operator (of index 0), it suffices to show that {\color{black} $ker(S)=\{0\}$. } This is equivalent to say that for any $(a_{n})_{n\in\mathbb{N}^{*}}\in l^{2}$, such that,
\begin{equation}
\label{omega-indep0}
\sum\limits_{n\in\mathbb{N}^{*}}a_{n}q_{n} =0,
\end{equation}
one has $a_{n} =0$, for all $n\in\mathbb{N}^{*}$. In other words this is equivalent to show that $(q_n)_{n\in\mathbb{N}^*}$ is $\omega$-independent in $\textcolor{black}{H}$ (see Definition \ref{def-riesz-bas}). Notice that, 
\begin{equation*}
    q_n= \sum_{p}\frac{\varphi_p}{\lambda_n- \lambda_p+\lambda} \; \textrm{ and }  \;  \overline{q_n}= \sum_{p}\frac{{\color{black}\overline{\varphi_p}}}{\lambda_p- \lambda_n+\lambda}.
\end{equation*}
Thus we have the following Lemma:
\begin{lem}\label{lem:L2infqnbarqn}
The sequences $(q_n)_{n\in\mathbb{N}^{*}}$ and $(\overline{q_n})_{n\in\mathbb{N}^{*}}$ satisfy the following:
\begin{itemize}
    \item[(i)] $(q_n)_{n\in\mathbb{N}^{*}}$  \textcolor{black}{is} $\omega$-independent in $\textcolor{black}{H}$ or $\textcolor{black}{H}$-dense.
     \item[(ii)] $(\overline{q_n})_{n\in\mathbb{N}^{*}}$ \textcolor{black}{is} $\omega$-independent in $\textcolor{black}{H}$ or $\textcolor{black}{H}$-dense.
    \item[(iii)] $(q_n)_{n\in\mathbb{N}^{*}}$ is $\omega$-independent in $\textcolor{black}{H}$ $\Longleftrightarrow$  $(\overline{q_n})_{n\in\mathbb{N}^{*}}$ is $\omega$-independent in $\textcolor{black}{H}$.
     \item[(iv)] $(q_n)_{n\in\mathbb{N}^{*}}$ is  $\textcolor{black}{H}$-dense $\Longleftrightarrow$  $(\overline{q_n})_{n\in\mathbb{N}^{*}}$ is $\textcolor{black}{H}$-dense.
      \item[(v)] $(q_n)_{n\in\mathbb{N}^{*}}$ is  $\textcolor{black}{H}$-dense $\Longleftrightarrow$  $(\overline{q_n})_{n\in\mathbb{N}^{*}}$ is $\omega$-independent in $\textcolor{black}{H}$.
       \item[(vi)] $(\overline{q_n})_{n\in\mathbb{N}^{*}}$ is  $\textcolor{black}{H}$-dense $\Longleftrightarrow$  $(q_n)_{n\in\mathbb{N}^{*}}$ is $\omega$-independent in $\textcolor{black}{H}$.
\end{itemize}
Consequently, we know that $\textcolor{black}{(q_n)}_{n\in \N^*}$ ($resp.$ $\textcolor{black}{(\overline{q_n})}_{n\in \N^*}$) is both $\omega$-independent in $\textcolor{black}{H}$ and $\textcolor{black}{H}$\textcolor{black}{-}dense. 
\end{lem}
\begin{proof}[Proof of Proposition \ref{qn:Riesz}]
The proof of Proposition \ref{qn:Riesz} is equivalent to prove \eqref{omega-indep0} implies $a_n=0, \forall n\in \N^*$. But since from Lemma \ref{lem:L2infqnbarqn}, $(q_n)_{n\in\mathbb{N}^{*}}$ is $\omega$-independent in $\textcolor{black}{H}$, we conclude directly that $a_n=0, \forall n\in \N^*$, hence the proof. 
\end{proof}

Hence, it remains to prove Lemma \ref{lem:L2infqnbarqn}.

\begin{proof}[Proof of Lemma \ref{lem:L2infqnbarqn}]
The relations $(iii)$ and $(iv)$ are direct consequences of  the conjugacy of $q_n$ and $\overline{q_n}$. We only focus on the proof of $(i)$ and $(vi)$, as $(ii)$ and $(v)$ can be treated similarly. \\

The proof of $(vi)$ is further separated by two parts:
\begin{align*}
   \textrm{ $(\overline{q_n})_{n\in\mathbb{N}^{*}}$ is  not $\textcolor{black}{H}$-dense $\Longrightarrow$  $(q_n)_{n\in\mathbb{N}^{*}}$ is not $\omega$-independent in $\textcolor{black}{H}$, } \\
    \textrm{ $(\overline{q_n})_{n\in\mathbb{N}^{*}}$ is not $\textcolor{black}{H}$-dense $\Longleftarrow$  $(q_n)_{n\in\mathbb{N}^{*}}$ is not $\omega$-independent in $\textcolor{black}{H}$. }
\end{align*}

On the one hand, suppose that $(\overline{q_n})$ is not  $\textcolor{black}{H}$-dense, then there exists some {\it nontrivial} function $\overline a=\sum_{p\in\N^*} {\color{black}\overline{a_p} \overline{\varphi_p}}\in \textcolor{black}{H}$, namely $(\overline{a_p})_{p\in\mathbb{N}^{*}}\in l^2$, such that 
\begin{gather*}
    \langle\overline{q_n}, \overline a\rangle_{\textcolor{black}{H}}= 0, \;  \forall n\in \mathbb{N}^{*},
\end{gather*}
which is equivalent to 
\begin{equation*}
    \sum_{p\in \N^*}  \frac{ a_p}{\lambda_p- \lambda_n+ \lambda}=0, \;  \forall n\in \mathbb{N}^{*}.
\end{equation*}
Thus 
\begin{equation*}
    \sum_{p\in \N^*}  a_p q_p= \sum_{p\in\mathbb{N}^{*}} a_p\sum_{m\in\mathbb{N}^{*}} \frac{\varphi_m}{\lambda_p- \lambda_m+ \lambda}=  \sum_{m\in\mathbb{N}^{*}}\varphi_m\sum_{p\in\mathbb{N}^{*}} \frac{ a_p}{\lambda_p- \lambda_m+ \lambda}=0,
\end{equation*}
and consequently $(q_p)_{p\in\mathbb{N^{*}}}$ is not $\omega$-independent in $\textcolor{black}{H}$.\\ 

 On the other hand, suppose that  $(q_n)_{n\in\mathbb{N}^{*}}$ is not $\omega$-independent in $\textcolor{black}{H}$, then there exists some {\it nontrivial} sequence $(a_p)_{p\in\mathbb{N}^{*}}\in l^2$, namely $a= \sum_p a_p\varphi_p\in \textcolor{black}{H}$, such that 
 \begin{equation*}
    0= \sum_{p\in \N^*}  a_p q_p= \sum_{p\in\mathbb{N}^{*}} a_p\sum_{m\in\mathbb{N}^{*}} \frac{\varphi_m}{\lambda_p- \lambda_m+ \lambda}=  \sum_{m\in\mathbb{N}^{*}}\varphi_m\sum_{p\in\mathbb{N}^{*}} \frac{ a_p}{\lambda_p- \lambda_m+ \lambda}.
\end{equation*}
Since $(\varphi_{m})_{m\in\mathbb{N}^{*}}$ is a basis of $\textcolor{black}{H}$, this implies that 
\begin{equation*}
    \sum_{p} \frac{ a_p}{\lambda_p- \lambda_n+ \lambda}= 0, \; \forall n\in \N^*,
\end{equation*}
which is equivalent to
\begin{gather*}
    \left(\overline{q_n}, \overline a\right)_{\textcolor{black}{H}}= 0, \;  \forall n\in \mathbb{N}^{*}.
\end{gather*}
Hence the sequence $(\overline{q_n})_{n\in\mathbb{N}^*}$ is not $\textcolor{black}{H}$-dense. This ends the proof of the property $(vi)$. \\

Finally, we turn to the property $(i)$. The proof of this point is now classical (see for instance \cite{Gagnon-Hayat-Xiang-Zhang}) and we put it in Appendix \ref{sec:app:lemme410} for readers' convenience. We however underline that the proof in Appendix \ref{sec:app:lemme410} does not require that the family $(q_n)_{n\in\mathbb{N}^*}$ is quadratically close.
\end{proof}

\subsection{Step (3): extending the Riesz basis property to a sharp range of Sobolev spaces}\label{subsec:step3}
We are now going to use Proposition \ref{qn:Riesz} as well as Proposition \ref{S:fredholm} to show the following.
\begin{prop}
\label{qnr:Riesz}
For any $r\in(\textcolor{black}{1/2-\alpha,\alpha-1/2})$, the family $(n^{-r}q_{n})_{n\in\mathbb{N}^{*}}$ is a Riesz basis of $\Hdiv^{r}$.
\end{prop}
\begin{proof}[Proof of Proposition \ref{qnr:Riesz}]
 First of all, note, given the definition of $S$ in  \eqref{def-S}, that showing Proposition \ref{qnr:Riesz} is equivalent to showing that $S$ is an isomorphism from $\Hdiv^{r}$ to itself. Since we know, from Proposition \ref{S:fredholm}, that $S$ is a Fredholm operator (of index 0) from $\Hdiv^{r}$ to itself, it is enough to show that $ker(S)=\{0\}$ (where now $ker(S)$ is now a subset of $\Hdiv^{r}$). 
As previously, this is equivalent to show that $(n^{-r}q_{n})_{n\in\mathbb{N}^{*}}$ is $\omega$-independent in $\Hdiv^{r}$ in the sense of Definition \ref{def-riesz-bas}, $e. g.$ for any $(a_{n})_{n\in\mathbb{N}^{*}}\in l^{2}$, such that
\begin{equation}
\label{omega-indepr}
\sum\limits_{n\in\mathbb{N}^{*}}a_{n} n^{-r} q_{n} =0,
\end{equation}
one has $a_n=0$, for all $n\in \N^*$.
\\

The case $r=0$ is treated in Proposition \ref{qn:Riesz}, so we now assume $r\neq0$. We consider two different cases:
\begin{itemize}
    \item \textbf{Case (1) $r>0$.} Let $(a_{n})_{n\in\mathbb{N}^{*}}\in l^{2}$, and assume that \eqref{omega-indepr} holds, we can set $c_{n} = (a_{n}n^{-r})_{n\in\mathbb{N}^{*}}\in l^{2}$, and \eqref{omega-indepr} becomes
    \begin{equation}
    \label{eq:qnindepc}
        \sum\limits_{n\in\mathbb{N}^{*}}c_{n}q_{n} =0.
    \end{equation}
    Since $(q_{n})_{n\in\mathbb{N}^{*}}$ is a Riesz basis from Proposition \ref{qn:Riesz}, we have that $(q_{n})_{n\in\mathbb{N}^{*}}$ is $\omega$-independent in $\textcolor{black}{H}$ (see in particular Lemma \ref{lem:L2infqnbarqn}) and therefore \eqref{eq:qnindepc} implies that
    \begin{equation*}
        c_{n} = 0,\quad \forall n\in\mathbb{N}^{*},
    \end{equation*}
    which implies that $a_{n} = 0$ for all $n\in\mathbb{N}^{*}$. Thus $(n^{-r}q_{n})_{n\in\mathbb{N}^{*}}$ is $\omega$-independent, hence  $ker(S)=\{0\}$ (in $\Hdiv^{r}$) and this ends the proof.
    \begin{remark}
    The argument above is summarized by the following: if $r>0$, $ker(S)$ seen as a subset of $\Hdiv^{r}$ is included in $ker(S)$ seen as a subset of $\textcolor{black}{H}$, which is $\{0\}$ from Proposition \ref{qn:Riesz}.
    \end{remark}
    
    \item \textbf{Case (2) $r<0$}. This is the more {\color{black}difficult} 
    case as $ker(S)$, seen as a subset of $\Hdiv^{r}$, is not anymore included in $ker(S)$ seen as a subset of $\textcolor{black}{H}$ (but rather the opposite holds). To show this, we proceed by contradiction and use a dual argument between $\omega$-independence in $\Hdiv^{r}$ and density in $\Hdiv^{-r}$. \\ Let's assume by contradiction that $(n^{-r}q_{n})_{\omega\in\mathbb{N}^{*}}$ is not $\omega$-independent in $\Hdiv^r$. Then there exists some nontrivial  $(a_{n})_{n\in\mathbb{N}^{*}}\in l^{2}$  such that \eqref{omega-indepr} holds.
    Projecting on $(m^{-r}\varphi_{m})_{m\in\mathbb{N}^{*}}$, we have
    \begin{equation}
    \label{eq:scalarp}
    \begin{split}
        0 &= \left(\sum\limits_{n\in\mathbb{N}^{*}}a_{n}n^{-r}\left(\sum\limits_{p\in\mathbb{N^{*}}}\frac{p^{-r}p^{r}\varphi_{p}}{\lambda_{n}-\lambda_{p}+\lambda}\right),  m^{-r}\varphi_{m}\right)_{\Hdiv^{r}}\\
         &= \sum\limits_{n\in\mathbb{N}^{*}}\frac{a_{n}n^{-r}m^{r}}{\lambda_{n}-\lambda_{m}+\lambda},\;\;\forall\;m\in\mathbb{N}^{*}.
        \end{split}
    \end{equation}
    Now, let us set $h = \sum\limits_{n\in\mathbb{N}^{*}} a_{n} n^{r}\varphi_{n}$ which belongs to $\Hdiv^{-r}$ since $(a_{n})_{n\in\mathbb{N}^{*}}\in l^{2}$. From assumption, $a_{n}$ is not identically 0 thus $h\neq 0$ since $n^{r}\varphi_{n}$ is a basis of $\Hdiv^{-r}$. Since $r<0$, $(-r)>0$ and therefore, from Case (1), $(n^{r}q_{n})_{n\in\mathbb{N}^{*}}$ is a Riesz basis of $\Hdiv^{-r}$. In particular, this implies that $(n^{r}q_{n})_{n\in\mathbb{N}^{*}}$ is a dense family in $\Hdiv^{-r}$ and there exists $m_{0}\in\mathbb{N}^{*}$ such that 
    \begin{equation*}
        \langle h, m_{0}^{r}q_{m_{0}}\rangle_{\Hdiv^{-r}}\neq 0,
    \end{equation*}
    Expending the expression of $h$ and $q_{m_{0}}$, using the fact that the $\lambda_{i}$ are imaginary, this means
    \begin{equation*}
    \sum\limits_{n\in\mathbb{N}^{*}} a_{n}m_{0}^{r}\frac{n^{-r}}{\lambda_{n}-\lambda_{m_{0}}+\lambda} \neq 0,
    \end{equation*}
    which is in contradiction with \eqref{eq:scalarp}. Thus $(a_{n})_{n\in\mathbb{N}^{*}}$ is identically 0 and $(n^{-r}q_{n})_{n\in\mathbb{N}^{*}}$ is $\omega$-independant in $\Hdiv^r$, which means that $ker(S)$ (seen as a subset of $\Hdiv^{r}$) is reduced to $\{0\}$ and $S$ is an isomorphism from $\Hdiv^{r}$ to $\Hdiv^{r}$. This ends the proof of Proposition \ref{qnr:Riesz}.
\end{itemize}
\end{proof}

\section{Construction of the isomorphism-feedback pair}
\label{sec:proof2}
In this section, we construct the isomorphism-feedback pair to end the proof of Proposition \ref{prop:th}. In the previous section, we have established some important properties, namely that specific families of functions form Riesz basis in appropriate spaces. Accordingly, this allowed us to define a general isomorphism, given by $S$, using the first backstepping equality in \eqref{op-eq00}. We now use the second one to propose a feedback law. Then, using the Riesz basis properties and the isomorphism $S$, we prove that the corresponding transformation $T$ is an isomorphism in an array of spaces.  This will end the proof of Proposition \ref{prop:th} and Theorem \ref{th1}. It will then remain to prove that the closed-loop system given by the feedback is well-posed to have Corollary \ref{coro:expostab}.

\subsection{Step (4): construction and basic properties of the stabilizing feedback}
Let us now provide a candidate for $(K_{n})_{n\in\mathbb{N}^{*}}$. Recall that our goal is to have $(K_{n})_{n\in\mathbb{N}^{*}}$ such that the $TB=B$ condition holds in some sense. Expressing this condition on the basis $\varphi_{n}$, it becomes
\begin{equation}
\label{eq:tbbf0}
    \sum\limits_{n\in\mathbb{N}^{*}}(-K_{n}b_{n})\tau q_{n}=\sum\limits_{n\in\mathbb{N}^{*}}b_{n}\varphi_{n}.
\end{equation}
Note that the right-hand side belongs to $\Hdiv^{-\textcolor{black}{\alpha/2}}$ and in fact, given \eqref{condB}, it even belongs to $\Hdiv^{-\textcolor{black}{\frac{1}{2}}-\varepsilon}$ for any {\color{black}$\varepsilon >0$}. From Proposition \ref{qnr:Riesz} and since $\tau : \Hdiv^r \rightarrow \Hdiv^r$ is an isomorphism for any $r\in \R$, we have that  $(n^{\textcolor{black}{-r}}\tau q_{n})_{n\in\mathbb{N}^{*}}$ is a Riesz basis of $\Hdiv^{\textcolor{black}{r}}$ for {\color{black} $r\in (\frac{1}{2}-\alpha,\alpha-\frac{1}{2})$. Since $\alpha>1$, we are able to give a meaning to \eqref{eq:tbbf0} in $\Hdiv^r$ for $r\in (\frac{1}{2}-\alpha,-\frac{1}{2})$, or equivalently in $\Hdiv^{-\frac{1}{2}-\varepsilon}$ for $\varepsilon \in (0,\alpha-1)$, thanks to Proposition \ref{qnr:Riesz}.} \textcolor{black}{This leads to the following lemma:}
\begin{lem}
\label{lem:defk}
There exists a unique sequence $(-K_{n})_{n\in\mathbb{N}^{*}}$ such that for any {\color{black}$\varepsilon\in(0,\alpha-1)$} the condition \eqref{eq:tbbf0} holds in $\Hdiv^{-\textcolor{black}{\frac{1}{2}}-\varepsilon}$ and,
\begin{equation*}
((-K_{n}b_{n})n^{-\textcolor{black}{\frac{1}{2}}-\varepsilon})_{n\in\mathbb{N}^{*}}\in l^{2}.   
\end{equation*}
\end{lem}

Let us set, 
\begin{equation}
\label{defkn}
    k_{n}:=-(K_{n}b_{n}+\lambda),
\end{equation}
The goal of this section will be to show the following lemma.
\begin{lem}
\label{lemkncomp}
The two following hold,
\begin{enumerate}
\item The sequence $(K_{n})_{n\in\mathbb{N}^{*}}$ defined by Lemma \ref{lem:defk} is uniformly bounded.
\item For any $r\in(\textcolor{black}{\frac{1}{2}-\alpha},\textcolor{black}{\alpha-\frac{1}{2}})$, the operator $k$ defined by,
\begin{equation}
\label{def-op-k}
    k:\varphi_{n}\mapsto k_{n}\tau q_{n},
\end{equation}
is continuous from $\Hdiv^{r}$ to $\Hdiv^{r+\varepsilon}$ {\color{black} for $\varepsilon \in [0,r_0)$ where $r_0=\min(\alpha-\frac{1}{2}-r,\alpha-1)$}. In particular, this operator is a compact operator from $\Hdiv^{r}$ to $\Hdiv^{r}$.
\end{enumerate}
\end{lem} 

\begin{proof}
From the definition of $K_{n}$ given by Lemma \ref{lem:defk} and the $TB=B$ condition given by \eqref{eq:tbbf0}, we have that,
\begin{equation} 
\label{tbb2}
\sum\limits_{n\in\mathbb{N}^{*}}(-K_{n}b_{n})\tau q_{n}= \sum\limits_{n\in\mathbb{N}^{*}}b_{n}\varphi_{n},
\end{equation}
holds in $\Hdiv^{-\textcolor{black}{1/2-\varepsilon}}$ for any \textcolor{black}{$\varepsilon\in(0,\alpha-1)$}. {\color{black}Moreover, from Lemma \ref{lem:defk} and the definition \eqref{defkn} of $k_n$, we have, since $ (n^{-\frac{1}{2}-\varepsilon}\textcolor{black}{\lambda}) \in \ell^2$ for any $\varepsilon>0$, that 
\begin{equation}\label{knnepss0}
(k_n n^{-r})_{n\in \N^*} \in \ell^2, \quad \textrm{ for any }  r\in \left(\frac{1}{2},\alpha-\frac{1}{2}\right).
\end{equation}}
Hence, using the definition of $k_{n}$ and $\tau q_{n}$, we get in {\color{black}the $\Hdiv^{-1/2-\varepsilon}$ sense, for any \textcolor{black}{$\varepsilon\in(0,\alpha-1)$}, that,} 
\begin{equation*}
\sum\limits_{n\in\mathbb{N}^{*}}\lambda \sum\limits_{p\in\mathbb{N}^{*}}\frac{b_{p}\varphi_{p}}{\lambda_{n}-\lambda_{p}+\lambda}+\sum\limits_{n\in\mathbb{N}^{*}}k_{n} \tau q_{n} = \sum\limits_{n\in\mathbb{N}^{*}}b_{n}\varphi_{n},
\end{equation*}
which gives,
\begin{equation*}
\sum\limits_{n\in\mathbb{N}^{*}}\lambda \sum\limits_{p\in\mathbb{N}^{*}\setminus\{n\}}\frac{b_{p}\varphi_{p}}{\lambda_{n}-\lambda_{p}+\lambda}+\sum\limits_{n\in\mathbb{N}^{*}}b_{n}\varphi_{n}+\sum\limits_{n\in\mathbb{N}^{*}}k_{n} \tau q_{n} = \sum\limits_{n\in\mathbb{N}^{*}}b_{n}\varphi_{n}.
\end{equation*}
Hence,
\begin{equation}
\label{tbbkn}
\sum\limits_{n\in\mathbb{N}^{*}}\lambda\sum\limits_{p\in\mathbb{N}^{*}\setminus\{n\}}\frac{b_{p}\varphi_{p}}{\lambda_{n}-\lambda_{p}+\lambda}= -\sum\limits_{n\in\mathbb{N}^{*}}k_{n}\tau q_{n}.
\end{equation}
Let us remark that this equality a priori holds in $\Hdiv^{-\textcolor{black}{\frac{1}{2}}-\varepsilon}$ for any {\color{black}$\varepsilon\in(0,\alpha-1)$}. 
However, we removed from \eqref{eq:tbbf0} the most singular part and thus both terms of \eqref{tbbkn} are in fact more regular. {\color{black} Indeed, if $\alpha>3/2$, then the left-hand side of \eqref{tbbkn} belongs to $\Hdiv^\varepsilon$ with $\alpha - 3/2 > \varepsilon > 0$. \textcolor{black}{To show this, observe that} using Fubini's Theorem in $\Hdiv^{\textcolor{black}{-1/2-\varepsilon}}$, 
\begin{equation*}
\sum\limits_{n\in\mathbb{N}^{*}}\lambda\sum\limits_{p\in\mathbb{N}^{*}\setminus\{n\}}\frac{b_{p}\varphi_{p}}{\lambda_{n}-\lambda_{p}+\lambda} =
\sum\limits_{p\in\mathbb{N}^{*}}b_{p}\varphi_{p}\lambda\sum\limits_{n\in\mathbb{N}^{*}\setminus\{p\}}\frac{1}{\lambda_{n}-\lambda_{p}+\lambda},
\end{equation*}
and, \textcolor{black}{using \eqref{eq:imaginary}},
\begin{align*}
\left\|\sum\limits_{p\in\mathbb{N}^{*}}b_{p}\varphi_{p}\lambda\sum\limits_{n\in\mathbb{N}^{*}\setminus\{p\}}\frac{1}{\lambda_{n}-\lambda_{p}+\lambda} \right\|_{\Hdiv^{\varepsilon}}^2 & =
\sum_{p\in \N^*} p^{2\varepsilon}|b_p|^2 \lambda^2 \left| \sum_{n\in \N^*{\color{black}\setminus\{p\}}}  \dfrac{1}{\lambda_n-\lambda_p+\lambda} \right|^2 \\
& \lesssim \sum_{p\in \N^*} p^{2\varepsilon} \left| \sum_{n\in \N^*{\color{black}\setminus\{p\}}}  \dfrac{1}{\lambda_n-\lambda_p+\lambda} \right|^2 \\
& \leq \sum_{p\in \N^*} p^{2\varepsilon} \left( \sum_{n\in \N^*{\color{black}\setminus\{p\}}}  \dfrac{1}{|\lambda_n-\lambda_p|} \right)^2,
\end{align*}
where we used the fact that $b_p$ are uniformly bounded and $\{\lambda_n\}_{n\in \N^*}$ are imaginary while $\lambda$ is real. Then, applying Lemma \ref{lem:tech3} yields,
\[
\left\| \sum\limits_{n\in\mathbb{N}^{*}}\lambda\sum\limits_{p\in\mathbb{N}^{*}\setminus\{n\}}\frac{b_{p}\varphi_{p}}{\lambda_{n}-\lambda_{p}+\lambda} \right\|_{\Hdiv^{\varepsilon}}^2 \lesssim \sum_{p \in \N^*} p^{2\varepsilon + 2(1-\alpha)}\log^2(p),
\]
which converges for $\varepsilon \in (0,\alpha-\frac{3}{2})$ and $\alpha>3/2$. Therefore, going back to \eqref{tbbkn} and rewriting it as, 
\begin{equation}
\label{tbbkn2}
\sum\limits_{n\in\mathbb{N}^{*}}\lambda\sum\limits_{p\in\mathbb{N}^{*}\setminus\{n\}}\frac{b_{p}\varphi_{p}}{\lambda_{n}-\lambda_{p}+\lambda}= -\sum\limits_{n\in\mathbb{N}^{*}}(k_{n}n^{\varepsilon})(n^{-\varepsilon}\tau q_{n}).
\end{equation}
using that $(n^{-\textcolor{black}{\varepsilon}}\tau q_n)_{n\in \N^*}$ is a Riesz basis of $\Hdiv^{\textcolor{black}{\varepsilon}}$ for $\textcolor{black}{\varepsilon}\in (0,\alpha-3/2)$ \textcolor{black}{from Proposition \ref{qnr:Riesz}}, we deduce that,
\begin{equation}\label{knneps00}
(n^{\varepsilon}k_n)_{n\in \N^*} \in \ell^2, \textrm{ for any } \varepsilon \in \left(0,\alpha-\frac{3}{2}\right), \textrm{ with } \alpha > 3/2. 
\end{equation}
Moreover, by definition $K_nb_n = -(\lambda + k_n)$, and together with the uniform boundedness of $b_n$ and \eqref{knneps00}, we obtain, 
\[
(K_n)_{n\in \N^*} \in \ell^\infty, \textrm{ for any } \alpha>3/2. 
\]
}

Yet, if $3/2 \geq \alpha >1$ the gain of regularity in \eqref{tbbkn} is not sufficient to conclude the first point of Lemma \ref{lemkncomp}. Note that $\alpha=3/2$ is exactly the critical case.  
However, we can show that it belongs to $\Hdiv^{-\varepsilon}$ for any $\varepsilon>\textcolor{black}{3/2-\alpha}$. Indeed, a similar computation leads to,
\begin{equation}
\label{eq:calc1}
\begin{split}
\left\|\sum\limits_{n\in\mathbb{N}^{*}}\lambda\sum\limits_{p\in\mathbb{N}^{*}\setminus\{n\}}\frac{b_{p}\varphi_{p}}{\lambda_{n}-\lambda_{p}+\lambda}\right\|_{\Hdiv^{-\varepsilon}}^{2} &= \sum\limits_{p\in\mathbb{N}^{*}}p^{-2\varepsilon}|b_{p}|^{2}\lambda^{2}\left|\sum\limits_{n\in\mathbb{N}^{*}\setminus\{p\}}\frac{1}{\lambda_{n}-\lambda_{p}+\lambda}\right|^{2}\\
& \lesssim 
\sum\limits_{p\in\mathbb{N}^{*}}p^{-2\varepsilon}\left(\sum\limits_{n\in\mathbb{N}^{*}\setminus\{p\}}\frac{1}{|\lambda_{n}-\lambda_{p}|}\right)^{2} \\
& \lesssim
\sum\limits_{p\in\mathbb{N}^{*}}\textcolor{black}{p^{-2\varepsilon+2(1-\alpha)}\log^{2}(p)},
\end{split}
\end{equation}
which converges for any $\varepsilon>\textcolor{black}{3/2-\alpha}$. Hence, the left-hand side of \eqref{tbbkn} belongs to $\Hdiv^{-\varepsilon}$ for any $\varepsilon>\textcolor{black}{3/2-\alpha}$. From Proposition \ref{qnr:Riesz} and the fact that $\tau$ is an isomorphism from $\Hdiv^{r}$ to $\Hdiv^{r}$ for any \textcolor{black}{$r\in(1/2-\alpha,\alpha-1/2)$}, $(n^{\varepsilon}\tau q_{n})_{n\in\mathbb{N}^{*}}$ is a Riesz basis of $\Hdiv^{-\varepsilon}$ and therefore there exists a unique $(k_{n}n^{-\varepsilon})_{n\in\mathbb{N}^{*}}\in l^{2}$ such that \eqref{tbbkn2} holds. In particular, this shows that $(k_{n})_{n\in\mathbb{N}^{*}}$ defined in \eqref{defkn} satisfies,{\color{black}
\begin{equation}
\label{knnepsl2}
(k_{n}n^{-\varepsilon})_{n\in\mathbb{N}^{*}}\in l^{2}\text{ for any } \varepsilon\in \left(\frac{3}{2}-\alpha,\alpha - \frac{1}{2}\right), \textrm{ with } \alpha \in (1,3/2],
\end{equation}
and therefore, {\color{black} thanks to the definition of $k_n$},
\begin{equation}
\label{Knnepsl2}
(K_{n}n^{-\varepsilon})_{n\in\mathbb{N}^{*}}\in l^{\infty}\text{ for any } \varepsilon\in \left(\frac{3}{2}-\alpha,\alpha - \frac{1}{2}\right), \textrm{ with } \alpha \in (1,3/2]. 
\end{equation}
Comparing \eqref{knnepsl2} to \eqref{knnepss0}, we are able to see the $(\alpha-1)$ gain of regularity between \eqref{tbb2} and \eqref{tbbkn2}.
The core of the proof of Lemma \ref{lemkncomp} in the critical range $\alpha\in (1,3/2]$ is to {\color{black} iterate} this $(\alpha-1)$ gain of regularity. To do so, we derive an asymptotic analysis of $K_n$ by induction up to an order which depends on $\alpha$, and we gain at each order the $(\alpha-1)$ regularity, allowing us to conclude thanks to the fact that $\alpha>1$. Indeed, consider \eqref{tbb2} again and explicit the expression of $\tau q_{n}$, }
\begin{equation*}
\sum\limits_{n\in\mathbb{N}^{*}}(-K_{n}b_{n})\left(\frac{b_{n}\varphi_{n}}{\lambda}+\sum\limits_{p\in\mathbb{N}^{*}\setminus\{n\}}\frac{b_{p}\varphi_{p}}{\lambda_{n}-\lambda_{p}+\lambda}\right)=\sum\limits_{n\in\mathbb{N}^{*}}b_{n}\varphi_{n},
\end{equation*}
which gives, by decomposing $K_{n}$ in $\lambda$ and $\textcolor{black}{k_{n}}$ in the first term only,
\begin{equation*}
\sum\limits_{n\in\mathbb{N}^{*}}\frac{k_{n}b_{n}}{\lambda}\varphi_{n}+\left(\sum\limits_{n\in\mathbb{N}^{*}}(-K_{n}b_{n})\sum\limits_{p\in\mathbb{N}^{*}\setminus\{n\}}\frac{b_{p}\varphi_{p}}{\lambda_{n}-\lambda_{p}+\lambda}\right)=0.  
\end{equation*}
Using again Fubini theorem in $\Hdiv^{\textcolor{black}{-1/2-\varepsilon}}$ for $\varepsilon \in (0,\alpha-1)$ and identifying the coefficients along the orthonormal basis $(\varphi_{n})_{n\in\mathbb{N}^{*}}$, we obtain the following expression.
\begin{equation}\label{expressionkn}
    \frac{k_{n}b_{n}}{\lambda}= \sum\limits_{m\in\mathbb{N}^{*}\setminus\{n\}}K_{m}b_{m}\frac{b_{n}}{\lambda_{m}-\lambda_{n}+\lambda}.
\end{equation}
The asymptotic analysis is done in the following manner. First, denote $e^0_n= \lambda$ and $k^0_n= k_n, n\in \N^*$. In order to simplify the notation, we will assume \textcolor{black}{without loss of generality} $b_n=1, n\in \N^*$. {\color{black}Thus $-K_n= e^0_n+ k^0_n$ and the expression \eqref{expressionkn} becomes, 
\begin{equation}\label{expressionknsim}
    k_n= \lambda\sum\limits_{m\in\mathbb{N}^{*}\setminus\{n\}}\frac{K_{m}}{\lambda_{m}-\lambda_{n}+\lambda}=  -\lambda\sum\limits_{m\in\mathbb{N}^{*}\setminus\{n\}}\frac{e^0_m+ k^0_m}{\lambda_{m}-\lambda_{n}+\lambda}.
\end{equation}
We now decompose further $k_n= k_n^0= e^1_n+ k^1_n$ as follows, 
\begin{equation}
\label{def:kn1}
    e^1_n= -\lambda\sum\limits_{m\in\mathbb{N}^{*}\setminus\{n\}}\frac{e^0_m}{\lambda_{m}-\lambda_{n}+\lambda}, \; k^1_n= -\lambda\sum\limits_{m\in\mathbb{N}^{*}\setminus\{n\}}\frac{k^0_m}{\lambda_{m}-\lambda_{n}+\lambda}.
\end{equation}
From Lemma \ref{lem:tech3}, we observe that, 
\begin{equation*}
    |e^1_n|\lesssim n^{1-\alpha}\log n\lesssim 1.
\end{equation*}
We therefore focus on the regularity of $k_{n}^{1}$. Let us denote $s_{0} := 3/2-\alpha$. Then we have, \textcolor{black}{using \eqref{eq:imaginary}}
{\color{black}\begin{align}
    \|(n^{-r}k^1_n)_n\|_{l^2}^2&\lesssim \sum_{n\in \N^*} n^{-2r} \left(\sum\limits_{m\in\mathbb{N}^{*}\setminus\{n\}}\frac{k^0_m}{\lambda_{m}-\lambda_{n}+\lambda}\right)^2 \nonumber \\
    &\lesssim  \sum_{n\in \N^*}n^{-2r} \left(\sum\limits_{m\in\mathbb{N}^{*}\setminus\{n\}}\frac{|k^0_m|^2}{|\lambda_{m}-\lambda_{n}|}\right)\left(\sum\limits_{m\in\mathbb{N}^{*}\setminus\{n\}}\frac{1}{|\lambda_{m}-\lambda_{n}|}\right) \nonumber \\
      &\lesssim  \sum_{n\in \N^*}n^{-2r+1- \alpha} \log n  \left(\sum\limits_{m\in\mathbb{N}^{*}\setminus\{n\}}\frac{|k^0_m|^2}{|\lambda_{m}-\lambda_{n}|}\right) \nonumber \\
      &\lesssim  \sum_{m\in \N^*} |k^0_m|^2 \sum\limits_{n\in\mathbb{N}^{*}\setminus\{m\}} n^{-2r+1- \alpha} \log (n) \left(\frac{1}{|\lambda_{m}-\lambda_{n}|}\right) \nonumber \\
       &\lesssim  \sum_{m\in \N^*} |k^0_m|^2 m^{-2(r+(\alpha-1))} \log^2(m), \label{calc:kn1iteration}
\end{align}
which, from \eqref{knnepsl2}, converges if $r+(\alpha-1)\in (\frac{3}{2}-\alpha,\alpha-\frac{1}{2})$, that is,
\begin{equation}
\label{knnepsl22}
(n^{-\varepsilon}k^1_n)_{n\in\mathbb{N}^{*}}\in l^{2}\text{ for any } \varepsilon\in \left(\frac{5}{2}-2\alpha, \frac{1}{2}\right), \textrm{ with } \alpha \in (1,3/2].
\end{equation}
}
Define $s_1:= 5/2- 2\alpha $. Notice that if $s_1< 0$, namely $\alpha> 5/4$, then we can conclude that  $(k_n^1)_n n ^\delta$ belongs to $l^2$ with $\delta>0$. Therefore,  from the definition of $(k_n^{1})_n$ and the fact that $(e_{n}^{1})_n$ is uniformly bounded, we deduce that $(k_n)_n$ is uniformly bounded, which further yields similarly that $(K_n)_n$ is uniformly bounded.

Otherwise, we have $s_1\geq 0$ and  need to continue the iteration procedure to further decompose $k^1_n$. {\color{black} Indeed, injecting the definition of $k_m^0$ again into \eqref{expressionknsim}, we have
\begin{align}\nonumber 
k_n &= -\lambda\sum\limits_{m\in\mathbb{N}^{*}\setminus\{n\}}\frac{e^0_m+ k_m^0}{\lambda_{m}-\lambda_{n}+\lambda} \nonumber \\
=&  -\lambda\sum\limits_{m\in\mathbb{N}^{*}\setminus\{n\}}\left(\frac{e^0_m}{\lambda_{m}-\lambda_{n}+\lambda}\right) - \lambda\sum\limits_{m\in\mathbb{N}^{*}\setminus\{n\}} \left(\frac{k_m}{\lambda_{m}-\lambda_{n}+\lambda}\right) \nonumber \\
=&  -\lambda\sum\limits_{m\in\mathbb{N}^{*}\setminus\{n\}}\left(\frac{e^0_m}{\lambda_{m}-\lambda_{n}+\lambda}\right) - \lambda\sum\limits_{m\in\mathbb{N}^{*}\setminus\{n\}} \left(\frac{1}{\lambda_{m}-\lambda_{n}+\lambda} \left( -\lambda\sum\limits_{\textcolor{black}{p}\in\mathbb{N}^{*}\setminus\{\textcolor{black}{m}\}}\frac{e^0_{\textcolor{black}{p}}+ k_{\textcolor{black}{p}}^0}{\lambda_{\textcolor{black}{p}}-\lambda_{\textcolor{black}{m}}+\lambda} \right) \right) \nonumber \\
=&  -\lambda\sum\limits_{m\in\mathbb{N}^{*}\setminus\{n\}}\left(\frac{e^0_m}{\lambda_{m}-\lambda_{n}+\lambda}\right) - \lambda\sum\limits_{m\in\mathbb{N}^{*}\setminus\{n\}} \left(\frac{e_m^1}{\lambda_{m}-\lambda_{n}+\lambda}\right) - \lambda\sum\limits_{m\in\mathbb{N}^{*}\setminus\{n\}} \left(\frac{k_m^1}{\lambda_{m}-\lambda_{n}+\lambda}\right). \label{eq:defiteration}
\end{align}
}  
We therefore define $e^2_n$ and $k^2_n$ as,
\begin{equation}
    e^2_n= -\lambda\sum\limits_{m\in\mathbb{N}^{*}\setminus\{n\}}\frac{e^1_m}{\lambda_{m}-\lambda_{n}+\lambda}, \; k^2_n= -\lambda\sum\limits_{m\in\mathbb{N}^{*}\setminus\{n\}}\frac{k^1_m}{\lambda_{m}-\lambda_{n}+\lambda}.
\end{equation}
Since $k_{n} = e_{n}^{1}+k_{n}^{1}$, and using the definition of $e_{n}^{1}$ given by \eqref{def:kn1} together with \eqref{eq:defiteration}
we deduce that $k_n^1 = e_n^2 + k_n^2$. Again, it is straightforward to observe that, 
\begin{equation*}
    |e^2_n|\lesssim n^{1-\alpha}\log n\lesssim 1.
\end{equation*}
We further get estimates on $(k^2_n)$, \textcolor{black}{using again \eqref{eq:imaginary},}
\begin{align*}
    \|(n^{-r}k^2_n)_n\|_{l^2}^2&\lesssim \sum_{n\in \N^*} n^{-2r} \left(\sum\limits_{m\in\mathbb{N}^{*}\setminus\{n\}}\frac{k^1_m}{\lambda_{m}-\lambda_{n}+\lambda}\right)^2 \\
    &\lesssim  \sum_{n\in \N^*}n^{-2r} \left(\sum\limits_{m\in\mathbb{N}^{*}\setminus\{n\}}\frac{|k^1_m|^2}{|\lambda_{m}-\lambda_{n}|}\right)\left(\sum\limits_{m\in\mathbb{N}^{*}\setminus\{n\}}\frac{1}{|\lambda_{m}-\lambda_{n}|}\right) \\
      &\lesssim  \sum_{m\in \N^*} |k^1_m|^2 \sum\limits_{n\in\mathbb{N}^{*}\setminus\{m\}} n^{-2{\color{black}r}+1- \alpha} \log n  \left(\frac{1}{|\lambda_{m}-\lambda_{n}|}\right) \\
       &\lesssim  \sum_{m\in \N^*} |k^1_m|^2 m^{-2(r+(\alpha-1))}\log^2(m),
\end{align*}}
{\color{black}
which, from \eqref{knnepsl22}, converges if $r+(\alpha-1) \in (\frac{5}{2}-2\alpha,\frac{1}{2})$, that is, 
\begin{equation}\label{knnepsl23}
(n^{-\varepsilon}k^2_n)_{n\in\mathbb{N}^{*}}\in l^{2}\text{ for any } \varepsilon\in \left(\frac{7}{2}-3\alpha, \frac{3}{2}-\alpha \right), \textrm{ with } \alpha \in (1,3/2].
\end{equation}
}
{\color{black}
Define $s_2:=7/2- 3\alpha$. We deduce that if $s_2< 0$, namely $\alpha> 7/6$, then $k_n$ is uniformly bounded and so is $K_n$. Otherwise, $s_2\geq 0$ and we continue the iteration procedure.} {\color{black} From \eqref{eq:defiteration}, we easily obtain by induction that, $k_n^i=e_n^{i+1}+k_n^{i+1}$ for $i\in \N$ where,
\begin{align}
    &e^{i+1}_n:= -\lambda\sum\limits_{m\in\mathbb{N}^{*}\setminus\{n\}}\frac{e^i_m}{\lambda_{m}-\lambda_{n}+\lambda}, \label{eq:defenplus1}\\
     &k^{i+1}_n:= -\lambda\sum\limits_{m\in\mathbb{N}^{*}\setminus\{n\}}\frac{k^i_m}{\lambda_{m}-\lambda_{n}+\lambda}.\label{eq:defknplus1}
\end{align}
A straightforward computation show that, 
\begin{equation}\label{eq:decompKN}
-K_n= \left(\sum_{j=0}^{i} e^j_n\right)+ k^i_n, \quad \forall i \in \N.
\end{equation}
Moreover, by the definition \eqref{eq:defenplus1} of $e_n^i$, the simple bound $|e_n^{i-1}| \lesssim 1, i\in N^* $ and Lemma \ref{lem:tech3}, we have,
\begin{equation}\label{eq:estimEN}
|e^i_n|\lesssim n^{1-\alpha}\log n\lesssim 1, \quad \forall i \in \N^*.
\end{equation}
By the definition \eqref{eq:defknplus1} of $k_n^{i+1}$ \textcolor{black}{and \eqref{eq:imaginary}}, also have,
\begin{align*}
    \|(n^{-r}k^{i+1}_n)_n\|_{l^2}^2&\lesssim \sum_{n\in \N^*} n^{-2r} \left(\sum\limits_{m\in\mathbb{N}^{*}\setminus\{n\}}\frac{k^{i}_m}{\lambda_{m}-\lambda_{n}+\lambda}\right)^2 \\
    &\lesssim  \sum_{n\in \N^*}n^{-2r} \left(\sum\limits_{m\in\mathbb{N}^{*}\setminus\{n\}}\frac{|k^{i}_m|^2}{|\lambda_{m}-\lambda_{n}|}\right)\left(\sum\limits_{m\in\mathbb{N}^{*}\setminus\{n\}}\frac{1}{|\lambda_{m}-\lambda_{n}|}\right) \\
      &\lesssim  \sum_{m\in \N^*} |k^i_m|^2 \sum\limits_{n\in\mathbb{N}^{*}\setminus\{m\}} n^{-2{\color{black}r}+1- \alpha} \log n  \left(\frac{1}{|\lambda_{m}-\lambda_{n}|}\right) \\
       &\lesssim  \sum_{m\in \N^*} |k^i_m|^2 m^{-2(r+(\alpha-1))}\log^2(m),
\end{align*}}
which, by induction for $i\geq 2$ together with \eqref{knnepsl23}, allows to deduce,
\begin{equation}\label{knnepsl2inf}
(n^{-\varepsilon}k^i_n)_{n\in\mathbb{N}^{*}}\in l^{2}\text{ for any } \varepsilon\in \left(s_i, s_{i-2} \right), \textrm{ with } \alpha \in (1,3/2],
\end{equation}
with $s_i:= s_0+ (1-\alpha) i=1/2+(1-\alpha)(i+1), \,  \forall i\in \N$.
We are therefore able to conclude the first part of Lemma \ref{lemkncomp}. Indeed, for any $\alpha \in (1,3/2]$, there exists $M \in \N$ such that $s_M<0$. Hence, for such $M$, from \eqref{knnepsl2inf}, we deduce that $k^M_n \in \ell^\infty$ and, from \eqref{eq:decompKN} and \eqref{eq:estimEN}, that,
\[
|K_n| \lesssim M + \|k_n^M \|_{\ell^\infty},
\]
which implies $(K_{n})_{n\in\mathbb{N}^{*}}\in l^{\infty}$. This ends the proof of the first point of Lemma \ref{lemkncomp}.

Let us now study the operator $k$ given by \eqref{def-op-k}. Let $r\in(\textcolor{black}{1/2-\alpha,\alpha-1/2})$ {\color{black} and consider $\varepsilon \in (0,\alpha-\frac{1}{2}-r)$.}
Let $f\in \Hdiv^{r}$ and denote $f_{n} = \langle f, n^{-r}\varphi_{n}\rangle_{\Hdiv^{r}}$ such that $f = \sum\limits_{n\in\mathbb{N}^{*}}f_{n}n^{-r}\varphi_{n}$. We have, 
\begin{equation*}
    k (f) = \sum\limits_{n\in\mathbb{N}^{*}} f_{n}k_{n}n^{-r}\tau q_{n}
    =\sum\limits_{n\in\mathbb{N}^{*}} (f_{n}k_{n}n^{\varepsilon})(n^{-(r+\varepsilon)}\tau q_{n}),
\end{equation*}
where the last equality holds a priori in $\Hdiv^{r}$ since $k_{n}$ is bounded. In fact, we have {\color{black}
\begin{equation}
\label{kneps}
(k_{n}n^{\varepsilon})_{n\in\mathbb{N}^*}\in l^{\infty}, \; \text{ for any } \varepsilon \in [0,\alpha-1),
\end{equation}
Indeed, for $\alpha>3/2$, we use the decomposition $k_n=e_n^1+k_n^1$, the estimation  $|e_n^1| \lesssim 1$ from Lemma \ref{lem:tech3}, the estimation as in \eqref{eq:defiteration}, 
\[
\|(n^r k_n^1)_{n}\|^2_{\ell^2} \lesssim \sum_{m\in \N^*} |k_m|^2 m^{2r-2(\alpha-1)}\log^2(m),
\]
and \eqref{knneps00} to deduce 
\[
(k_{n}n^{\varepsilon})_{n\in\mathbb{N}^*}\in l^{\infty}, \; \text{ for any } \varepsilon \in (\alpha-1,2\alpha-\frac{5}{2}),
\]
which is sufficient to conclude. 
For $\alpha \in (1,3/2]$, it suffices to use the fact that, 
\[
k_n = \left(\sum_{j=1}^i e_n^j\right) + k_n^i.
\]
On the one hand, using \eqref{eq:estimEN}, we have 
\[
|e_n^i n^{\varepsilon}| {\color{black} \lesssim} n^{1-\alpha+\varepsilon} \log(n) {\color{black} \lesssim} 1, \; \; \forall i \in \N^*,
\]
as long as $\varepsilon < \alpha-1$. On the other hand, for any $\alpha \in (1,3/2]$, we have that \eqref{knnepsl2inf} holds for any $i\in \N$, and since $s_i \rightarrow -\infty$ as $i\rightarrow \infty$, we deduce that there exists $M \in \N$ such that $(n^{\varepsilon} k_n^M)_{n \in \N^*} \in \ell^2$ for {\color{black} all}  $\varepsilon < \alpha -1$. Hence, $(n^{\varepsilon} k_n^M)_{n \in \N^*} \in \ell^\infty$ for $\varepsilon \in [0,\alpha-1)$ and,
\[
|n^{\epsilon} k_n| \lesssim M + \| (n^{\varepsilon} k_n^M) \|_{\ell^\infty}, 
\]
for $\varepsilon \in [0,\alpha-1)$ and $\alpha\in (1,3/2]$.

Thus, for any $\alpha>1$,  $(f_{n}k_{n}n^{\varepsilon})_{n\in\mathbb{N}^{*}}\in l^{2}$ for $\varepsilon \in [0,\alpha-1)$. 

Now assume $\varepsilon \in [0,r_0)$ where $r_0:=\min(\alpha-\frac{1}{2}-r,\alpha-1)$.} Since $r+\varepsilon<\textcolor{black}{\alpha-1/2}$, we have from Proposition \ref{qnr:Riesz} that $(n^{-(r+\varepsilon)}q_{n})_{n\in\mathbb{N}^{*}}$ is a Riesz basis of $\Hdiv^{r+\varepsilon}$ therefore so is $(n^{-(r+\varepsilon)}\tau q_{n})_{n\in\mathbb{N}^{*}}$. Thus $k(f)\in \Hdiv^{r+\varepsilon}$ and,
\begin{equation*}
\|k(f)\|_{\Hdiv^{r+ \varepsilon}}\lesssim \|f_{n}k_{n}n^{\varepsilon}\|_{l^2} \lesssim \|f\|_{\Hdiv^r}.
\end{equation*}
This ends the proof of Lemma \ref{lemkncomp}.
\end{proof}

\subsection{Step (5): boundedness of the corresponding backstepping transformation}
\label{subsec:step5}
Let us now look at the boundedness of $T$ and show that it satisfies the operator equality \eqref{op-eq0} in some sense. We show the following. 
\begin{lem}
\label{T:boundedopeq}
The operator $T$ given by \eqref{def-T} is a bounded operator from $\Hdiv^{r}$ to $\Hdiv^{r}$ for any $r\in\textcolor{black}{(1/2-\alpha,\alpha-1/2)}$. Moreover, we have the following operator equality,
\begin{equation}
\label{op-eq}
    T(A+BK) = (A-\lambda \textcolor{black}{I})T \; \text{ in }\; \mathcal{L}(\Hdiv^{\textcolor{black}{\alpha/2}+s}; \Hdiv^{-\textcolor{black}{\alpha/2}+s}), \; \forall s\in \left(\textcolor{black}{-\frac{\alpha-1}{2}, \frac{\alpha-1}{2}}\right).
\end{equation}
\end{lem}
\begin{proof}[Proof of Lemma \ref{T:boundedopeq}]
We start by proving that $T$ is a bounded operator using the previous section. Let $r\in\textcolor{black}{(1/2-\alpha,\alpha-1/2)}$. We introduce the operator $\tau_{K}$
    \begin{equation}\label{def:op:tauK}
    \tau_{K} : n^{-r}\tau q_{n}\rightarrow (-K_{n})n^{-r}\tau q_{n},
\end{equation}
which is well defined since $(n^{-r}\tau q_{n})_{n\in\mathbb{N}^{*}}$ is a Riesz basis of $\Hdiv^{r}$ (recall that $(n^{-r} q_{n})_{n\in\mathbb{N}^{*}}$ is a Riesz basis of $\Hdiv^{r}$ \textcolor{black}{for $r\in(1/2-\alpha,\alpha-1/2)$} and $\tau$ is an isomorphism from $\Hdiv^{r}$ to $\Hdiv^{r}$ \textcolor{black}{for $r \in \R$}). From the definition of $T$ given in \eqref{def-T}, we have,
\begin{equation}\label{Tcircdef}
    T = \tau_{K}\circ\tau\circ S.
\end{equation}
Since both $\tau$ and $S$ are bounded from $\Hdiv^{r}$ to itself, it suffices to show that $\tau_{K}$ is bounded to show that $T$ is bounded from $\Hdiv^{r}$ to itself. Since we showed in the previous section that $(K_{n})_{n\in\mathbb{N}^{*}}$ is uniformly bounded from above (see Lemma \ref{lemkncomp}), and $(n^{-r}\tau q_{n})_{n\in\mathbb{N}^{*}}$ is a Riesz basis of $\Hdiv^{r}$, we have that for any $f= \sum_{n\in \N^*} f_n n^{-r}\tau q_{n}\in \Hdiv^r$, 
\begin{equation*}
 \|\tau_{K} f\|_{\Hdiv^{r}}=   \left\|\sum_{n\in \N^*} (- K_n)f_n n^{-r}\tau q_{n}\right\|_{\Hdiv^{r}}\lesssim  \|(f_n K_n)_n\|_{l^2}\lesssim \|(f_n)_n\|_{l^2}\lesssim \|f\|_{\Hdiv^{r}}.
\end{equation*}
Thus $\tau_{K}$ is bounded from $\Hdiv^{r}$ to itself and so is $T$.\\

Let us now prove the operator equality \eqref{op-eq}. We proceed by ensuring that the modified operator equality,
\[TA+BK=(A-\lambda I)T,\]
which have been using to define $T$ and $K$, indeed holds on the right array of functional spaces.

Observe first that all terms make sense: indeed, $K:n^{-r}\varphi_{n}\mapsto n^{-r}K_{n}$ is a bounded operator from $\Hdiv^{\textcolor{black}{\alpha/2}+s}$ to $\mathbb{C}$ for any $s\in\textcolor{black}{(-(\alpha-1)/2,(\alpha-1)/2)}$. On the other hand, $B\in \Hdiv^{-\textcolor{black}{\alpha/2}}$ and in fact in $\Hdiv^{-\textcolor{black}{\alpha/2}+s}$ for any $s\in(\textcolor{black}{-(\alpha-1)/2,(\alpha-1)/2})$ since $(b_{n})_{n\in\mathbb{N}^{*}}$ is uniformly bounded.
So $B$ can be formally seen as an operator from $\mathbb{C}$ to $\Hdiv^{-\textcolor{black}{\alpha/2}+s}$ for any $s\in(\textcolor{black}{-(\alpha-1)/2,(\alpha-1)/2})$. Thus, $BK$ is a bounded operator from $\Hdiv^{\textcolor{black}{\alpha/2}+s}$ to $\Hdiv^{-\textcolor{black}{\alpha/2}+s}$ for any $s\in(\textcolor{black}{-(\alpha-1)/2,(\alpha-1)/2})$. Similarly one can check that both $AT$ and $TA$ are bounded operators from $\Hdiv^{\textcolor{black}{\alpha/2}+s}$ to $\Hdiv^{-\textcolor{black}{\alpha/2}+s}$ when $s\in(\textcolor{black}{-(\alpha-1)/2,(\alpha-1)/2})$ (since $T$ is a bounded operator in $\Hdiv^{r}$ for any $r\in(\textcolor{black}{1/2-\alpha,\alpha-1/2})$). To show \eqref{op-eq} it suffices to check that it holds against $n^{-\textcolor{black}{\alpha/2}-s}\varphi_{n}$ for any $n\in\mathbb{N}^{*}$ and $s\in(\textcolor{black}{-(\alpha-1)/2,(\alpha-1)/2})$. Since the operators are linear, it suffices to verify that it holds against $\varphi_{n}$ for any $n\in\mathbb{N}^{*}$. From the definition of $T$ (see \eqref{def-T}) we have (in $\Hdiv^{-\textcolor{black}{\alpha/2}+s}$)
\begin{equation*}
\begin{split}
\left[(TA+BK)\right.&\left.-(A-\lambda I)T\right]\varphi_{n} =\lambda_{n}T\varphi_{n}+BK_{n}-(A-\lambda I)T\varphi_{n}\\
&=\lambda_{n}(-K_{n})\left(\sum\limits_{p\in\mathbb{N}^{*}}\frac{b_{p}\varphi_{p}}{\lambda_{n}-\lambda_{p}+\lambda}\right)+BK_{n}-(-K_{n})\left(\sum\limits_{p\in\mathbb{N}^{*}}\frac{b_{p}(\lambda_{p}-\lambda)\varphi_{p}}{\lambda_{n}-\lambda_{p}+\lambda}\right).\\
&=(-K_{n})\left(\sum\limits_{p\in\mathbb{N}^{*}}b_{p}\varphi_{p}\right)+BK_{n}  \\
&= 0.
\end{split}
\end{equation*}
Then, using the $TB=B$ condition{\color{black}, which holds in $\Hdiv^{-\alpha/2+s}$ with $s\in (-(\alpha-1)/2,(\alpha-1)/2)$ thanks to Lemma \ref{defkn},} we have,
\[TA+BK=T(A+BK), {\color{black} \; \text{ in }\; \mathcal{L}(\Hdiv^{\textcolor{black}{\alpha/2}+s}; \Hdiv^{-\textcolor{black}{\alpha/2}+s}), \; \forall s\in \left(\textcolor{black}{-\frac{\alpha-1}{2}, \frac{\alpha-1}{2}}\right),}\]
which ends the proof of Lemma \ref{T:boundedopeq}.
\end{proof}

\subsection{Step (6): Fredholm operator property in $\Hdiv^{-\alpha/2}$}
We now prove the following lemma. 
\begin{lem}
\label{lem:Tfredholm}
$T$ defined by \eqref{def-T} with $K_{n}$ defined by Lemma \ref{lem:defk} is a Fredholm operator (of index 0) from $\Hdiv^{-\textcolor{black}{\alpha/2}}$ to $\Hdiv^{-\textcolor{black}{\alpha/2}}$.
\end{lem}
\begin{proof}[Proof of Lemma \ref{lem:Tfredholm}]
In order to show this\footnote{One could wonder: why showing this in $\Hdiv^{-\textcolor{black}{\alpha/2}}$ while this seems to hold in $\textcolor{black}{H}$ as well? This will become clearer in the next section: it is easier to show first that $T$ is an isomorphism in $\Hdiv^{-\textcolor{black}{\alpha/2}}$ rather than in $\textcolor{black}{H}$ and then deduce that $T$ is an isomorphism in $\textcolor{black}{H}$.} it suffices to show that $T_{c} = T-Id$ is a compact operator in $\Hdiv^{-\textcolor{black}{\alpha/2}}$. Let $f\in \Hdiv^{-\textcolor{black}{\alpha/2}}$, and denote $\left( f,n^{\textcolor{black}{\alpha/2}}\varphi_{n}\right)_{\Hdiv^{-\textcolor{black}{\alpha/2}}}= f_{n}\in \ell^2$ such that,
\begin{equation*}
    f =\sum\limits_{n\in\mathbb{N}^{*}} f_{n}(n^{\textcolor{black}{\alpha/2}}\varphi_{n}).
\end{equation*}
We have, using the fact that $T$ is bounded from $\Hdiv^{-\textcolor{black}{\alpha/2}}$ to $\Hdiv^{-\textcolor{black}{\alpha/2}}$,
\begin{equation}
\label{eq:tc}
\begin{split}
    T_{c}f &= Tf-f\\
    &= \sum\limits_{n\in\mathbb{N}^{*}}f_{n} (-K_{n})n^{\textcolor{black}{\alpha/2}}\tau q_{n}-\sum\limits_{n\in\mathbb{N}^{*}}f_{n}n^{\textcolor{black}{\alpha/2}} \varphi_{n}\\
    &= \left(\sum\limits_{n\in\mathbb{N}}\frac{f_{n}}{b_{n}}n^{\textcolor{black}{\alpha/2}}\lambda \tau q_{n}\right) -\left(\sum\limits_{n\in\mathbb{N}^{*}}f_{n}n^{\textcolor{black}{\alpha/2}} \varphi_{n}\right)+\left(\sum\limits_{n\in\mathbb{N}}\frac{f_{n}}{b_{n}}n^{\textcolor{black}{\alpha/2}}k_{n} \tau q_{n}\right).
\end{split}    
\end{equation}
Now observe that, as previously,
\begin{equation*}
\begin{split}
&\left(\sum\limits_{n\in\mathbb{N}^{*}}\frac{f_{n}}{b_{n}}n^{\textcolor{black}{\alpha/2}}\lambda \tau q_{n}\right)-\left(\sum\limits_{n\in\mathbb{N}^{*}}f_{n}n^{\textcolor{black}{\alpha/2}} \varphi_{n}\right)\\
&=\left(\sum\limits_{n\in\mathbb{N}^{*}}\frac{f_{n}}{b_{n}}n^{\textcolor{black}{\alpha/2}}\lambda\sum\limits_{p\in\mathbb{N}^{*}\setminus\{n\}}\frac{b_{p}\varphi_{p}}{\lambda_{n}-\lambda_{p}+\lambda}\right)+\left(\sum\limits_{n\in\mathbb{N}^{*}}f_{n}n^{\textcolor{black}{\alpha/2}}\lambda\frac{\varphi_{n}}{\lambda}\right)-\left(\sum\limits_{n\in\mathbb{N}^{*}}f_{n}n^{\textcolor{black}{\alpha/2}} \varphi_{n}\right)\\
&=\lambda\left(\sum\limits_{n\in\mathbb{N}^{*}}\frac{f_{n}}{b_{n}}n^{\textcolor{black}{\alpha/2}}\sum\limits_{p\in\mathbb{N}^{*}\setminus\{n\}}\frac{b_{p}\varphi_{p}}{\lambda_{n}-\lambda_{p}+\lambda}\right).
\end{split}
\end{equation*}
Together with \eqref{eq:tc} and the definition of the operator $k$ \eqref{def-op-k}, this gives
\begin{equation*}
T_{c}f = \lambda \left(\sum\limits_{n\in\mathbb{N}^{*}}\frac{f_{n}}{b_{n}}n^{\textcolor{black}{\alpha/2}}\sum\limits_{p\in\mathbb{N}^{*}\setminus\{n\}}\frac{b_{p}\varphi_{p}}{\lambda_{n}-\lambda_{p}+\lambda}\right)+k(\tau^{-1}f).
\end{equation*}
From the definition $S_{c}$ given by Lemma \ref{lem4} and $\tau$ given in \eqref{def-tau}, this can be expressed as
\begin{equation*}
    T_{c}f = \lambda \tau\circ S_{c}\circ \tau^{-1} f+k\circ \tau^{-1}f.
\end{equation*}
We know from Lemma \ref{lemkncomp} that $k$ is a compact operator from $\Hdiv^{-\textcolor{black}{\alpha/2}}$ to $\Hdiv^{-\textcolor{black}{\alpha/2}}$ and $\tau^{-1}$ is an isomorphism from $\Hdiv^{-\textcolor{black}{\alpha/2}}$ to $\Hdiv^{-\textcolor{black}{\alpha/2}}$ thus $k \circ \tau^{-1}$ is a compact operator from $\Hdiv^{-\textcolor{black}{\alpha/2}}$ to $\Hdiv^{-\textcolor{black}{\alpha/2}}$. Similarly $S_{c}$ is a compact operator from $\Hdiv^{-\textcolor{black}{\alpha/2}}$ to $\Hdiv^{-\textcolor{black}{\alpha/2}}$ from Lemma \ref{lem4}, and therefore $\tau \circ S_{c}\circ \tau^{-1}$ is a compact operator from $\Hdiv^{-\textcolor{black}{\alpha/2}}$ to $\Hdiv^{-\textcolor{black}{\alpha/2}}$.
Hence, 
$T_{c}$ is 
a compact operator from $\Hdiv^{-\textcolor{black}{\alpha/2}}$ to $\Hdiv^{-\textcolor{black}{\alpha/2}}$. This ends the proof of Lemma \ref{lem:Tfredholm}.
\end{proof}
\subsection{Step (7): invertibility in $\Hdiv^{-\alpha/2}$}\label{subsec:step8}
We now turn to the following lemma. 
\begin{lem}
\label{lem:iso-34}
$T$ defined by \eqref{def-T} with $K_{n}$ given by Lemma \ref{lem:defk} defines an isomorphism from $\Hdiv^{-\textcolor{black}{\alpha/2}}$ to $\Hdiv^{-\textcolor{black}{\alpha/2}}$.
\end{lem}
\begin{proof}[Proof of Lemma \ref{lem:iso-34}]
We already know that $T$ is a bounded operator from $\Hdiv^{-\textcolor{black}{\alpha/2}}$ to $\Hdiv^{-\textcolor{black}{\alpha/2}}$ and that $T$ is also a Fredholm operator of index 0. Therefore, using the fact that the adjoint of a Fredholm operator is still a Fredholm operator from Schauder's theorem, and that $dim (ker(T)) =dim(coker(T))= dim (ker(T^{*}))<+\infty$, it suffices to show that, 
\begin{equation}
\label{kert*=0}
ker(T^{*}) = \{0\} \;  \textrm{ in } \Hdiv^{-\textcolor{black}{\alpha/2}},
\end{equation}
to prove that $T$ is an isomorphism from $\Hdiv^{-\textcolor{black}{\alpha/2}}$ to $\Hdiv^{-\textcolor{black}{\alpha/2}}$, where $T^*$ is the adjoint of $T$ (taken as an operator from $\Hdiv^{-\textcolor{black}{\alpha/2}}$ for $\Hdiv^{-\textcolor{black}{\alpha/2}}$). From that point, the method is inspired by  what is done for the Schrödinger equation or the heat equation in \cite{CGM, Gagnon-Hayat-Xiang-Zhang}: the proof is composed of three main steps:
\begin{enumerate}
    \item[1.] There exists $\rho\in\mathbb{C}$ such that both
           $A+ \textcolor{black}{B}K+\lambda \textcolor{black}{Id} +\rho \textcolor{black}{Id}$ and $A+ \rho \textcolor{black}{Id}$ are invertible operator from $\Hdiv^{\textcolor{black}{\alpha/2}}$ to $\Hdiv^{-\textcolor{black}{\alpha/2}}$.
    \item[2.] For such a $\rho$, if $ker(T^{*})\neq \{0\}$, then $(A+\rho \textcolor{black}{Id})^{-1}$ has an eigenvector $h$ which belongs to $ker(T^*)$.
\item[3.] No eigenvector of $(A+\rho \textcolor{black}{Id})^{-1}$ belong to $ker(T^{*})$.
\end{enumerate}
From {\it Step 2} and {\it Step 3} we deduce that $ker(T^{*})= \{0\}$. Given that this is very similar to what is done in \cite{CGM,Gagnon-Hayat-Xiang-Zhang}, the rigorous proof is postponed to the Appendix \ref{isomorph}.  In addition, we also provide another direct method to prove the first step instead of the method using the perturbation theory of operators that is introduced in \cite{CGM}.
\end{proof}
\subsection{Step (8): invertibility on a range of Sobolev spaces}
Now that we know from Lemma \ref{lem:iso-34} that $T$ is an isomorphism from $\Hdiv^{-\textcolor{black}{\alpha/2}}$ to $\Hdiv^{-\textcolor{black}{\alpha/2}}$, we are going to show the following proposition.
\begin{prop}
\label{prop:isot}
For any $r\in\textcolor{black}{(1/2-\alpha,\alpha-1/2)}$, the operator $T$ given by \eqref{def-T} is an isomorphism from $\Hdiv^{r}$ to $\Hdiv^{r}$. In particular $T$ is an isomorphism from $\textcolor{black}{H}$ to $\textcolor{black}{H}$.
\end{prop}
\begin{proof}[Proof of Proposition \ref{prop:isot}]
We first show that for any $n\in\mathbb{N}^{*}$, $K_{n}\neq 0$. Let $m\in\mathbb{N}^{*}$. Since $T$ is an isomorphism from $\Hdiv^{-\textcolor{black}{\alpha/2}}$ to $\Hdiv^{-\textcolor{black}{\alpha/2}}$ from Lemma \ref{lem:iso-34} there exists $h\in \Hdiv^{-\textcolor{black}{\alpha/2}}$ such that, 
\begin{equation}
\label{defhn0}
T h  = m^{\textcolor{black}{\alpha/2}}\tau q_{m}.
\end{equation}
As $h\in \Hdiv^{-\textcolor{black}{\alpha/2}}$ there exists $(h_{n})_{n\in \mathbb{N}^{*}}\in l^{2}$ such that,
\begin{equation*}
   h = \sum\limits_{n\in\mathbb{N}^{*}} h_{n}n^{\textcolor{black}{\alpha/2}}\varphi_{n},
\end{equation*}
hence,
\begin{equation*}
    Th = \sum\limits_{n\in\mathbb{N}^{*}} (-K_{n})h_{n}n^{\textcolor{black}{\alpha/2}}\tau q_{n}.
\end{equation*}
As $(n^{\textcolor{black}{\alpha/2}} q_{n})_{n\in\mathbb{N}^{*}}$ is a Riesz basis of $\Hdiv^{-\textcolor{black}{\alpha/2}}$ and $\tau:\; \Hdiv^{-\textcolor{black}{\alpha/2}}\rightarrow \Hdiv^{-\textcolor{black}{\alpha/2}}$ is an isomorphism, then $(n^{\textcolor{black}{\alpha/2}}\tau q_{n})_{n\in\mathbb{N}^{*}}$ is also a Riesz basis of $\Hdiv^{-\textcolor{black}{\alpha/2}}$ which means, together with \eqref{defhn0}, that,
\begin{equation*}
h_{m}(-K_{m}) = 1, 
\end{equation*}
hence $K_{m}\neq 0$. Since it is true for any $m\in \N^*$, we deduce that $K_n \neq 0$, for any $n\in \N^*$.\\

Now, we are going to show that $(K_{n})_{n\in\mathbb{N}^{*}}$ is bounded \textcolor{black}{from} below. Note that from \eqref{defkn},
\begin{equation*}
    |K_{n}b_{n}|=|\lambda+k_{n}|,
\end{equation*}
where $(k_{n}n^{\varepsilon})_{n\in\mathbb{N}^{*}}\in l^{\infty}$, for some $\varepsilon>0$ (see \eqref{kneps}), which means that $k_{n}\rightarrow 0$ when $n\rightarrow +\infty$. Hence, there exists $n_{0}\in\mathbb{N}^{*}$ such that for any $n\geq n_{0}$,
\begin{equation}
\label{boundinfkn}
    |K_{n}b_{n}|\geq \frac{\lambda}{2}>0.
\end{equation}
As $K_{n}$ and $b_{n}$ do not vanish, $\min_{n\leq n_{0}}(|K_{n}b_{n}|)>0$. This together with \eqref{boundinfkn} means that $|K_{n}b_{n}|$ is uniformly bounded \textcolor{black}{from} below. Since $b_{n}$ is uniformly bounded we deduce that there exists a constant $c>0$ independent of $n$ such that,
\begin{equation*}
    |K_{n}|\geq c>0,\;\;\forall\; n\in\mathbb{N}^{*}.
\end{equation*}
We can now conclude: since $(K_{n})_{n\in\mathbb{N}^{*}}$ is uniformly bounded \textcolor{black}{from} below and above (see Lemma \ref{lemkncomp}), the operator,
\begin{equation*}
    \tau_{K} : n^{-r}\tau q_{n}\mapsto (-K_{n})n^{-r}\tau q_{n},
\end{equation*}
is an isomorphism from $\Hdiv^{r}$ to $\Hdiv^{r}$, for any $r\in(\textcolor{black}{1/2-\alpha,\alpha-1/2})$. We used here that $(n^{-r}\tau q_{n})_{n\in\mathbb{N}^{*}}$ is a Riesz basis of $\Hdiv^{r}$, for any $r\in(\textcolor{black}{1/2-\alpha,\alpha-1/2})$. Using the same fact, the operator,
\begin{equation*}
n^{-r}\varphi_{n} \mapsto n^{-r}\tau q_{n},
\end{equation*}
is also an isomorphism from $\Hdiv^{r}$ to $\Hdiv^{r}$, for any $r\in(\textcolor{black}{1/2-\alpha,\alpha-1/2})$. Hence, composing using $\tau_{K}$,
\begin{equation*}
    T : n^{-r}\varphi_{n} \mapsto (-K_{n})n^{-r}\tau q_{n},
\end{equation*}
is an isomorphism from $\Hdiv^{r}$ to $\Hdiv^{r}$. \textcolor{black}{In particular, given that $\Hdiv^{0}=H$, it is an isomorphism from $H$ to itself.} This ends the proof of Proposition \ref{prop:isot}.
\begin{remark}
This last argument simply means that $T = \tau_{K}\circ \tau\circ S$, each of these three operator\textcolor{black}{s} being an isomorphism from $\Hdiv^{r}$ to $\Hdiv^{r}$.
\end{remark}
\end{proof}

\section{Well-posedness and stability  of the closed-loop system}\label{sec-wellposed}

{\color{black}Our proof of the well-posedness of the closed-loop system is based on semigroup theory. }The closed-loop system has the form, 
\begin{equation}\label{sys:cloloopwell}
\left\{ \begin{array}{ll}
\partial_t u= Au +BKu, & t \in (0,T), \\
u|_{t=0} = u_0(x),
\end{array} \right.
\end{equation}
with,
\begin{equation*}
    B= \sum_{n\in \N^*}  b_n \varphi_n  \in \Hdiv^{-\textcolor{black}{\alpha/2}},
\end{equation*}
\textcolor{black}{Observe in addition that since $B$ satisfies Assumption \ref{asump1}, $B\in  \Hdiv^{-1/2-\varepsilon}$ for $\varepsilon$ sufficiently small. {Similarly we have}}
{\color{black}
\begin{equation*} 
\begin{array}{rcl}
    K:\;  \Hdiv^{\textcolor{black}{1/2}+\varepsilon}&\rightarrow & \mathbb{C} ,\\
    \varphi_n & \textcolor{black}{\longmapsto} & K_n.
\end{array}
\end{equation*}
}
We begin by studying the closed-loop operator $A+BK$ and the spaces on which the operator equality is defined. Then we move on to the proof of well-posedness by the semigroup theory. 

\subsection{Properties of the spaces $D_r(A+BK)$.}

Clearly, $A+ BK$ can be defined on $\Hdiv^{1/2+\varepsilon}$. {\color{black}In particular, }
\begin{equation*}
    A+ BK: \Hdiv^{\textcolor{black}{\alpha/2}}\rightarrow  \Hdiv^{-\textcolor{black}{\alpha/2}}.
\end{equation*} 
However, it does not map the \textcolor{black}{natural space $\Hdiv^{\alpha}$ (which is exactly $D(A)$)} to the ``desired" $\textcolor{black}{H}$ space \textcolor{black}{(recall that $H=\Hdiv^{0}$)}:
\begin{equation*}
     A+ BK: \Hdiv^{\textcolor{black}{\alpha}}\rightarrow  \Hdiv^{-1/2-\varepsilon},
\end{equation*}
{\color{black}as $B \in \Hdiv^{-1/2-\varepsilon}$.} Therefore, {\color{black}for $r\in (1/2-\alpha, \alpha-1/2)$, we shall define the unbounded operator  $A+ BK$ as follows:
\begin{gather}
\label{A+BK-operator}
     A+ BK: D_r(A+BK)\subset \Hdiv^r \rightarrow  \Hdiv^r, \\
     D_r(A+BK):= \{f\in \Hdiv^r: (A+ BK)f\in \Hdiv^r\},
\end{gather}
where the space $D_r(A+BK)$ is endowed with the usual norm 
\begin{equation*}
    \|y\|_{D_r(A+BK)}:= \|y\|_{\Hdiv^r}+ \|(A+ BK) y\|_{\Hdiv^r}, \; \forall y\in D_r(A+BK).
\end{equation*}}
{\color{black}

Let us first point out that every $f\in D_r(A+BK)\subset \Hdiv^{r}$ satisfies $Af\in \Hdiv^{r-\textcolor{black}{\alpha}}$. Since $r<\alpha-1/2$, 
\[BKf\in \Hdiv^{r-\alpha}\subset \Hdiv^{-1/2-\varepsilon}\setminus \Hdiv^{-1/2},\]
which implies that $Kf$ is well-defined 
(and in particular $Kf\in\mathbb{C}$). 
From the definition of $D_r(A+ BK)$, we then get $Af\in \Hdiv^{-1/2-\varepsilon}$. 
Hence, $f= A^{-1} (Af)\in \Hdiv^{\alpha-1/2-\varepsilon}$. As a result, for $r\in \left(1/2-\alpha, \alpha-1/2\right)$, and for $\e<\alpha-r-1/2$, $D_r(A+BK)\subset \Hdiv^{r+\e}$.


In order to further our study, we now use the operator equality \eqref{op-eq} in the appropriate functional setting.
We first show that for $r\in \left(1/2-\alpha, \alpha-1/2\right)$,
\begin{gather}\label{oe.eq.opeq}
    T(A+ BK)f= (A-\lambda I)Tf \in \Hdiv^r, \quad \forall f\in  D_r(A+BK).
\end{gather}
Recall from \eqref{op-eq} that this equality holds in $\Hdiv^{-\textcolor{black}{\alpha/2}+s}$, for $f\in \Hdiv^{\textcolor{black}{\alpha/2}+s}$  for any $s\in (\textcolor{black}{-(\alpha-1)/2, (\alpha-1)/2)}$ but \emph{a priori} we do not have any information in $\Hdiv^r$. Since $D_r(A+BK)\subset \Hdiv^{\alpha-1/2-\e}$, \eqref{op-eq} holds for $f\in D_r(A+BK)$, so that we only need to check the regularity of both sides of the equality. Now, for $f\in D_r(A+BK)$, by property of $T$ and definition of $D_r(A+BK)$, we have $T(A+ BK)f\in \Hdiv^r$, hence $(A-\lambda I)Tf\in \Hdiv^r$, and \eqref{oe.eq.opeq} follows. \\

We can now use the operator equality to characterize the spaces $D_r(A+BK)$.

\begin{lem}\label{lem:defdomaTAB}
Let $r\in \left(1/2-\alpha, \alpha-1/2\right)$. Then, $D_r(A+ BK)= T^{-1} (\Hdiv^{\textcolor{black}{r+\alpha}})$.
\end{lem}

\begin{proof}[Proof of Lemma \ref{lem:defdomaTAB}]
As a direct consequence of the above, for $f\in D_r(A+BK)$, we get $ATf\in \Hdiv^r$, thus $Tf\in \Hdiv^{\textcolor{black}{r+\alpha}}\subset \Hdiv^r$. This implies by Proposition \ref{prop:isot} that
\[D(A+BK) \subset T^{-1} (\Hdiv^{\textcolor{black}{r+\alpha}}).\]
It remains to prove that $T^{-1} (\Hdiv^{\textcolor{black}{r+\alpha}})\subset D(A+BK)$.  Let $\tilde{f}\in T^{-1} (\Hdiv^{\textcolor{black}{r+\alpha}})\subset \Hdiv ^{\alpha-1/2-\e}$ so that $(A-\lambda I)T \tilde f\in \Hdiv^r$. Applying the operator equality \eqref{op-eq} to $f$, and using Proposition \ref{prop:isot}, we get
\begin{equation*}
    T(A+ BK)\tilde f\in \Hdiv^r, \quad i.e., \quad (A+ BK)\tilde f\in \Hdiv^r , \quad i.e., \quad \tilde f\in D_r(A+BK).
\end{equation*}
This concludes the proof of the lemma.
\end{proof}
Since $T$ is an isomorphism on $\Hdiv^r$, and by density of $\Hdiv^{r+\alpha}$ in $\Hdiv^r$, we immediately get:
\begin{cor}\label{cor-domain-density}
For $r\in \left(1/2-\alpha, \alpha-1/2\right)$, the space $D_r(A+BK)$ is dense in $\Hdiv^r$.
\end{cor}

We show the following lemma.
\begin{lem}
\label{lem:well3}
For $r\in \left(1/2-\alpha, \alpha-1/2\right)$, $D_{\textcolor{black}{r}}(A+ BK)$ is a Hilbert space. Moreover, 
\begin{equation*}
    T: D_{\textcolor{black}{r}}(A+ BK)= T^{-1}(\Hdiv^{\textcolor{black}{r+\alpha}})\rightarrow \Hdiv^{\textcolor{black}{r+\alpha}},
\end{equation*}
is an isomorphism, and 
\begin{equation}\label{eq:opequalitylem63}
     T(A+ BK)= (A-\lambda \textcolor{black}{I})T\in \mathcal{L}(D_{\textcolor{black}{r}}(A+ BK); \textcolor{black}{\Hdiv^r}).
\end{equation}
\end{lem}
\begin{proof}[Proof of Lemma \ref{lem:well3}]
We first show that $D_{\textcolor{black}{r}}(A+BK)$ is complete and hence is a Hilbert space, the proof of which is classical and that we recall. Given $(f_n)_{n\in\mathbb{N}^{*}}$ a Cauchy sequence in $D_r(A+BK)$, \textcolor{black}{we have}
\begin{equation*}
    \|(A+ BK) (f_n- f_m)\|_{\Hdiv^r}+ \| f_n- f_m\|_{\Hdiv^r}\xrightarrow[n, m\rightarrow +\infty]{}  0.
\end{equation*}
Since $(f_n)_{n\in \N^*}$ is a sequence of elements of $D_r(A+ BK)$, according to the operator equality \eqref{oe.eq.opeq}, 
\begin{equation*}
   \|(A-\lambda I)(T f_n- T f_m)\|_{\Hdiv^r} =\|T(A+ BK) (f_n- f_m)\|_{\Hdiv^r}\lesssim\|(A+ BK) (f_n- f_m)\|_{\Hdiv^r}\rightarrow 0,
\end{equation*}
thus
\begin{equation*}
   \|(A(T f_n- T f_m)\|_{\Hdiv^r}\rightarrow 0\;\;\text{ which implies }\;\; \|T f_n- T f_m\|_{\Hdiv^{r+\alpha}}\rightarrow 0.
\end{equation*}
Since $\Hdiv^{r+\alpha}$ is complete, there exists some $Tf\in \Hdiv^{r+\alpha}$ such that $Tf_n \xrightarrow[n\to\infty]{\Hdiv^{r+\alpha}}Tf$. As a direct consequence we immediately obtain that, 
\begin{equation}
\label{cv-1}f_{n}\xrightarrow[n\to\infty]{\Hdiv^r}f.
\end{equation}

Moreover, 
\begin{align}\label{cv-2}
    \|(A+BK)f_n-(A+BK)f\|_{\Hdiv^r} &\lesssim\|T(A+ BK)f_n- T(A+ BK)f\|_{\Hdiv^r}  \nonumber \\
    &= \|(A-\lambda I)T (f_n- f)\|_{\Hdiv^r}\rightarrow 0,
\end{align}
as $Tf_{n}\xrightarrow[n\to\infty]{\Hdiv^{r+\alpha}} Tf$. Thus, putting \eqref{cv-1} and \eqref{cv-2} together, we get $f_n \xrightarrow[n\to\infty]{D_r(A+BK)} f$.\\

Next, we turn to the second part of Lemma \ref{lem:well3} and we show that $T: D_r(A+BK)\rightarrow \Hdiv^{r+\alpha}$ is an isomorphism.   For any $y\in D_r(A+ BK)= T^{-1}(\Hdiv^{r+\alpha})$, we know that, since the operator equality holds in $\Hdiv^r$,
\begin{equation*}
\begin{split}
     \|y\|_{D_r(A+ BK)}&= \|(A+ BK)y\|_{\Hdiv^r}+ \|y\|_{\Hdiv^r} \\
     &\simeq  \|T(A+ BK)y\|_{\Hdiv^r}+ \|y\|_{\Hdiv^r}\\
   &= \|(A-\lambda I)Ty\|_{\Hdiv^r}+ \|y\|_{\Hdiv^r}\\
   &\simeq \|Ty\|_{\Hdiv^{r+\alpha}},
 \end{split}
\end{equation*}
which implies that $T: D_r(A+BK)\rightarrow \Hdiv^{r+\alpha}$ is an isomorphism. \\

Finally, we immediately get from the definition of $D_r(A+BK)$ and the fact that it is a Hilbert space that
\begin{equation*}
   T(A+ BK), (A-\lambda \textcolor{black}{I})T\in \mathcal{L}(D(A+ BK); H).  
\end{equation*}
This proves \eqref{eq:opequalitylem63} and thus concludes the proof of Lemma \ref{lem:well3}.
\end{proof}
}

{\color{black}
\subsection{Regularity of the domains and spectral properties of $A+BK$.}

{\color{black}
At this point, our characterization of the $D_r(A+BK)$ is rather indirect. To gain a better understanding of these spaces, we will now explore their relationship with the spaces $\Hdiv^s$, an endeavour which we have already started at the beginning of this subsection.

Our first remark is that, since $T$ is an isomorphism on $\Hdiv^s$ for $s\in (1/2-\alpha, \alpha-1/2) $, we have
\begin{equation}
\label{low-reg-domains}
D_r(A+BK)=T^{-1}(\Hdiv^{r+\alpha})=\Hdiv^{r+\alpha}, \quad \forall r\in \left(\frac12-\alpha, -\frac12\right).\end{equation}
In particular, for $r\in \left(\frac12-\alpha, -\frac12\right)$ the eigenvectors $(\varphi_n)_{n\in \mathbb{N}^\ast}$ form a Riesz basis of $D_r(A+BK)$.
On the other hand, recall that $B\notin \Hdiv^{-\frac12}$. Hence, for $r\in [-1/2, \alpha-1/2)$, and for $f\in \Hdiv^{r+\alpha}$, since $Af \in \Hdiv^r \subset \Hdiv^{-\frac12}$,  $f \notin D_r(A+BK)$, \ie
\begin{equation}
\label{high-reg-domains}
\Hdiv^{r+\alpha} \cap D_r(A+BK) = \{0\}, \quad \forall r\in \left[-\frac12, \alpha-\frac12\right).\end{equation}
In particular, for $r\in \left[-\frac12, \alpha-\frac12\right)$, $\varphi_n \notin D_r(A+BK), \ \forall n \in \mathbb{N}^\ast$.

Let us stress this interesting dichotomy: for low regularity spaces, \eqref{low-reg-domains} tells us that the domain $D_r(A+BK)$ is the expected $\Hdiv^{r+\alpha}$. Then, past the threshold given by $r=-1/2$, for higher regularity spaces, \eqref{high-reg-domains} tells us that $D_r(A+BK)$ is completely unrelated to $\Hdiv^{r+\alpha}$.
This is to be understood in light of the regularity of the control operator $B$. We have implemented the backstepping method in an array of spaces, using a controllability assumption. 
Now, in all these spaces, $B$ is an admissible control operator, but it is bounded \textit{only} in the low regularity spaces given by \eqref{low-reg-domains}. 
In that case, it is indeed usual that the closed-loop operator should have the same domain as $A$. Accordingly, \eqref{low-reg-domains} tells us that $T$ preserves regularity, \textit{i.e.,} $T(\Hdiv^{r+\alpha})= \Hdiv^{r+\alpha}, \ \forall r\in \left(\frac12-\alpha, -\frac12\right)$.

On the other hand, in the higher regularity spaces, $B$ is unbounded.  
Accordingly, the backstepping transformation does not preserve higher regularity, \textit{i.e.,} $T(\Hdiv^{r+\alpha})\cap \Hdiv^{r+\alpha}=\{0\}, \ \forall r\in \left[-\frac12, \alpha-\frac12\right)$ . As a result, the domains are more implicitly defined. The same difficulties regarding the domain of the closed-loop operator were already mentioned in \cite{komornikStab, Vest}. Here, thanks to a comprehensive study of the operator and its domain, we are able to give a more precise definition of the domain, and even a Riesz basis.
}

Indeed, having characterized the domains $D_r(A+BK)$, we can now study the spectral properties of $A+BK$, using again the operator inequality. Consider the formal calculation:
\begin{equation}
    \label{A+BK-formal-eigenvectors}
    \begin{aligned}
(A+BK)T^{-1}\varphi_n&=T^{-1}T(A+BK)T^{-1}\varphi_n\\
&=T^{-1}(A-\lambda I)TT^{-1}\varphi_n\\
&=T^{-1}(A-\lambda I)\varphi_n \\
&=(\lambda_n-\lambda)T^{-1}\varphi_n, \quad \forall n \in \mathbb{N}^\ast.
\end{aligned}
\end{equation}
Now, we know from \eqref{oe.eq.opeq} that this formal derivation holds if one can apply the operator equality to $T^{-1}\varphi_n$, \ie if $T^{-1}\varphi_n \in D_r(A+BK)$. By Lemma \ref{lem:defdomaTAB}, this always holds. Thus, \eqref{A+BK-formal-eigenvectors} leads to the following proposition:
\begin{prop}
For $r\in (1/2-\alpha, \alpha-1/2)$, the unbounded operator $A+BK$ with domain $D_r(A+BK)$ admits a Riesz basis of eigenvectors in $\Hdiv^r$, given by $(T^{-1}\varphi_n)_{n\in \mathbb{N}^\ast}$, with corresponding eigenvalues $(\lambda_n-\lambda)_{n\in \mathbb{N}^\ast}$.
\end{prop}

\subsection{Well-posedness and exponential stability.}
\begin{prop}\label{prop:wellposedcls}
The unbounded operator $A+ BK$ with domain $D_r(A+BK)$ is the infinitesimal generator of a strongly continuous $C^0$-semigroup on $\Hdiv^r$. 
\end{prop}
\begin{proof}[Proof of Proposition \ref{prop:wellposedcls}]
This proof is inspired by \cite{coron:hal-03161523} for the well-posedness of the closed-loop water tank system. Note that we unfortunately cannot adapt the proof from  \cite{Gagnon-Hayat-Xiang-Zhang} concerning well-posedness of the heat equations for {\color{black}the} lack of smoothing effects. \\

For $r\in (1/2-\alpha, \alpha-1/2)$, it is easy to prove, thanks to its spectral properties, that the operator $A-\lambda \textcolor{black}{I}$ with domain $D_r(A-\lambda I)=\Hdiv^{\textcolor{black}{r+\alpha}}$ (which always contains the eigenfunctions of $A-\lambda I$) generates a strongly continuous semigroup on $\Hdiv^r$, which we denote $e^{t(A-\lambda I)}$. 

Now, define the following strongly continuous semigroup on $\Hdiv^r$:
\begin{equation}
    \begin{array}{rrcl}
    S: &\mathbb{R}^+ &\rightarrow& \mathcal{L}(\Hdiv^r)\\
    & t &\mapsto& T^{-1}e^{t(A-\lambda I)}T \in \Hdiv^r.
    \end{array}
    \end{equation}
We prove that the infinitesimal generator of $S(t), \, t\geq 0$ is $A+BK$ with domain $D_r(A+BK)$. Let $f\in D_r(A+BK)$. Then, by Lemma \ref{lem:defdomaTAB}, $Tf\in \Hdiv^{r+\alpha}=D_r(A-\lambda I)$ and clearly, by definition of $e^{t(A-\lambda I)}$,
\[\begin{aligned}
\frac{S(t)f-f}t&=T^{-1}\left(\frac{e^{t(A-\lambda I)}Tf-Tf}t \right)\\
&\xrightarrow[t\to 0^+]{\Hdiv^r} T^{-1}(A-\lambda I)Tf.
\end{aligned}\]
Conversely, let $f\in \Hdiv^r$ be such that there exists $h\in \Hdiv^r$ such that
\[\frac{S(t)f-f}t\xrightarrow[t\to 0^+]{\Hdiv^r} h.\]
Then, as $T$ is an isomorphism on $\Hdiv^r$,
\[\begin{aligned}\frac{e^{t(A-\lambda I)}Tf-Tf}t&=\frac{TS(t)T^{-1} (Tf) - Tf}t\\
&=T\frac{S(t)f-f}t\\
&\xrightarrow[t\to 0^+]{\Hdiv^r} Th.\end{aligned}\]
This implies that $Tf\in D_r(A-\lambda I)=\Hdiv^{r+\alpha}$, \ie $f\in T^{-1}(\Hdiv^{r+\alpha})=D_r(A+BK)$.

Thus, $D_r(A+BK)$ is indeed the domain of the infinitesimal generator of $S(t), \, t\geq 0$.
Moreover, by \eqref{oe.eq.opeq},
\[\begin{aligned}
\lim_{t\to 0^+} \frac{S(t)f-f}t&=T^{-1}(A-\lambda I) Tf \\
&=T^{-1} T (A+BK)f \\
&=(A+BK)f, \quad \forall f \in D_r(A+BK),
\end{aligned}\]
which concludes the proof of the proposition.
\end{proof}
Now that we have established the well-posedness of the closed-loop system, we can turn to its exponential stability, keeping in mind that, remarkably, the same feedback achieves exponential stabilization in an array of functional spaces.
\begin{proof}[Proof of Corollary \ref{coro:expostab}]
Let $r\in (1/2-\alpha, \alpha-1/2)$, and $u_0\in \Hdiv^r$. From the definition of the semigroup $S(t)$ on $\Hdiv^r$, we immediately get
\begin{equation*}
\label{eq:decay-exp}
    \|S(t)u_0\|_{\Hdiv^r}=\|T^{-1}e^{t(A-\lambda I)} Tu_0\|_{\Hdiv^r}\leq e^{-\lambda t}\|T^{-1}\|\|T\| \|u_0\|_{\Hdiv^r}.
\end{equation*}
This ends the proof of Corollary \ref{coro:expostab}.
\end{proof}

} 
{\color{black}


\section{Linearized water waves}
\label{sec:water-waves}

In this 
section,
we illustrate the application of our main results to the rapid stabilization of the linearized water waves. We first recall its derivation following \cite{Alazard-jems, alazardburq,  zbMATH02172228, zbMATH06168816} and proceed to describe the functional setting fitting our abstract framework.

\subsection{Derivation of the linearized water waves equation}
Consider the 2-D capillary-gravity water waves for an  homogeneous, inviscid, incompressible, irrotational fluid over a flat bottom on which an external pressure is applied. The volume of the fluid is described by
\[
\Omega(t)=\{(x,y)\in \T \times \R \, | \, -h \leq y \leq \eta(x,t) \}, 
\]
where $y=-h$ is the bottom of the fluid, $y=\eta(x,t)$ is the deformation from the rest $y=0$ of the free surface and $\T=\R/2\pi \Z $ . The evolution of the velocity field $U$ of the fluid and of the free surface are governed by 2-D free surface Euler equation, 
\begin{equation*}
\begin{cases}
\d_t U + (U. \nabla ) U  = - \nabla p - g e_2, \quad & (x,y) \in \Omega(t),  \\
\textrm{div } U  = 0, \; 
\textrm{rot } U  = 0, \quad & (x,y) \in \Omega(t), \\
U.n  = 0, \quad & (x,y) \in \T \times \{-h\},
\end{cases}   
\end{equation*}
satisfying the  boundary conditions on the free surface  $y= \eta(t, x)$,
\begin{equation*}
\begin{cases}
\d_t \eta = \sqrt{1+|\nabla \eta|^2 } U .n  , \quad & (x,y)\in \T \times \{\eta(t, x)\}, \\
p =p_{atm} + P_{ext} - \sigma \kappa(\eta),  \quad & (x,y)\in \T \times \{\eta(t, x)\},
\end{cases}   
\end{equation*}
where $p$ is the pressure, $g$ the gravitational constant, $n:=\frac{1}{\sqrt{1+|\nabla \eta|^2}}(-\nabla \eta, 1)^t $ the outward normal vector to the surface $\eta$, $e_2=(0,1)^t$ the unit vector, $\sigma>0$ is the surface tension coefficient and 
\[
\kappa(\eta)=\d_x  \left( \dfrac{ \d_x \eta}{\sqrt{1+|\d_x  \eta|^2}} \right) = \dfrac{\d_x^2 \eta }{(1+(\d_x \eta)^2)^{3/2}},
\]
is the mean curvature of the surface. 
The first part is  the Euler equation on $U$ describing incompressible and irrotational fluids with an impermeable bottom respectively. The  second part is the boundary conditions on the free surface:  the kinematic equation on the surface for $\eta$ asserting that particles on the surface remains on the surface along time, and the pressure at the surface, including the surface tensison and the localized external pressure $P_{ext}(t, x,\eta(t, x))$.   
The incompressible and irrotational assumption implies that the velocity field is represented by a velocity potential $\Phi : \R^+  \times \R^2 \rightarrow \R$ such that $U= \nabla_{x, y} \Phi$. The 2-D free surface Euler equation  implies that the velocity potential satisfies 
\begin{equation*}
\begin{cases}
\d_t \Phi + \dfrac{1}{2} | \nabla \Phi |^2 + gy = -(p-p_{atm}), \quad & (x,y) \in \Omega(t),  \\
\Delta \Phi = 0, \quad & (x,y) \in \Omega(t),  \\
\d_n \Phi =0, \quad & (x,y) \in (0,2\pi ) \times \{-h\},  \\
\d_t \eta = \sqrt{1+|\nabla \eta|^2 } \d_n \Phi , \quad & (x,y)\in (0,2\pi) \times \{\eta(t, x)\}.
\end{cases}
\end{equation*}

It was first noted by Zakharov \cite{Zakharov} that   the previous equation on the velocity potential $\Phi$  is an Hamiltonian system, where $\psi:=\Phi|_{y=\eta}$ and $\eta$ are (generalized) canonical variables. Moreover, $\Phi$ is completely determined through the Laplace equation and the knowledge of $\psi:=\Phi|_{y=\eta}$ and $\eta$. This leads to study the Dirichlet-to-Neumann map 
\[
G[\eta,h]: \psi \mapsto \sqrt{1+| \nabla \eta|^2} \d_n \Phi|_{y=\eta}= \d_y \Phi(x,\eta,t)-\d_x \eta(x,t) \d_x \Phi(x,\eta,t).
\]
We refer for instance to \cite{Alazard-jems,  alazardburq, zbMATH02172228, zbMATH06168816} and the references therein for the properties of the Dirichlet-to-Neumann map as well as its application to the well-posedness of the Cauchy problem of the gravity and capillary-gravity water waves.

Using the Dirichlet-to-Neumann map, one may reformulate the capillary-gravity water waves as,
\begin{equation}\label{nonlinearwaterwaves}
\begin{cases}
\d_t \eta - G[\eta,h] \psi = 0, \\
\d_t \psi + g \eta  + \dfrac{1}{2} | \nabla \psi|^2 - \dfrac{(G[\eta,h]\psi + \nabla \eta . \nabla \psi)^2}{2(1+|\nabla \eta|^2)}=\sigma \kappa(\eta)-P_{ext}.
\end{cases}
\end{equation}
The Dirichlet-to-Neumann operator is nonlinear with respect to the surface elevation. We therefore consider the linearization around $(\eta,\psi)=(0,0)$, yielding (fixing $\sigma=1$),
\begin{equation}\label{linearwaterwaves}
\begin{cases}
\d_t \eta - G[0,h] \psi = 0, \\
\d_t \psi + g \eta  - \d^2_x \eta =-P_{ext},
\end{cases}
\end{equation}
where $G[0,h]=|D_x|\textrm{tanh}(h|D_x|)$, defined as a Fourier multiplier on periodic functions. Set $\LL$ the operator
\begin{equation}
\label{op-WW-2}
\mathscr{L}:= -i \left((g- \d_x^2)G[0,h]\right)^{1/2},
\end{equation}
and let $u=\psi+\LL G[0,h]^{-1}\eta$, we end up with,
\[
\d_t u = \LL u + P_{ext}.
\]

To be more precise, we consider the external pressure (the control) to be of the distributed control form $P_{ext}=B_1(x)w_1(t)+ B_2(x)w_2(t)$. Notice that $\mathscr{L}$ has double eigenvalues (see Section \ref{sec:notations}), which means according to \cite{Gagnon-Hayat-Xiang-Zhang} two distributed controls are required to control/stabilize the system instead of one.  Hence, for ease of notations we consider
\begin{equation}
\label{eq:linearizedWW}
    \partial_t u= \mathscr{L}u+ B w(t),
\end{equation}
where $B$ is a rank 2  control operator $B: w\in \C^2 \rightarrow w_1B_1+ w_2 B_2$. \textcolor{black}{Note that since all nonzero eigenvalues are doubled, the system would not be controllable with a rank 1 operator and $2$ is the smallest possible dimension to have controllability.} }
{\color{black}
\subsection{Functional setting}

The spectral decomposition of $\mathscr{L}$ is given by,
\[
\mathscr{L} e^{\pm inx} = -i\left((g-n^2)|n|\tanh( h |n|)\right)^{1/2}e^{\pm inx}, \quad n\in \N^*, 
\]
where, 
\begin{equation}\label{part2:eig}
 \lambda_n := -i\left((g-n^2)|n|\tanh( h |n|)\right)^{1/2}.
\end{equation}
Thus,
\[
\mathscr{L} \cos(nx)=\lambda_n \cos(nx), \quad \mathscr{L} \sin(nx)=\lambda_n \sin(nx),
\]
Since every nonzero eigenvalues has multiplicity two, any given function can be decomposed by odd and even parts under the orthonormal eigenbasis, 
\[
\varphi_n^1:=\dfrac{1}{\sqrt{\pi}}\sin(nx), \quad \varphi_n^2:=\dfrac{1}{\sqrt{\pi}}\cos(nx), \textrm{ for } n\in \N^*,
\]
and 
\[
\varphi_0^2=\dfrac{1}{2\pi}.
\]
Hence, we define 
\begin{equation}\label{eq:defHdiv12}
\Hdiv_1^s:=\overline{\textrm{span}\{\varphi_n^1 \, | \, n\in \N^*\}}^{\|. \|_{\Hdiv_1^s}}, \quad \Hdiv_2^s:=\overline{\textrm{span}\{\varphi_n^2 \, | \, n\in \N^*\}}^{\|. \|_{\Hdiv_2^s}}.
\end{equation}
\begin{remark}
As pointed out in Remark \ref{remark:15}, the decomposition of the spaces $\Hdiv_i^s, i=1,2$ is not unique. One could also include the space $\Hdiv_3^s$, where space $\Hdiv_3^s$ could be generated by any finite combination of $\varphi_n^i, i=1,2$ and $n\in \N^*$, without changing the nature of the result. One could also consider $\varphi_0^2$ in the space $\Hdiv_3^s$, but this would mean that the "mass" would not be preserved. This does not present any mathematical obstruction, but is less physically relevant (see Remark \ref{remark:mass}). 
\end{remark}
\subsection{Application of the main results}
We now show that the water-wave operator $\mathscr{L}$ satisfies Assumption \ref{hyp:alpha} with $\alpha_i=3/2$ for $i=1,2$. First, from the definition of $\lambda_n$ given by \eqref{part2:eig},
\[
cn^3 \leq \left| -i\left((g-n^2)|n|\tanh( h |n|)\right)^{1/2} \right| \leq C n^3, \quad n\in \N^*,
\]
for $C,c>0$ since $\tanh$ is bounded above and below. Now, inspired by the definition of $\lambda_n$ we define, 
\begin{equation*}
    g(x):=  \left((g+ x^2) x \tanh(bx)\right)^{\frac{1}{2}}, \forall x\in (0, +\infty),
\end{equation*}
which is strictly increasing and verifies that for any $x\in [1, +\infty)$, 
\begin{align*}
 g'(x)&= \frac{1}{2}  \left((g+ x^2) x \tanh(bx)\right)^{-\frac{1}{2}} \left((g+ 3x^2) \tanh(bx)+ b(g+ x^2)x (1- \tanh^2 (bx)) \right) \\
 &\geq C x^{-\frac{3}{2}} x^2= C x^{\frac{1}{2}}.
\end{align*}
We also observe that 
\begin{equation*}
    g^2(2x)\geq 2 g^2(x), \; \forall x\in [1, +\infty). 
\end{equation*}
 
 Let given $m, n\in \N^*$ with $m\neq n$. \textcolor{black}{Assume without loss of generality that $m<n$.} If  $n\geq 2m$, then 
 \begin{equation*}
    |\lambda_n- \lambda_m|= g(n)- g(m)\geq g(n)- g(n/2)\geq \frac{2- \sqrt{2}}{2} g(n)\geq C n^{\frac{3}{2}} \geq C (n-m)n^{\frac{1}{2}}.
 \end{equation*}
 If $m< n< 2m$, then  there exists some $y\in [m, n]$ such that 
 \begin{equation*}
     |\lambda_n- \lambda_m|= g(n)- g(m)= (n- m) g'(y)\geq C (n- m) m^{\frac{1}{2}}\geq C(\sqrt{2})^{-1} (n- m)n^{\frac{1}{2}}.
 \end{equation*}
 \textcolor{black}{Hence the water-wave operator does satisfy Assumption \ref{hyp:alpha} with $\alpha_{1}=\alpha_{2} = 3/2$.} Then, it suffices to consider $B$ satisfying either Assumption \ref{asump1} or Assumption \ref{asump2} to deduce the controllability of the linearized water waves equation, thanks to Proposition \ref{prop:moreregularcontfrollabilitylinear}. 
\begin{prop}

Let $T>0$, $s\geq 0$ and let $B_i \in \Hdiv_i^s$ satisfying Assumption \ref{asump2}. Then, for any $(u_0, u_f)\in (\Hdiv_1^s \times \Hdiv_2^s)^2$, there exists a control $w\in L^2((0,T);\textcolor{black}{\mathbb{R}^{2}} )$ such that the unique solution of \eqref{eq:linearizedWW} with initial state $u_0$ satisfies $u(T)= u_f$ in $\Hdiv_1^s \times \Hdiv_2^s$.
\end{prop}
From Assumption \ref{asump1}, Assumption \ref{asump2} as well as our main theorems, we also deduce the rapid stabilization given by Corollary \ref{th1} and Corollary \ref{coro:expostab}.

}

{\color{black}
\begin{remark}
[Conservation of mass\textcolor{black}{-like condition}]\label{remark:mass}
In this paper we have not taken into account \textcolor{black}{the possibility of a conserved quantity, for instance a} ``conservation of mass" condition \textcolor{black}{
which could be relevant for physical equations such as the water waves equation. In this case, denoting $(\varphi_{n}^{2})_{n\in\mathbb{N}^{*}}$ the even eigenfunctions and $(\varphi_{n}^{1})_{n\in\mathbb{N}^{*}}$ the odd eigenfunctions (recall that for the water waves equation the eigenfunctions are sin and cosine fuctions) associated to $\lambda_{n}=|n|^{3/2}$, the conservation of mass condition becomes
}
\begin{equation*}
        \int_{\T} u(x) dx
    = \langle u(t), \varphi_0^2\rangle=0, 
\end{equation*}
\textcolor{black}{In this case the space $\Hdiv^{r}_{2}$ is generated by the even eigenfunctions and}
  by choosing $B$ such that $\langle B_2, \varphi_0^2\rangle= 0$, the backstepping method can \textcolor{black}{still} be applied. Indeed, by following the same steps as in our proof, one can build an isomorphism that maps the even part of the water-waves system, with mass conservation, to the target system
 \[\partial_t v=\mathscr{L} v-\lambda v,\]
 with mass conservation. \textcolor{black}{This system is} of course still exponentially stable. The key here is that {\color{black} the target system under consideration also preserves mass, \textit{i.e.}, its semigroup $e^{t(A-\lambda I)}$ leaves $\operatorname{Span}\{\varphi_n^2, \ n \geq 1 \} $ invariant, allowing us to implement the backstepping method between the operators $\mathscr{L}$ and $\mathscr{L}-\lambda I$, both with domain $\operatorname{Span}\{\varphi_n^2, \ n \geq 1 \} $.} 
 
\end{remark}

\begin{remark}[
Water waves in bounded domains
]
\textcolor{black}{In Corollaries \ref{th1} and \ref{coro:expostab},} we have investigated the linearized water waves system in a periodic domain. In fact we can also study the same system in a bounded domain with Neumann boundary conditions, the controllability of which was obtained by Reid \cite{Reid-linearww}. In this framework, since all eigenvalues are simple, \textcolor{black}{the situation is simpler and} we are able to establish controllability, and rapid stabilization by backstepping, using only one control term.
\end{remark}
}

\section{Conclusion}
We have presented a compactness/duality method to overcome the limitations of the classical Fredholm backstepping method. This allows to prove
the rapid stabilization of the linearized capillary-gravity water waves system \eqref{introclosedloopsystem}.
More precisely, this compactness/duality method allows to construct a Riesz basis for skew-adjoint operators behaving like $i|D_x|^\alpha$ for $\alpha>1$, that is beyond the $\alpha > 3/2$ threshold imposed by the typical quadratically close criterion. We were moreover able to prove that the uniqueness condition $TB=B$ can also be handled without the quadratically close criterion, using fine estimations. The rapid stabilization was proved in the spaces $\Hdiv^r,\; r\in (1/2-\alpha,\alpha-1/2)$. These bounds are sharp in the sense that the operator $A+BK$ cannot generate a strong semigroup for $r=\alpha-1/2$, while $r=1/2-\alpha$ is the limit for the operator equality \eqref{op-eq} to hold. Moreover, the feedback law is shown to be independent of $r$. Finally, we are able to prove the existence of the isomorphism $T$ in regular spaces $\Hdiv^s, s\in \R$ -- as long as the control operator $B$ is regular enough and satisfies the equivalent of the controllability Assumption \ref{asump1} -- a crucial step in proving the local rapid stabilization of the nonlinear system \eqref{nonlinearwaterwaves} using the regularity-consuming nonlinear estimates of \cite{alazardburq,zbMATH02172228}.  \\

\paragraph{{\bf Acknowledgements}}
Amaury Hayat and Shengquan Xiang would like to warmly thank Mathematisches Forschungsinstitut Oberwolfach for their hospitality, their kindness, and the inspiration this place and this working environment provided. The authors would also like to thank the PEPS 2022 program of CNRS. Ludovick Gagnon was partially supported by the French Grant ANR ODISSE (ANR-19-CE48-0004-01). Ludovick Gagnon and Christophe Zhang were partially supported by the French Grant ANR TRECOS (ANR-20-CE40-0009). 
Amaury Hayat was financially supported by Ecole des Ponts Paristech.  Shengquan Xiang was financially supported by School of Mathematical Sciences at Peking University and  the Chair of Partial Differential Equations at EPFL. \\

\appendix
\section{Riesz basis in Hilbert spaces}
\label{sec:Rieszdef}
We recall here 
some definitions about vector families in Hilbert spaces (see for instance  \cite{BeauchardLaurent, christensen2003introduction,  CGM, Gagnon-Hayat-Xiang-Zhang}).
\begin{defn}[Vector family]\label{def-riesz-bas}
Let $X$ be a Hilbert space. A family of vectors $\{\xi_n\}_{n\in \textcolor{black}{\mathcal{I}}}$, \textcolor{black}{where $\mathcal{I}=\Z$, $\N$, or $\N^*$} is said to be
\begin{itemize}
    \item[(1)]  \textbf{Dense} in X, if  $ \overline{span\{\xi_i; i\in \textcolor{black}{\mathcal{I}}\}}= X$. 
     \item[(2)]  \textbf{$\omega$-independent} in X, if  
     \begin{equation*}
       \sum_{k\in \textcolor{black}{\mathcal{I}}} c_k \xi_k=0 \textrm{ in $X$ with } \{c_n\}_{n\in\textcolor{black}{\mathcal{I}}}\in\ell^2(\mathcal{I}) \Longrightarrow   c_n=0,\, \forall n\in \textcolor{black}{\mathcal{I}}.
            \end{equation*}
     \item[(3)]  \textbf{Quadratically close} to a family of vector $\{e_n\}_{n\in \textcolor{black}{\mathcal{I}}}$, if
     \begin{equation*}
         \sum_{k\in \textcolor{black}{\mathcal{I}}} \|\xi_k- e_k\|_X^2<+\infty.
     \end{equation*}
      \item[(4)]  \textbf{Riesz basis} of $X$, if it is the image of an isomorphism (on $X$) of some orthonormal basis.
\end{itemize}
\end{defn}
An equivalent definition of Riesz basis can also be stated as follows
\begin{defn}[Riesz basis]
A family of vectors $\{\xi_n\}_{n\in \textcolor{black}{\mathcal{I}}}$, \textcolor{black}{where $\mathcal{I}=\Z$, $\N$, or $\N^*$} of $X$ is a called a Riesz basis of X, if it is  dense in $X$ and if there exist $C_1, C_2>0$ such that for any $\{a_n\}_{n\in \textcolor{black}{\mathcal{I}}}\in\ell^2(\mathcal{I})$ we have 
       \begin{equation}\label{ineq-riesz}
           C_1\sum_{k\in \textcolor{black}{\mathcal{I}}} |a_k|^2\leq  \|\sum_{k\in \textcolor{black}{\mathcal{I}}} a_k \xi_k\|_{X}^2\leq  C_2\sum_{k\in \textcolor{black}{\mathcal{I}}} |a_k|^2.
       \end{equation}
\end{defn}

 The following lemma  has been heavily used in the literature as a criterion for  Riesz basis. 
\begin{lem}\label{lem-cri-riesz}
Let $\{\xi_n\}_{n\in
\mathcal{I}}$ be quadratically close to an orthonormal basis $\{e_n\}_{n\in\mathcal{I}}$. Suppose that  $\{\xi_n\}_{n\in\mathcal{I}}$ is either dense in $X$ or $\omega$-independent in $X$, then $\{\xi_n\}_{n\in\mathcal{I}}$ is a Riesz basis of $X$.
\end{lem}

Finally in this paper we also make use of the following Lemma

\begin{lem}\label{lem-iso-riesz-equi}
Let $X$, $Y$ be Hilbert spaces. Let $T:X\rightarrow Y$ be an isomorphism. Suppose that $\{\xi_n\}_{n\in \textcolor{black}{\mathcal{I}}}$ is a Riesz basis of $X$, then with $\zeta_n:= T\xi_n$, \textcolor{black}{the family} $\{\zeta_n\}_{n\in \textcolor{black}{\mathcal{I}}}$ is a Riesz basis of $Y$.
\end{lem}

{\color{black}
\section{Controllability properties in the abstract setting}\label{sec:controllability_proof}
In this section, by the decomposition $\Hdiv=\Hdiv_1 \oplus ... \oplus \Hdiv_m$, we consider the restriction of $e^{At}$ on each $\Hdiv_i, i\in \{1,\ldots,m\}$. Hence, for sake of conciseness, we prove the results for an individual $\Hdiv_i$ that we denote $\Hdiv$. We also denote the corresponding control operator $B=(B_1,\ldots,B_m)$ as $B \in \Hdiv^{-\alpha/2}$. 

Let us first recall Haraux's Ingham inequality,
\begin{thm}\label{thm:Haraux}
{\rm (}\cite[Théorème 2]{zbMATH04123624}{\rm)} Let $J\subset \R$ be a bounded interval and $\gamma, \omega >0$. Assume there exists $\{\mu_n\}_{n\in \mathbb{N}^*}$ a real sequence such that,
\begin{enumerate}
    \item $|\mu_{n+1}-\mu_n| \geq \omega, \forall n\in \mathbb{N}^*$;
    \item there exists $N\in \N^*$ such that $|\mu_{n+1}-\mu_n| \geq \gamma$, for all $|n|\geq N$;
    \item $|J|>2\pi / \gamma$.
\end{enumerate}
Then, there exist $c,C>0$ such that for any sequence $\{a_n\}_{n\in \mathbb{N}^*} \in \ell^2(\mathbb{N}^*;\C)$,
\[
c \sum_{n\in \mathbb{N}^*} |a_n|^2 \leq \int_J \left| \sum_{n\in \mathbb{N}^*} a_n e^{i \mu_n t} \right|^2 dt \leq C \sum_{n\in \mathbb{N}^*} |a_n|^2. 
\]
\end{thm}

We then deduce,
\begin{lem}\label{lem:harauxbasis}
Let $T>0$ and assume that $A$ is skew-adjoint and that Assumption \ref{hyp:alpha} holds. Then
{\color{black}
\begin{itemize}
    \item[(i)] the family $\{e^{\lambda_n t}|_{t\in (0, T)}\}_{n\in \N^*}$ is a Riesz basis of 
$$\Theta:= \overline{{\rm Span}\{e^{\lambda_n t}|_{t\in (0, T)}: n\in \mathbb{N}^*}\}^{L^2((0,T);\C)}.$$ 
\item[(ii)] There exists a bi-orthogonal sequence $(g_n(t)|_{t\in (0, T)})_n$ satisfying 
\begin{gather*}
    g_n\in \Theta, \forall n\in \mathbb{N}^*,\\
    \langle e^{\lambda_n \cdot}, g_m\rangle_{L^2(0, T)}= \delta_{n, m}, \forall n, m\in \mathbb{N}^*,
\end{gather*}
  and $(g_n(t)|_{t\in (0, T)})_n$ is also a Riesz basis of $ \Theta$.
\item[(iii)] There exist constants $c, C>0$ such that, for every $f\in L^2(0, T)$, there is 
\begin{equation*}
   c\left( \sum_{n\in \mathbb{N}^*} |\langle f(\cdot), e^{\lambda_n \cdot} \rangle|^2\right)^{\frac{1}{2}}\leq  \left\|\sum_{n\in \mathbb{N}^*} \langle f(\cdot), e^{\lambda_n \cdot} \rangle_{L^2(0, T)} g_n\right\|_{L^2(0, T)}\leq  C \|f\|_{L^2(0, T)}.
\end{equation*}
\end{itemize}

}
\end{lem}
\begin{proof}
{\color{black}We mimic the proof by Beauchard--Laurent in  \cite[Proposition 19]{BeauchardLaurent}.}

(i)
Since $A$ is skew-adjoint, the eigenvalues $\lambda_n$ are purely imaginary. Moreover, from Assumption \ref{hyp:alpha}, we deduce that there exists $\omega>0$ such that $|\lambda_{n+1}-\lambda_{n}|\geq \omega >0, \forall n\in \N^*$ and that there exists, for any $\gamma>0$, $N\in \N^*$ such that $|\lambda_{n+1}-\lambda_{n}|\geq \gamma >0, \forall n\geq N$. Hence, Haraux's Ingham inequality holds for $\{e^{\lambda_n t}|_{t\in (0, T)}\}_{n\in \N^*}$ and, from \eqref{ineq-riesz}, $\{e^{\lambda_n t}|_{t\in (0, T)}\}_{n\in \N^*}$ is a Riesz basis of $\Theta= \overline{{\rm Span}\{e^{\lambda_n t}|_{t\in (0, T)}: n\in \mathbb{N}^*}\}^{L^2((0,T);\C);}$.      

{\color{black}
(ii) From the fact that $\{e^{\lambda_n t}|_{t\in (0, T)}\}_{n\in \N^*}$ is a Riesz basis of $\Theta$, we also deduce the existence of a bi-orthogonal sequence to $\{e^{-\lambda_n s}|_{s\in (0, T)}\}_{n\in \Z^*}$, that we denote $\{g_m\}_{m\in \N^*}$, also a Riesz basis of $L^2((0,T);\C)$. 

(iii) The first inequality comes from the fact that $(g_n)_n$ is a Riesz basis of $\Theta$, it remains to prove the second one. 
Since $\Theta$ is a closed subspace of $L^2(0, T)$, we decompose $L^2(0, T)$ by $\Theta\oplus \Theta^{\perp}$ with the orthogonal projection $P: L^2(0, T)\rightarrow \Theta$. Thus, for every $f\in L^2(0, T)$, there is 
\begin{equation*}
    Pf=  \sum_{n\in \mathbb{N}^*} \langle f(\cdot), e^{\lambda_n \cdot} \rangle_{L^2(0, T)} g_n.
\end{equation*}
Hence we obtain the second inequality from the fact that $\|P f\|_{L^2(0, T)}\leq \|f\|_{L^2(0, T)}$.}
\end{proof}

We now prove the admissibility of the control operator stated in Lemma \ref{lem:admissible}.
\begin{proof}
We prove Lemma \ref{lem:admissible} for a single space $\Hdiv$, as the proof follows readily in the general case. We apply directly Lemma \ref{lem:harauxbasis} with the hypothesis of Lemma \ref{lem:admissible} to deduce that {\color{black}both  $\{e^{\lambda_k t}\}_{k\in \N^*}$ and its bi-orthogonal sequence $\{g_n(t)|_{t\in (0, T)}\}_{n\in \mathbb{N}^*}$ are  Riesz bases of $\overline{{\rm Span}\{e^{\lambda_n t}|_{t\in (0, T)}: n\in \mathbb{N}^*}\}^{L^2((0,T);\C)}$ for any $T>0$. We can therefore write, for any $T>0$, 
\begin{gather*}
L^2((0,T);\C)=\Theta \oplus \Theta^{\perp} \textrm{ and the corresponding orthogonal projections }\\
P: L^2(0, T)\rightarrow \Theta, \\
P^{\perp}: L^2(0, T)\rightarrow \Theta^{\perp}.
\end{gather*}
where 
$$ \Theta:=\overline{{\rm Span}\{e^{\lambda_n (T- t)}|_{t\in (0, T)}: n\in \mathbb{N}^*}\}^{L^2((0,T);\C)}= \overline{{\rm Span}\{g_n(t)|_{t\in (0, T)}: n\in \mathbb{N}^*}\}^{L^2((0,T);\C)}.$$
Now, for 
any $T>0$, let $w \in L^2((0,T);\C)$ such that $w=P w+ P^{\perp}w$. } 
We have
\begin{equation*}
    (P w) (t)= \sum_{k\in \N^*} w_k g_k(t) \textrm{ with } w_k= \langle w(\cdot), e^{\lambda_k \cdot} \rangle_{L^2(0, t)}, \quad \forall t\in (0, T).
\end{equation*}
Then, on the one hand, from Assumption \ref{asump1}, we deduce the admissibility in $\Hdiv^s$ for $s\leq 0$. Indeed, it suffices to show the case $s=0$. We have,
\begin{align*}
\left\|\int_0^T e^{A(T-t)}Bw(t) dt \right\|_{H}^2 &=\left\|\int_0^T e^{A(T-t)}B(w_1(t)+w_2(t)) dt \right\|_{H}^2 \\
&=\left\|\int_0^T \sum_{k\in \N*}e^{\lambda_k(T-t)}b_k \varphi_k (P w(t)+P^{\perp}w(t)) dt \right\|_{H}^2 \\
&= \sum_{k\in \N^*} |b_k|^2 \left|\int_0^T e^{\lambda_k(T-t)} (P w(t)+P^{\perp}w(t)) dt \right|^2 \\
&= \sum_{k\in \N^*} |b_k|^2 \left|\int_0^T e^{\lambda_k(T-t)} P w(t) dt \right|^2 \\
&= \sum_{k\in \N^*} |b_k|^2 |w_k|^2 \\
&\leq C \sum_{k\in \N^*}  |w_k|^2 \\
&= C\|P w(t)\|_{L^2(0,T)}^2 \\
&\leq C\|w(t)\|_{L^2(0,T)}^2.
\end{align*}
{\color{black}
Therefore, for every $T>0$ there exists a constant $C_T>0$ such that, for every $w\in L^2(0, T)$, and for every initial state $u_0\in H$,  there exists a solution $u(t)$ such that 
\begin{equation*}
    \|u(T)\|_{H}\leq \|u_0\|_{H}+ C_T\|w\|_{L^2(0, T)}. 
\end{equation*}
This further implies that, thanks to the fact that the semigroup $e^{ tA}$ preserves the $H$-norm, the solution belongs to $C([0, T]; H)$ and satisfies 
\begin{equation*}
    \|u(t)\|_{H}\leq \|u_0\|_{H}+ C_T\|w\|_{L^2(0, t)}, \forall t\in [0, T].
\end{equation*} 
}

On the other hand, assume the admissibility in $H$. Then, for any $T>0$,
\[
\left\|\int_0^T e^{A(T-t)}Bw(t) dt \right\|_{H} \leq C \|w\|_{L^2(0, T)}, \quad \forall w\in L^2((0,T);\C).
\]
Taking successively $w(t)=g_m(t), m\in \N^*$, we have,
\begin{align*}
\left\|\int_0^T e^{A(T-t)}Bw(t) dt \right\|_{H}^2 &=\left\|\int_0^T \sum_{k\in \N^*}e  ^{\lambda_k(T-t)}b_k \varphi_k w(t) dt \right\|_{H}^2 \\
&=\left\|\sum_{k\in \N^*}b_k\varphi_k \int_0^T e  ^{\lambda_k(T-t)} g_m(t) dt \right\|_{H}^2 \\
&= \sum_{k\in \N^*}  |b_k|^2 \left|\int_0^T e^{\lambda_k(T-t)}  g_m(t) dt \right|^2 \\
&= |b_m|^2 \leq C \|g_m\|_{L^2(0, T)}^2\leq C.
\end{align*}
Therefore, the sequence $b_m$ is uniformly bounded from above. 
\end{proof}

We now turn to the proof of Proposition \ref{prop:contfrollabilitylinear}.
\begin{proof}
We study the controllability of the abstract system 
\begin{equation}\label{appendix:cont}
\begin{cases}
\d_t u(t) = Au(t) + Bw(t), t\in (0,T)\\
u(0)=u_0,
\end{cases}
\end{equation}
with $u_0\in \Hdiv$ and $w\in L^2((0,T);\textcolor{black}{\mathbb{C}} )$.

From the time-reversibility of \eqref{appendix:cont} and the linearity of the equation, it is sufficient to prove that it is always possible to drive the solution from any initial data to $u(T)=0$, the proof of which relies on the moment method. Indeed, let $u_{0}\in \Hdiv$ and use the Duhamel formula, thanks to \eqref{admissible}, to write the solution of \eqref{appendix:cont} as,
\[
u(t)=e^{At }u_0 + \int_0^t e^{A(t-s)} Bv(s) ds. 
\]
Denote the decomposition of $u_0$ and $B$ in the eigenbasis,
\begin{equation*}
    u_0= \sum_{n\in \N^*} u^n_0 \varphi_n \;  \textrm{  and  } B= \sum_{n\in \N^*} b_n \varphi_n.
\end{equation*}
Then, the controllability is equivalent to,
\begin{align*} 
0 & =\sum_{n\in \N^*} e^{\lambda_n T} u_0^n \varphi_n + \int_0^T e^{A(T-s)} \sum_{n\in \N^*} b_n \varphi_n w(s) ds \\
& =\sum_{n\in \N^*} e^{\lambda_n T} u_0^n \varphi_n + \int_0^T \sum_{n\in \N^*} e^{\lambda_n (T-s)} b_n \varphi_n w(s) ds.
\end{align*}
Using Assumption \eqref{condB}, and more precisely, $0<c\leq |b_n|$ to ensure $b_{n}\neq 0$, this is equivalent to, 
\begin{equation}\label{eq:coeffmoment}
-\dfrac{u_0^n}{b_n}= \int_0^T e^{-\lambda_n s}  w(s) ds, \quad \forall n \in \N^\ast.
\end{equation}
Thanks to this controllability assumption, we have
\[
\dfrac{u_0^n}{b_n} \in \ell^2(\N^*;\C).
\]
Therefore, if we are able to prove that there exists $w\in L^2((0,T);\C)$ such that \eqref{eq:coeffmoment} is satisfied, then we deduce the null-controllability. Indeed, Lemma \ref{lem:harauxbasis} implies that $\{e^{-\lambda_n s}\}_{n\in \N^*}$ and its bi-orthogonal family $(g_n(s))_{n \in \N^\ast}$ are Riesz basis of $\overline{{\rm Span}\{e^{\lambda_n t}|_{t\in (0, T)}: n\in \mathbb{N}^*}\}^{L^2((0,T);\C)}$. Therefore, the control can be chosen in such fashion as, 
\begin{equation*}
    w(t):= -\sum_{m\in\N^*} \frac{u_0^m}{b_m} g_m(t) ,
\end{equation*}
satisfying $\|w\|_{L^2(0, T)}\lesssim \|u_0\|_{\Hdiv}$.
Hence the null-controllability of the system  \eqref{appendix:cont}.

\end{proof}

}

{\color{black}
\section{Extension to the case $\beta\neq 0$}
\label{sec:beta}

Suppose that Proposition \ref{prop:th} is proved for the case that $\beta=0$. Then, for any $B= \sum_n b_n \varphi_n$ \textcolor{black}{satisfying} Assumption \ref{asump2},
  we consider  $\tilde B= \sum_n \tilde b_n \varphi_n$ with $\tilde b_n= n^{ \beta} b_n$, which satisfies
  \begin{equation*}
      c\leq |\tilde b_n| \leq C, \quad \forall n\in \mathbb{N}. 
  \end{equation*}
According to Proposition \ref{prop:th}  for the case that $\beta=0$, for any $\lambda>0$, for the controlled system, 
  \begin{equation*}
      \partial_t \tilde u= A\tilde u+ \tilde B y(t),
  \end{equation*}
  we are able to find \textcolor{black}{an} operator $\tilde K\in \mathcal{L}(\Hdiv^{\frac{\alpha}{2}}; \mathbb{C})$ and an isomorphism \textcolor{black}{$\tilde{T}$} from 
  which is an isomorphism 
  \textcolor{black}{from} $\Hdiv^r$ to itself for any $r\in ( 1/2- \alpha, \alpha- 1/2)$ such that, $(\tilde T, \tilde K)$ is an isomorphism-feedback pair to the preceding system and, 
  \begin{equation*}
      \partial_t \tilde u= A\tilde u+ \tilde B \tilde K u,
  \end{equation*}
  is exponentially stable with decay rate $\lambda$.
  Now we introduce a map $M$ from $\Hdiv^{\infty}$ to $\Hdiv^{\infty}$ as, 
  \begin{align*}
      M: \Hdiv^{\infty}&\rightarrow \Hdiv^{\infty}, \\
      n^{-\beta} \varphi_n&\mapsto \varphi_n.
  \end{align*}
  Clearly, $M$ is an isomorphism from $\Hdiv^{\beta+ s}$ to $\Hdiv^{s}$ for every $s\in \mathbb{R}$, and the operator $M$  commutes with the operator $A$. 
  Then, for the controlled system,
\begin{equation*}
   \partial_t   u= A  u+ B y(t),
\end{equation*}
where $B= M^{-1} \tilde B$, we consider the feedback law (namely, we set $y(t)= K u(t)$),
$$K= \tilde K M\in \mathcal{L}(\Hdiv^{\beta+ \frac{\alpha}{2}}; \mathbb{C}),$$
as well as the transformation, 
\begin{equation*}
    T:= M^{-1} \tilde T M,
\end{equation*}
which is an isomorphism from $\Hdiv^{r}$ to itself  for any $r\in (\beta+ 1/2- \alpha, \beta+ \alpha- 1/2)$. It follows from simple calculation that the above-defined $(T,K)$ is an isomorphism-feedback pair. Actually, considering $\tilde u= M u$ we obtain,
\begin{align*}
    \partial_t \tilde u= \partial_t M u= M \partial_t  u= M(Au+ BK u)= A \tilde u+ \tilde B\tilde K \tilde u.
\end{align*}
\textcolor{black}{This ends the proof of this claim.}
 }

\section{Proof of Lemmas \ref{lem1}--\ref{lem:tech2}}\label{sec:appendixc}

\begin{proof}[Proof of Lemma \ref{lem1}] 
By assumption \textcolor{black}{(see Assumption \ref{hyp:alpha} and \eqref{hyp:1})}
\[
|\lambda_n-\lambda_m| \textcolor{black}{\geq} c n^{\alpha-1}|n-m|.
\]
\textcolor{black}{T}here exists a constant $C>0$ such that  $n^{\alpha-1}|n-m| \geq C|n-m|^{\alpha}$ if and only if (assuming $n\neq m$, otherwise the proof is trivial), 
\[
n^{\alpha-1} \geq C|n-m|^{\alpha-1}.
\]
But this follows from the strict monotonicity of the function $f(x)=x^{\alpha-1}$. 
\end{proof}

\begin{proof}[Proof of Lemma \ref{lem:tech1} and Lemma \ref{lem:tech2}]
Define $f(x):= x^s$ for $x\in [1, +\infty)$. There exists $c_0, C_0>0$ such that \begin{equation*}
    c_0 n^{s}\leq f(x)\leq C_0 n^{s}, \; \forall n\in \N^*, \forall x\in [n, n+1].
\end{equation*}
Concerning Lemma  \ref{lem:tech1},  suppose that $s\neq -1$, then 
\begin{equation*}
     \sum\limits_{n=1}^{p} n^{s}\leq c_0^{-1} \int_{1}^{p+1} f(x) dx= \frac{c_0^{-1}}{s+ 1} \left((p+1)^{s+1} -1\right)\leq C (1+ p^{1+ s}),
\end{equation*}
this ends the proof of Lemma \ref{lem:tech1}.

Next we turn to the proof of Lemma \ref{lem:tech2}. Given $s\in \mathbb{R}$, {\color{black} we are able to choose $\varepsilon>0$ such that $s+ \varepsilon\neq -1$.} Because there exists $C>0$ such that,
\begin{equation*}
    \log (n)\leq C n^{\varepsilon}, \; \forall n\in \N^*,
\end{equation*}
we can use Lemma \ref{lem:tech1} as well as the fact that $s+ \varepsilon\neq -1$. This yields,
\begin{equation*}
     \sum\limits_{n=1}^{p} n^{s}\log(n)\leq   \sum\limits_{n=1}^{p} n^{s+ \varepsilon}\leq C(1+p^{1+s+\varepsilon}).
\end{equation*}
\textcolor{black}{Once this estimation holds for any $\varepsilon>0$ such that $s+ \varepsilon\neq -1$, we deduce it for any $\varepsilon>0$.}

\end{proof}

\section{Proof of Property $(i)$ in Lemma \ref{lem:L2infqnbarqn}}\label{sec:app:lemme410}
 
To simplify the notations we assume that $\lambda_n\neq 0$.
 The proof can be easily adapted to the  case where $\lambda_1= 0$, since
the resolvant $\mathcal{A}_\lambda$ defined below is well-defined and invertible  on \textcolor{black}{$H$}.

Recall that $q_n\in {\color{black}\Hdiv}^{{\color{black}\alpha-\frac{1}{2}}-\varepsilon}$ for any $\varepsilon>0$.
By defining
$r_n= (\lambda_n+ \lambda/2)^{-1}$ we obtain, by definition \eqref{def-qn} of the $(q_n)$,
\begin{equation}\label{res:gn}
    (A-\lambda- \lambda_n) q_n=- \sum_{p\in \N^*} \varphi_p=: h, \; \textrm{ in } {\color{black}\Hdiv}^{{\color{black}-\frac{1}{2}-\varepsilon}}, \; \textrm{for } \varepsilon>0,
\end{equation}
which becomes, defining\footnote{It is not guaranteed that $\lambda/2$ is in the resolvent {\color{black}set} of $A$. However, a shift of the operator allows to conduct the same proof, with slightly more complicated notations (see for instance \cite{CGM}). For the sake of conciseness, we will assume that $\lambda/2$ is in the resolvent of $A$.} $\mathcal{A}_{\lambda}:= (A-\lambda/2)^{-1}$,
\begin{equation*}
    \mathcal{A}_{\lambda}q_n= r_n q_n-r_n \mathcal{A}_{\lambda}h, \; \textrm{ in } {\color{black}\Hdiv}^{{\color{black}\alpha -\frac{1}{2}-\varepsilon}}, \; \textrm{for } \varepsilon>0.
\end{equation*}

Now, suppose that the $\{q_n\}_{n\in \N^*}$ are not $\omega$-independent {\color{black} in ${\color{black}H}$}, then there exists a nontrivial sequence $\{c_n\}_{n\in \N^*}\in l^2(\N^{*})$ 
such that 
\begin{equation}\label{eq:448}
    \sum_{n\in \N^*} c_n q_n=0 \textrm{ in } {\color{black}H},
\end{equation}
which is well-defined thanks to Remark \ref{remark:rightbound}.

Next,  by applying $\mathcal{A}_{\lambda}$ to \eqref{eq:448}, we conclude,
\begin{equation*}
    \sum_{n\in \N^*} c_n r_n q_n= \left(\sum_{n\in \N^*} c_n r_n \right)\mathcal{A}_{\lambda} h \quad \textrm{ in } {\color{black}H},
\end{equation*}
where we have used the fact $r_n {\color{black}\in \ell^2}$.

Applying again $\mathcal{A}_{\lambda}$ we get, 
\begin{equation*}
      \sum_{n\in \N^*} c_n r_n^2 q_n= \left(\sum_{n\in \N^*} c_n r_n^2 \right)\mathcal{A}_{\lambda} h+ \left(\sum_{n\in \N^*} c_n r_n\right) \mathcal{A}_{\lambda}^{2}h \quad \textrm{ in }  {\color{black}H}.
\end{equation*}
By induction we easily derive,
\begin{equation}\label{App-C-induction-formula}
      \sum_{n\in \N^*} c_n r_n^m q_n= \sum_{i=1}^m \left(\sum_{n\in \N^*} c_n r_n^{m+1-i}\right) \mathcal{A}_{\lambda}^{i} h= \sum_{i=1}^m C_{m+1-i}  \mathcal{A}_{\lambda}^{i} h, \quad m\in \N^*,
\end{equation}
where, 
\begin{equation}\label{eq:452}
    C_{l}:= \sum_{n\in \N^*} c_n r_n^{l}< +\infty, \quad l\in \N^\ast.
\end{equation}
Let us now distinguish two cases:

\textbf{- First case: the $\{C_m\}$ are not identically zero.} We note,
\[m_0=\inf \{n\in \N^\ast, \quad C_n \neq 0\}.\]
Then, starting with \eqref{App-C-induction-formula} with $m=m_0$, we have by induction, 
\[\mathcal{A}_{\lambda}^m h\in \operatorname{span}\{q_n\}_{n\in \N^*}, \quad m\geq 1.\]
Suppose that  $\operatorname{span}\{q_n\}_{n\in\mathbb{N}^{*}}$ is not dense in ${\color{black}H}$, then there exists a nonzero $d=\sum_n d_n \varphi_n\in {\color{black}H}$ such that, 
\begin{equation}\label{eq:455}
    \langle g, d\rangle_{\textcolor{black}{H}}=0,\quad \forall g\in \textrm{span} \{q_n\}_{n\in\mathbb{N}^{*}},
\end{equation}
which in particular yields,
\begin{equation*}
    \langle \mathcal{A}_{\lambda}^{m}h, d\rangle_{{\color{black}H}}=0, \quad\forall \textcolor{black}{m}\in \N^*.
\end{equation*}
Recalling that $h=-\sum \varphi_n\in {\color{black}\Hdiv}^{-1}$, we get that, 
\begin{equation}\label{dnsum0}
   \sum_n \overline d_n r_n^m=0, \quad \forall \textcolor{black}{m}\in \N^*.
\end{equation}
By defining the complex variable function, 
\begin{equation*}
    G(z):= \sum_{n\in \N^*} \overline d_n r_n e^{r_n z}, \forall z\in \mathbb{C}.
\end{equation*}
By checking that the series expansion of the right-hand side is absolutely convergent, we deduce that this function is holomorphic. From \eqref{dnsum0} we know that $G^{\textcolor{black}{(m)}}(0)= 0, m\in \N$. Thus $G=0$, and further $d_n=0$, which leads to a contradiction. Therefore, we conclude in the first case if the $\{q_n\}$ are not $\omega$-independent in ${\color{black}H}$, then, 
\begin{equation}\label{spangndense}
   \textrm{span} \{q_n\}_{n\in \N^*} \textrm{ is  dense in }{\color{black}H}.
\end{equation}

\textbf{- Second case: the $\{C_m\}$ are identically zero.} Then we define the complex variable function,
\begin{equation*}
    \tilde{G}(z):= \sum_{n\in \N^*} c_n r_n e^{r_n z}.
\end{equation*}
This function is holomorphic. Moreover, as the $(C_m)$ are identically zero, it satisfies,
\[\tilde{G}^{(m)}(0)=0, \quad \forall m\in\mathbb{N},\] thus as previously $\tilde{G}= 0$ and therefore,
\[c_n=0, \quad \forall n\in\mathbb{N}^{*},\]
which is in contradiction with the definition of the $\{c_n\}_{n\in \N^*}$. Hence, in the second case that the $\{q_n\}$ are $\omega$-independent in ${\color{black}H}$. 

Therefore, since we either fall in the first or second case, this proves that the family $\{q_n\}$ is either $\omega$-independent or dense in ${\color{black}H}$, which ends the proof of the property $(i)$  in Lemma \ref{lem:L2infqnbarqn}.

\section{Proof that $ker(T^{*}) = \{0\}$}
\label{isomorph}
We proceed as presented in Section \ref{subsec:step8}.
Let $\rho\in\mathbb{C}$ to be chosen later on and let us look at $A+B K+\lambda \textcolor{black}{Id}+ \rho \textcolor{black}{Id}$. 

  {\bf 1)}
Let us denote $z:= \lambda+ \rho$, we try to investigate the invertibility of $\textcolor{black}{Id}+A^{-1}B K+z A^{-1}$ in the ${\color{black}\Hdiv}^{{\color{black}\alpha/2}}$ space. As $\rho$ can be chosen arbitrarily, $z\in\mathbb{C}$ can be as well. {\color{black}Remark here that if $A$ is not invertible, then, since the spectrum is countable, then there always exists $\delta\neq 0$ sufficiently close to $0$ such that $\tilde A:= A+ \delta$ is invertible. Hence, without loss of generality, we assume that $0$ is in the resolvent set of $A$.} We now consider two cases:
   
 \begin{itemize}
   \item  If $K(A^{-1}B )\neq -1$, then we \textcolor{black}{know} that the bounded operator $\textcolor{black}{Id}+A^{-1}B K$ is invertible. In fact, for any $f\in {\color{black}\Hdiv}^{{\color{black}\alpha/2}}$, we can check that 
   \begin{equation*}
       \varphi:= f-\frac{A^{-1}B  (Kf)}{1+ K(A^{-1}B )}\in {\color{black}\Hdiv}^{{\color{black}\alpha/2}},
   \end{equation*}
   is the unique solution to
   \begin{equation*}
       (\textcolor{black}{Id}+A^{-1}B K) \varphi= f.
   \end{equation*}
   Note that $A^{-1}$ is a compact operator in ${\color{black}\Hdiv}^{{\color{black}\alpha/2}}$ (since $A$ is a differential operator) thus a continuous operator in ${\color{black}\Hdiv}^{{\color{black}\alpha/2}}$  and $\textcolor{black}{Id}+A^{-1}B K$ is invertible, thus thanks to the openness of invertible operators, there exists $\varepsilon>0$ such that for any $|z|<\varepsilon$ 
   \begin{equation}
       (\textcolor{black}{Id}+A^{-1}B K)+ z A^{-1},
   \end{equation}
   is invertible in ${\color{black}\Hdiv}^{-{\color{black}\alpha/2}}$. 
   \item 
   If $K(A^{-1}B )= -1$, then one can check that 0 is an eigenvalue of $Id+A^{-1}B K$ with multiplicity 1 and the eigenspace is generated by $A^{-1}B $.   Indeed, it is clear that $A^{-1}B$ is an eigenvector of $Id+A^{-1}B K$ with eigenvalue 0. On the other hand, suppose that for some $v\in {\color{black}\Hdiv}^{{\color{black}\alpha/2}}$ we have $(Id+A^{-1}B K)v= 0$, then we immediately conclude that $ v= -(A^{-1}B) (Kv)\in \textrm{ span}\{A^{-1}B\}$. 
   
   Therefore, there exist small open neighborhoods $\Omega$ and  $\widetilde{\Omega}$ of $0$ in $\mathbb{C}$ satisfying (see for instance \cite{MR47254})
   \begin{gather}
   \label{eq:eqy}
       (Id+ A^{-1}B  K+ zA^{-1}) y(z)= \lambda(z) y(z), \\
       y(z): z\in \Omega\mapsto y(z)\in {\color{black}\Hdiv}^{{\color{black}\alpha/2}} \textrm{ is holomorphic,} \\
        \lambda(z): z\in \Omega\mapsto \lambda(z)\in \widetilde{\Omega}\subset \mathbb{C} \textrm{ is holomorphic,}\\
        \lambda(0)=0,\quad y_0:= y(0)= A^{-1}B,
   \end{gather}
in such fashion that {\color{black} for any $z\in \Omega$}, $\lambda(z)$ is the unique eigenvalue inside $\widetilde{\Omega}$. Recall that $\lambda(0)=0$, therefore only two cases are possible:
either $\lambda$ is identically $0$ in $\Omega$, or there exists a smaller neighborhood $\omega$ such that for any $z$ in  $\omega\setminus\{0\}$ there is $\lambda(z)\neq 0$. Let us show by contradiction that $\lambda$ is not identically $0$. Assume that it is. We now that there exists a sequence $(y_{k})_{k\in \N^*}$ in ${\color{black}\Hdiv}^{{\color{black}\alpha/2}}$ such that
\begin{equation*}
    y(z) = \sum\limits_{k=0}^{+\infty}y_{k}z^{k},
\end{equation*}
with $y_{0} = A^{-1}B$.
From \eqref{eq:eqy} and  the fact that $\lambda(z)=0$ in $\Omega$, 
\begin{equation*}
       (Id+ A^{-1}B  K+ zA^{-1}) \sum\limits_{k=0}^{+\infty}y_{k}z^{k}= 0, \; \textrm{ in } {\color{black}\Hdiv}^{{\color{black}\alpha/2}},
\end{equation*}
by unicity of the development in entire series we get,
\begin{equation}\label{eq:app:iteration}
         (Id+ y_0 K)y_k+ A^{-1} y_{k-1}=0, \; \textrm{ in } {\color{black}\Hdiv}^{{\color{black}\alpha/2}}, \; \forall k\in\mathbb{N}^{*}.  
\end{equation}
Recall that $K(y_0)= -1$, by applying  $K$ to the previous equation we  conclude that, 
 \begin{equation*}
      K(A^{-1}y_{k-1})=0, \; \forall k\in \N^* \Longrightarrow K(A^{-1}y_k)=0, \; \forall k\geq 0.
  \end{equation*}
By applying $K A^{-1}$ to Equation \eqref{eq:app:iteration} we  arrive at, 
\begin{equation*}
    K(A^{-1} y_k)+ K(A^{-1} y_0)(K y_k)+ K(A^{-2} y_{k-1})=0, \; \forall k\in \N^*,
\end{equation*}
thus,
  \begin{equation*}
       K(A^{-2}y_{k-1})=0, \; \forall k\in \N^* \Longrightarrow K(A^{-2}y_k)=0, \; \forall k\geq 0.
  \end{equation*}
  Then by successively applying $KA^{-(n-1)}$ to the same equation we arrive at, 
   \begin{equation*}
      K(A^{-n}y_k)=0, \; \forall k\geq 0, \forall n\geq 1,
  \end{equation*}
  which in particular yields, 
  \begin{equation*}
      K(A^{-n} y_0)=0,  \forall n\geq 1.
  \end{equation*}
  
  The previous equality implies,
  \begin{equation*}
      \sum_{m\in \N^*} \frac{b_m K_m}{\lambda_m^{l}}=0, \forall l\geq 2.
  \end{equation*}
  Again using the holomorphic function technique in Appendix \ref{sec:app:lemme410}, we conclude that $b_m K_m=0$, which is a contradiction.
   Therefore, we know that there exists $\varepsilon>0$ such that for any $z\in\mathbb{C}$ with $|z|<\varepsilon$, $z\in\Omega$ and $\lambda(z)\neq0$. Since $\lambda(z)$ is the unique eigenvalue inside $\widetilde{\Omega}$,  $\textcolor{black}{Id}+A^{-1}B K+ z A^{-1} $ is invertible.
  \end{itemize}
   In both cases there exists at least a sequence of $\{z_k\}$ converging to 0 such that ${Id}+A^{-1}B K+ \lambda A^{-1}+ (z_k-\lambda) A^{-1}$ is invertible in ${\color{black}\Hdiv}^{{\color{black}\alpha/2}}$. As the spectrum of $A+ \rho \textcolor{black}{Id}$ is discrete, we can find some $\rho:=z_k- \lambda$, such that both 
   \begin{equation*}
       A+ \rho \textcolor{black}{Id} \; \textrm{ and }\;  A+ B K+ \lambda \textcolor{black}{Id}+ \rho \textcolor{black}{Id}= A({Id}+A^{-1}B K+ \lambda A^{-1}+ \rho A^{-1}),
   \end{equation*}
   are invertible operators from ${\color{black}\Hdiv}^{{\color{black}\alpha/2}}$ to ${\color{black}\Hdiv}^{-{\color{black}\alpha/2}}$. This ends the proof of the first point.\\

 {\bf 1)*} Alternatively, this first step can be proved using the following direct method. We show that there exists some effectively computable real number $\rho_0>0$ such that for any $\rho\in [\rho_0, +\infty)$ the operators,   
   \begin{equation*}
       A+ \rho Id \; \textrm{ and }\;  A+ B K+ \lambda Id+ \rho Id: {\color{black}\Hdiv}^{{\color{black}\alpha/2}}\rightarrow {\color{black}\Hdiv}^{-{\color{black}\alpha/2}}
   \end{equation*}
are invertible. Notice that since the spectrum of $A$ belongs to $i\R$, it is straightforward that $A+ \rho Id$ is invertible. It suffices to show that $ A+ B K+ \lambda Id+ \rho Id$ is invertible.
We have the following lemma.
\begin{lem}\label{lemma:directinvertible}
There exists some effectively computable $\rho_0>0$ such that for any $\rho\in [\rho_0, +\infty)$ there is,
\begin{equation*}
   c_{\rho}:=  \sum_{m\in \N^*} \frac{b_m K_m}{\rho+ \lambda+ \lambda_m} \; \textrm{ satisfying } \; |c_{\rho}| \leq \frac{1}{2}.
\end{equation*}
\end{lem}
\begin{proof}[Proof of Lemma \ref{lemma:directinvertible}]
We know from Lemma \ref{lemkncomp} that $|K_m|$ is uniformly bounded, thus there exists some $C_0>0$ such that $|b_m K_m|\leq C_0$ for $\forall m\in \N^*$, hence
\begin{equation*}
   |c_{\rho}|\leq  \sum_{m\in \N^*} \frac{|b_m K_m|}{|\rho+ \lambda+ \lambda_m|}\leq  \sum_{m\in \N^*} \frac{C_0}{|\rho+ \lambda+ \lambda_m|}.
\end{equation*}
Also notice the existence of $N_0\in\N^*$ such that 
\begin{equation*}
    \sum_{m> N_0} \frac{C_0}{|\lambda_m|}\leq \frac{1}{4}.
\end{equation*}
By choosing $\rho_0>0$ large enough there is 
\begin{equation*}
    \sum_{m\in \N^*} \frac{|b_m K_m|}{|\rho+ \lambda+ \lambda_m|}\leq \sum_{m\leq N_0} \frac{C_0}{|\rho_0+ \lambda|}+ \sum_{m> N_0} \frac{C_0}{|\lambda_m|}\leq \frac{1}{2}.
\end{equation*}
\end{proof}
Now we come back and prove that for any $\rho\in [\rho_0, +\infty)$ the operator $ A+ B K+ \lambda Id+ \rho Id$ is invertible. It suffices to show that for any $g\in {\color{black}\Hdiv}^{-{\color{black}\alpha/2}}$ there exists a unique $f\in {\color{black}\Hdiv}^{{\color{black}\alpha/2}}$ such that, 
\begin{gather}
    (A+ B K+ \lambda Id+ \rho Id) f= g, \label{eq:app:inv:fg}\\
    \|f\|_{{\color{black}\Hdiv}^{{\color{black}\alpha/2}}}\lesssim \|g\|_{{\color{black}\Hdiv}^{-{\color{black}\alpha/2}}}. \label{eq:app:inv:fgine}
\end{gather}
Let us denote by, 
\begin{equation*}
    g:= \sum_{n\in \N^*} g_n n^{{\color{black}\alpha/2}} {\color{black} \varphi_n} , \; f:=  \sum_{n\in \N^*} f_n n^{-{\color{black}\alpha/2}} {\color{black} \varphi_n}.
\end{equation*}
By comparing the coefficients in both side of the equation \eqref{eq:app:inv:fg} we get, 
\begin{equation*}
    f_n n^{-{\color{black}\alpha/2}} (\lambda_n+ \lambda+ \rho)+ b_n K(f)= n^{{\color{black}\alpha/2}} g_n,
\end{equation*}
hence $(f_n)_n$ is implicitly solved by, 
\begin{equation*}
    f_n= \frac{ n^{{\color{black}\alpha/2}} g_n- b_n K(f)}{ n^{-{\color{black}\alpha/2}} (\lambda_n+ \lambda+ \rho)}.
\end{equation*}
The previous implicit formula of $f$ yields,
\begin{equation*}
    K(f)= \sum_{m\in \N^*} K_m f_m m^{-{\color{black}\alpha/2}}= \sum_{m\in \N^*} \frac{K_m g_m m^{{\color{black}\alpha/2}}}{\lambda_m+ \lambda+ \rho}- K(f)\left(\sum_{m\in \N^*} \frac{b_m K_m}{\lambda_m+ \lambda+ \rho}\right).
\end{equation*} 
Thanks to Lemma \ref{lemma:directinvertible}, the coefficients $(f_n)_n$ are uniquely determined by 
\begin{equation}
    f_n= \frac{n^{{\color{black}\alpha}} g_n}{\lambda_n+ \lambda+ \rho}- \frac{n^{{\color{black}\alpha/2}} b_n}{(1+ c_{\rho})(\lambda_n+ \lambda+ \rho)}\left(\sum_{m\in \N^*} \frac{K_m g_m m^{{\color{black}\alpha/2}}}{\lambda_m+ \lambda+ \rho}\right).
\end{equation}
Since 
\begin{equation*}
   \left| \sum_{m\in \N^*} \frac{K_m g_m m^{{\color{black}\alpha/2}}}{\lambda_m+ \lambda+ \rho}\right|\lesssim \sum_{m\in \N^*} \frac{m^{{\color{black}\alpha/2}}|g_m|}{|\lambda_m|}\lesssim \|(g_m)_m\|_{l^2},
\end{equation*}
we have 
\begin{equation*}
    \|(f_n)_n\|_{l^2}\lesssim \|(g_n)_n\|_{l^2}+ \|(g_m)_m\|_{l^2} \left\|\left(\frac{n^{{\color{black}\alpha/2}} b_n}{(1+ c_{\rho})(\lambda_n+ \lambda+ \rho)}\right)_n\right\|_{l^2}\lesssim \|(g_n)_n\|_{l^2}.
\end{equation*}
This finishes the proof of the inequality \eqref{eq:app:inv:fgine}.\\

    {\bf 2)} Let us assume that $\ker T^*\neq\{0\}$ and let $\rho$ be defined as in 1). We are going to show that there exists an eigenvector of $A+\rho Id$ in $ker(T^{*})$ and deduce that there exists $n\in \N^*$ such that $n^{{\color{black}\alpha/2}}\varphi_{n}\in ker(T^{*})$. We know from the operator equality \eqref{op-eq} that
    \begin{equation*}
        T(A+B K+\lambda \textcolor{black}{Id}+ \rho \textcolor{black}{Id})= AT+ \rho \textcolor{black}{T},
    \end{equation*}
    holds when the operators are seen as acting on ${\color{black}\Hdiv}^{{\color{black}\alpha/2}}$ to ${\color{black}\Hdiv}^{-{\color{black}\alpha/2}}$. Thus from the invertibility of the two operators (from point 1)),
    \begin{equation}
    \label{eq:appabove}
        (A+ \rho \textcolor{black}{Id})^{-1} T= T(A+B K+\lambda \textcolor{black}{Id}+\rho \textcolor{black}{Id})^{-1},
    \end{equation}
    where the operator are seen as acting on ${\color{black}\Hdiv}^{-{\color{black}\alpha/2}}$ to ${\color{black}\Hdiv}^{{\color{black}\alpha/2}}$.
    Since $ker(T^{*})\neq\{0\}$ we can take $h\neq 0$ such that $h\in \ker T^*$ and  $h\in {\color{black}\Hdiv}^{-{\color{black}\alpha/2}}$. We deduce from \eqref{eq:appabove} that for any $ \varphi\in {\color{black}\Hdiv}^{-{\color{black}\alpha/2}}$,
    \begin{align*}
        0&= \langle (A+\rho \textcolor{black}{Id})^{-1}T \varphi- T(A+B K+\lambda \textcolor{black}{Id}+\rho \textcolor{black}{Id})^{-1}\varphi, h\rangle_{{\color{black}\Hdiv}^{-{\color{black}\alpha/2}}}, \\
        &= \langle  \varphi, T^* (A^{*}+\overline{\rho} \textcolor{black}{Id})^{-1} h\rangle_{{\color{black}\Hdiv}^{-{\color{black}\alpha/2}}}- \langle  \textcolor{black}{(A+B K+\lambda \textcolor{black}{Id}+\rho \textcolor{black}{Id})^{-1}}\varphi, T^*  h\rangle_{{\color{black}\Hdiv}^{-{\color{black}\alpha/2}}}, \\
         &= \langle  \varphi, T^* (A^{*}+\overline{\rho} Id)^{-1} h\rangle_{{\color{black}\Hdiv}^{-{\color{black}\alpha/2}}}, 
    \end{align*}
    where $A^{*}$ is the adjoint of $A$. The above implies that $T^* (A^{*}+\overline{\rho} Id)^{-1} h= 0$ in ${\color{black}\Hdiv}^{-{\color{black}\alpha/2}}$, thus $(A^{*}+\overline{\rho} \textcolor{black}{Id})^{-1} h\in$ $\ker T^*$. Namely,  we have deduced that
    \begin{equation*}
        (A^{*}+ \overline{\rho} {Id})^{-1}: \ker T^*\rightarrow \ker T^*.
    \end{equation*}
Because $\ker T^*$ is of finite dimension (recall that $T$ is Fredholm, hence $T^{*}$ is) and not reduced to $\{0\}$ there exists an eigenfunction $h\in ker(T^{*})$ of $(A^{*}+ \overline{\rho} {Id})^{-1}$, associated to an eigenvalue $\mu\neq 0$ (since the operator $(A^{*}+ \overline{\rho} {Id})^{-1}$ is invertible). Thus \begin{equation}
\label{eq:eqh}
       (A^{*}+\overline{\rho} \textcolor{black}{Id})^{-1}h= \mu h,
   \end{equation} 
which in particular implies that $h\in {\color{black}\Hdiv}^{{\color{black}\alpha/2}}$. We immediately deduce that $h$ is an eigenfunction of $A^{*}$ in ${\color{black}\Hdiv}^{-{\color{black}\alpha/2}}$. Moreover, we know from \eqref{eq:eqh} that 
    \[A^{*} h= \frac{1-\overline{\rho}\mu}{\mu} h.\]
Now we would like to conclude that there exists $n\in \N^*$ and $C\neq 0$ such that $h= C n^{{\color{black}\alpha/2}}\varphi_{n}$. Note that since $(n^{{\color{black}\alpha/2}}\varphi_{n})_{n\in\mathbb{N}^{*}}$ is an orthonormal basis of eigenvectors of $A$, it is also a basis of eigenvector of $A^{*}$ (associated to eigenvalues $(\overline{\lambda_{n}})_{n\in\mathbb{N}^{*}}$).
To obtain such a conclusion, we notice that the eigenspaces of $A^{*}$ (in ${\color{black}\Hdiv}^{-{\color{black}\alpha/2}}$) have dimension 1,
in particular the dimension of the eigenspace associated to $(1-\overline{\rho}\mu)/\mu$ is one, and therefore there exist some $n\in \N^*$ and $C\neq0$ such that $h= C n^{{\color{black}\alpha/2}}\varphi_{n}$.
    \\
    
     {\bf 3)} From point 2), we know that if $ker(T^{*})\neq \{0\}$ there exists $n\in \N^*$ such that $n^{{\color{black}\alpha/2}}\varphi_{n}\in ker(T^{*})$, thus,
     \begin{equation*}
         \langle T\varphi, n^{{\color{black}\alpha/2}}\varphi_n \rangle_{H^{-{\color{black}\alpha/2}}}=0, \; \forall \varphi\in {\color{black}\Hdiv}^{-{\color{black}\alpha/2}}. 
     \end{equation*}
     We know  that this is impossible: as $TB=B$ holds in ${\color{black}\Hdiv}^{-{\color{black}\alpha/2}}$, we can take $\varphi:= B$ to achieve,
      \begin{equation*}
         0=\langle TB, n^{{\color{black}\alpha/2}}\varphi_n \rangle_{{\color{black}\Hdiv}^{-{\color{black}\alpha/2}}}= \langle B, n^{{\color{black}\alpha/2}}\varphi_{n} \rangle_{{\color{black}\Hdiv}^{-{\color{black}\alpha/2}}}= \frac{b_n}{n^{{\color{black}\alpha/2}}}, 
     \end{equation*}
     which is in contradiction with the fact that $|b_n|$ is uniformly bounded \textcolor{black}{from} below. Hence $ker(T^{*})=\{0\}$.

\section{Construction of the spectral decomposition}
\label{sec:decomp}

{\color{black}
Denote $(\lambda_n)_{n\in \N^{*}}$ the eigenvalues of $A$ and assume that the multiplicity $m_n\in \N^{*}$ of the eigenvalue $\lambda_n$ is finite. Assume $m=\max_{n\in \N^{*}} m_n < \infty$. We construct a spectral decomposition of $A$ such that
\begin{equation}
\label{eq:propdecomp}
\begin{split}
H &= H_{1}\oplus ...\oplus H_{m},\\
&\begin{cases}
H_{i} &= \text{Span}\{(\varphi_{n}^{i})_{n\in\mathbb{N}^{*}}\}\;\text{for}\;i\in\{1,...m_{l}\},\\
H_{i} &= \text{Span}\{(\varphi_{n}^{i})_{n\in\{1,...,N_{i}\}}\}\;\text{for}\;i\in\{m_{l}+1,...,m\},
\end{cases}
\end{split}
\end{equation}
where $m_{l}\in\{1,...,m\}$, $\varphi_{n}^{i}$ is an eigenvector of $A$ and $Span$ is taken with a closure in $H$. Moreover, denoting $\lambda_{n}^{i}$ the eigenvalue associated to $\varphi_{n}^{i}$, we have for any $n\neq m$
\begin{equation}
    \lambda_{n}^{i}\neq\lambda_{m}^{i}.
\end{equation}
In other words the spaces $H_{i}$ are eigenspaces of $A$ with only simple eigenvalues. To construct this decomposition we proceed as follows:

we denote by $m_{l}\in\{1,...,m\}$ the highest multiplicity such that there is an infinite number of eigenvalues with multiplicity $m_{l}$. Then we construct $m$ set of eigenvector--eigenvalue pairs as follows 
\begin{itemize}
    \item $(\varphi_{n}^{1},\lambda_{n}^{1})_{n\in\mathbb{N}^{*}}$ are eigenvalues--eigenvectors pairs selected such that $(\lambda_{n}^{1})_{n\in\mathbb{N}^{*}}$ enumerate all the distincts eigenvalues in increasing order and $(\varphi_{n}^{1})_{n\in\N^{*}}$ is a family where each $\varphi_{n}^{1}$ is an eigenvector associated to $\lambda_{n}^{1}$.
    \item $(\varphi_{n}^{2},\lambda_{n}^{{\color{black}2}})_{n\in\mathbb{N}^{*}}$ are eigenvalues--eigenvectors pairs selected such that $(\lambda_{n}^{2})_{n\in\mathbb{N}^{*}}$ enumerate all the distincts eigenvalues with multiplicity at least $2$ in increasing order and $(\varphi_{n}^{2})_{n\in\N^{*}}$ is a family where each $\varphi_{n}^{2}$ is an eigenvector associated to $\lambda_{n}^{2}$ and is orthogonal to all eigenvectors of $(\varphi_{n}^{1})_{n\in\N^{*}}$.
    \item for any $i\in\{1,...,m_{l}\}$, $(\varphi_{n}^{i},\lambda_{n}^{i})_{n\in\mathbb{N}^{*}}$ are eigenvalues--eigenvectors pairs selected such that $(\lambda_{n}^{i})_{n\in\mathbb{N}^{*}}$ enumerate all the distincts eigenvalues with multiplicity at least $i$ in increasing order and $(\varphi_{n}^{i})_{n\in\mathbb{N}^{*}}$ is a family where each $\varphi_{n}^{i}$ is an eigenvector associated to $\lambda_{n}^{i}$ and is orthogonal to all eigenvectors of  $(\varphi_{n}^{k})_{(n,k)\in\mathbb{N}^{*}\times\{1,...,i-1\}}$.
    \item for any $i\in\{m_{l}+1,...,m\}$, we denote by $N_{i}$ the (finite) number of eigenvalues  with multiplicity $i$ and $(\varphi_{n}^{i},\lambda_{n}^{i})_{n\in\{1,..,N_{i}\}}$ are eigenvalues--eigenvectors pairs selected such that $(\lambda_{n}^{i})_{n\in\{1,..,N_{i}\}}$ enumerate all the distincts eigenvalues with multiplicity at least $i$ in increasing order and $(\varphi_{n}^{i})_{n\in\{1,..,N_{i}\}}$ is a family where each $\varphi_{n}^{i}$ is an eigenvector associated to $\lambda_{n}^{i}$ and is orthogonal to all eigenvectors of  $(\varphi_{n}^{k})_{(n,k)\in\mathbb{N}\times\{1,...i-1\}}$.
\end{itemize}
Consequently we have with this decomposition
\begin{equation*}
    \cup\{(\varphi, \lambda): \textrm{eigenpairs of } A\}= \left( \cup_{i=m_{l}+1}^{m} \cup_{n=1}^{N_{i}} (\varphi^i_n, \lambda^i_n)\right) \cup \left(\cup_{i=1}^{m_{l}} \cup_{n\in\mathbb{N}^{*}} (\varphi^i_n, \lambda^i_n)\right), 
\end{equation*}
From this point it suffices to define, for any $i\in\{1,...,m\}$
\begin{equation}
    H_{i} = \text{Span}((\varphi_{n}^{i})_{n\in\mathbb{N}}),
\end{equation}
(where Span is again taken with a closure in $H$)
and for any $i\in\{m_{l}+1,...m\}$ (if any)
\begin{equation}
    H_{i} = \text{Span}((\varphi_{n}^{i})_{\{1,...,N_{i}\}}),
\end{equation}
to have \eqref{eq:propdecomp}.
}

\bibliographystyle{plain}    
\bibliography{waterwave}

\end{document}